\documentclass[11pt, reqno]{amsart}
\usepackage{amssymb}
\usepackage{amsmath}
\usepackage{amsthm}
\usepackage{mathabx}
\usepackage{braket}
\usepackage{pdfpages}
\usepackage{graphicx}
\usepackage{caption}
\usepackage{subcaption}
\usepackage{mathrsfs} 
\usepackage{tikz-cd} 
\usepackage{bm}
\usepackage{enumitem}
\usepackage{mathtools}
\usepackage{pgfplots}
\usepackage{import}
\usepackage{xifthen}
\usepackage{transparent}
\usepackage{mathtools}
\usepackage[new]{old-arrows}
\usepackage{stackengine}

\usepackage{umoline} 

\usepackage{tikz}
\usetikzlibrary{automata, arrows.meta, positioning}

\usepackage{todonotes}

\usepackage[utf8]{inputenc}
\usepackage[english]{babel}

\usepackage{hyperref}
\usepackage{cleveref}

\usepackage{lmodern,babel,adjustbox,booktabs,multirow}

\crefformat{footnote}{#2\footnotemark[#1]#3}

\numberwithin{equation}{section}
\theoremstyle{plain}
\newtheorem{theorem}{Theorem}[section]

\theoremstyle{theorem}
\newtheorem{prop}[theorem]{Proposition}

\newtheorem{lem}[theorem]{Lemma}
\newtheorem{cor}[theorem]{Corollary}

\newtheorem{question}[theorem]{Question}
\newtheorem*{question*}{Question}

\theoremstyle{plain}

\theoremstyle{definition}
\newtheorem{defn}[theorem]{Definition}
\newtheorem{rmk}[theorem]{Remark}
\newtheorem{example}[theorem]{Example}

\newcommand{\R}{\mathbb{R}}
\newcommand{\C}{\mathbb{C}}

\newcommand{\Q}{\mathbb{Q}}
\newcommand{\Z}{\mathbb{Z}}

\newcommand{\N}{\mathbb{N}}
\newcommand{\D}{\mathbb{D}}

\newcommand{\bp}{\mathcal{F}}
\newcommand{\rl}{\mathcal{R}}

\newcommand{\hCbb}{\widehat{\mathbb{C}}}

\newcommand{\Fcal}{\mathcal{F}}

\newcommand{\Pcal}{\mathcal{P}}
\newcommand{\Ncal}{\mathcal{N}}

\newcommand{\p}{p}

\newcommand{\fc}{\text{fc}}
\newcommand{\zc}{\text{zc}}

\newcommand{\pcf}{post-critically finite }
\newcommand{\Pcf}{Post-critically finite }
\newcommand{\shc}{relative hyperbolic component }
\newcommand{\shcs}{relative hyperbolic components }
\newcommand{\Mol}{\mathscr{M}}

\DeclareMathOperator{\mult}{mult}

\DeclareMathOperator{\BP}{\mathcal{B}}

\DeclareMathOperator{\Julia}{\mathcal{J}}
\DeclareMathOperator{\interior}{int}

\DeclareMathOperator{\Isom}{Isom}

\DeclareMathOperator{\Int}{Int}

\DeclareMathOperator{\diam}{diam}

\DeclareMathOperator{\Comp}{Comp}
\DeclareMathOperator{\End}{End}

\DeclareMathOperator{\val}{val}
\DeclareMathOperator{\ord}{ord}

\DeclareMathOperator{\Hdim}{H.dim}
\DeclareMathOperator{\Vertex}{\mathrm{Vert}}
\DeclareMathOperator{\Edge}{\mathrm{Edge}}
\DeclareMathOperator{\Path}{\mathrm{Path}}
\DeclareMathOperator{\depth}{\mathrm{depth}}

\DeclareMathOperator{\Crit}{\mathrm{Crit}}
\DeclareMathOperator{\CBrank}{\mathrm{rank_{CB}}}

\numberwithin{figure}{section}

\title[Polynomials with core entropy zero]{Polynomials with core entropy zero}

\begin{document}
\begin{author}[Y.~Luo]{Yusheng Luo}
\address{Department of Mathematics, Cornell University, Ithaca, NY 14853, USA}
\email{yusheng.s.luo@gmail.com}
\end{author}
\thanks{The first-named author is partially supported by NSF Grant DMS-2349929}

\begin{author}[I.~Park]{Insung Park}
\address{Institute for Mathematical Sciences, Stony Brook University, Stony Brook, NY 11794-3660, USA}
\email{insung.park@stonybrook.edu}
\end{author}

\begin{abstract}
This paper studies polynomials with core entropy zero.
We give several characterizations of polynomials with core entropy zero. In particular, we show that a degree $d$ \pcf polynomial $f$ has core entropy zero if and only if $f$ is in the degree $d$ main molecule $\Mol_d$.
The characterizations define several comparable quantities which measure the complexities of polynomials with core entropy zero and allow us to have a better understanding of the structure of the main molecule in higher degrees.
\end{abstract}

\maketitle

\setcounter{tocdepth}{1}
\tableofcontents

\section{Introduction}
Classically, topological entropy measures the complexity of a dynamical system.
For a degree $d$ polynomial $f: \C \longrightarrow \C$, the usual topological entropy of $f$ on $\C$ is always $\log d$, which is not very useful, and we need a better notion of entropy for polynomials.

If the polynomial $f$ is {\em post-critically finite}, i.e., if every critical point of $f$ has a finite orbit, then it has a natural forward invariant tree, called its Hubbard tree $\mathcal{T}_f$.
William Thurston defined the {\em core entropy} of $f$ as the topological entropy of the restriction of $f$ on its Hubbard tree $\mathcal{T}_f$,
$$
h(f):= h_{top}(f|_{\mathcal{T}_f}).
$$
The core entropy has been studied extensively in the literature.
It is known that $h(f)$ is a continuous function (see \cite{Tio16, DS20} for the quadratic case and \cite{GT21} for the general case), settling a conjecture of Thurston.

In this paper, we study polynomials with core entropy zero, and give the first result that relates the core entropy with the topology of the parameter space for polynomials of degree $\geq 3$ (see \S \ref{sec:ccz}).
We introduce several finer measures of the complexity that distinguish polynomials with core entropy zero, and we show that these measures are all comparable (see \S \ref{sec:fmc}). 
These finer measures allow us to have a better understanding of the structure of the main molecules, which are more mysterious for degree $\geq 3$ (see \S \ref{subsec:mm}).

\subsection{Characterizations of polynomials with core entropy zero}\label{sec:ccz}
Let $\mathcal{P}_d$ be the space of {\em monic} and {\em centered} polynomials of degree $d$ (i.e., polynomials of the form $z^d+a_{d-2}z^{d-2}+\cdots+a_0$).
Let $f \in \mathcal{P}_d$ be a \pcf polynomial.
The {\em \shc} $\mathcal{H}_f \subseteq \mathcal{P}_d$ consists of polynomials that are quasiconformally conjugate to $f$ near the Julia set (see \S \ref{sec:mm}).
Note that if $f$ is hyperbolic, then $\mathcal{H}_f$ is the {\em hyperbolic component} of $\mathcal{P}_d$ containing $f$.
Two \shcs $\mathcal{H}_f$ and $\mathcal{H}_g$ are said to be {\em adjacent} if $\partial{\mathcal{H}_f} \cap \partial{\mathcal{H}_g} \neq \emptyset$.

Let $\mathcal{H}_d$ be the {\em main hyperbolic component}, i.e., the \shc that contains $f(z) = z^d$.
The degree $d$ {\em main molecule} is
$$
\Mol_d := \overline{\bigcup_{\mathcal{H}\in \mathfrak{S}} \mathcal{H}},
$$
where $\mathfrak{S}$ consists of all \shcs $\mathcal{H}$ that are obtained from $\mathcal{H}_d$ through a finite adjacent sequence of relative hyperbolic components.

For a \pcf polynomial $f$, the {\em Hubbard tree} is the regulated hull of the critical and post-critical points (see \cite{DH85, Poi10}).
It is a finite invariant tree that gives a combinatorial description for the dynamics of $f$.
Conversely, one can define the combinatorial notion of a marked abstract Hubbard tree (see \cite{Poi10} or Definition \ref{defn:aht1}).
By \cite[Theorem 1.1]{Poi10}, every marked abstract Hubbard tree gives a \pcf polynomial in $\mathcal{P}_d$.

By \cite{McM88}, every relative hyperbolic component contains a unique \pcf polynomial. 
Thus, we use these marked abstract Hubbard trees to represent relative hyperbolic components in $\mathcal{P}_d$. 
A {\em simplicial tuning} is a combinatorial operation on Hubbard trees that would allow us to characterize adjacent relative hyperbolic components (see \S \ref{subsec:sthf} for details).


Our first result provides various characterizations of polynomials with core entropy zero.

\begin{theorem}\label{thm:A}
Let $f\in \mathcal{P}_d$ be a \pcf polynomial.
Then the following are equivalent.
\begin{enumerate}
\item $h(f) = 0$.
\item $f$ is in the main molecule $\Mol_d$.
\item $f$ is obtained from $z^d$ by a finite sequence of simplicial tunings.
\item $\Julia_f \cap \mathcal{T}_f$ is a countable set where $\mathcal{T}_f$ is the Hubbard tree of $f$.
\end{enumerate}
\end{theorem}

\begin{rmk}\label{rmk:mt}
	Here are some remarks on Theorem \ref{thm:A}.
\begin{itemize}
	\item For quadratic polynomials, the equivalences between (2), (3) and (4) are well known.
	It is known that the core entropy $h(f)$ is related by the simple formula
	$$
	\Hdim \mathscr{B}(f) = \frac{h(f)}{\log d},
	$$
	to the Hausdorff dimension of the set $\mathscr{B}(f) \subseteq \mathbb{S}^1$ of biaccessible angles (see \cite{Tio15, BS17}).
	It is proved in \cite[Proposition 2.11]{BS17} that the main molecule $\Mol_2$ is precisely the locus of parameters with $\Hdim \mathscr{B}(f) = 0$.
	Therefore, the equivalence between (1) and (2) for quadratic polynomials is also immediately obtained from these known results.
	
	\item For real quadratic polynomials, the equivalence between (1) and (2) is also proved in \cite{Dou95}, where it is shown that the core entropy is positive exactly beyond the Feigenbaum point.
	This argument also works, together with the monotonicity on {\em veins} proved in \cite{Tio15}, to prove the equivalence of (1) and (2) in general.
	Our approach is different from the above proofs, which allows us to generalize to higher degrees.
	
	\item
	It is suggested in \cite{GT21} that the core entropy may be a useful tool to define and investigate the hierarchical structure of the connectedness locus.
	Theorem \ref{thm:A} gives a precise connection between the core entropy and the parameter space for higher degrees. 
	
	
	\item There might be many other equivalent formulations of $\mathcal{J}_f \cap \mathcal{T}_f$ being countable. See Appendix \ref{sec:biset} for another characterization using bisets, suggested by L.\@ Bartholdi and V.\@ Nekrashevych.
\end{itemize}

\end{rmk}

\subsection{Finer measures of complexity}\label{sec:fmc}
It is natural to ask if there is a measure of complexity finer than core entropy that can distinguish polynomials with core entropy zero.
The perspectives of (1)-(4) in Theorem \ref{thm:A} give rise to different measures.

Let $f \in \mathcal{P}_d$ be a \pcf polynomial with core entropy zero.
\begin{enumerate}
\item (Growth rate) Let $A_f$ be the Markov matrix associated to the dynamics of $f$ on its Hubbard tree $\mathcal{T}_f$. Then $\| A_f^n\| \asymp  n^\alpha$ for some $\alpha \in \Z_{\ge0}$ (see \S \ref{sec:hd}). We define the {\em growth rate complexity} $\mathscr{C}_{gr}(f):= 1+\alpha$.
\item (Bifurcation) We define the {\em bifurcation complexity} $\mathscr{C}_{b}(f)$ as the smallest number of \shcs one needs to {\em bifurcate} to arrive at $\mathcal{H}_f$ from $\mathcal{H}_d$ (see \S \ref{sec:mm}).
\item (Combinatorial) We define the {\em combinatorial complexity} $\mathscr{C}_{c}(f)$ as the smallest number of {\em simplicial tunings} needed to obtain $f$ from $z^d$ (see \S \ref{sec:ss}). 
\item (Topological)
We define the {\em topological complexity} $\mathscr{C}_{t} (f)$ as the {\em Cantor-Bendixson rank} of $\Julia_f \cap \mathcal{T}_f$ (see \S \ref{subsec:cbd}).
\end{enumerate}
We emphasize that our bifurcation complexity is defined in terms of {\em dynamically meaningful transitions} from one \shc to another (see Definition \ref{defn:bif} and the subsection on open questions below for more discussions). 

The following theorem relates the four measures.
\begin{theorem}\label{thm:B}
Let $f \in \mathcal{P}_d$ be a \pcf with core entropy zero. Then we have
$$
\mathscr{C}_{b}(f)\leq \mathscr{C}_{gr}(f) = \mathscr{C}_c(f) = \mathscr{C}_{t}(f) \leq (d-1)\mathscr{C}_{b}(f).
$$
Moreover, the bounds are sharp. More precisely, for any degree $d \geq 2$, there exist \pcf polynomials $f, g$ with core entropy zero so that $\mathscr{C}_{b}(f) =  \mathscr{C}_{gr}(f)$ and $\mathscr{C}_{gr}(g) = (d-1)\mathscr{C}_{b}(g)$.
\end{theorem}

See Figure \ref{fig:EV} for an example of $\mathscr{C}_{gr}(g) = (d-1)\mathscr{C}_{b}(g)$.

\subsection{Main molecules of higher degrees vs degree 2}\label{subsec:mm}
We remark that there are many subtleties for main molecules in higher degrees compared to the quadratic case.

\begin{itemize}
\item First, unlike the quadratic case, in higher degrees there exists an infinite sequence of distinct hyperbolic components in $\mathscr{M}_d$ that accumulates to a \pcf polynomial $f$ with core entropy zero. For example, the \pcf polynomial $f$ in \S \ref{subsec:eg} is in the limit of an infinite sequence of adjacent relative hyperbolic component (Property (2)).
\item Moreover, in the quadratic case, there is a unique sequence of adjacent relative hyperbolic components with no backtracking that connects two \pcf polynomials in the main molecule $\mathscr{M}_2$. 
This is not the case in higher degrees. For example, the relative hyperbolic component of the \pcf polynomial $f$ in \S \ref{subsec:eg}  is adjacent to the main hyperbolic component (Property (3)), while also satisfying Property (2) described in the previous paragraph. As another example, the polynomials in Figure \ref{fig:EV} represent the centers of a sequence of adjacent hyperbolic components, each of which is also adjacent to the main hyperbolic component.
\end{itemize}

Theorem \ref{thm:B} gives a bound on the shortest path connecting a \pcf polynomial in $\mathscr{M}_d$ with $z^d$.
In particular, one can always find a path to access a \pcf polynomial in $\mathscr{M}_d$ through a finite sequence of adjacent \shcs from $\mathcal{H}_d$.
As a corollary, we have:

\begin{cor}
There are no \pcf polynomials in 
$$
\Mol_d - \bigcup_{\mathcal{H}\in \mathfrak{S}} \mathcal{H}.
$$
\end{cor}

We remark that the above corollary would be false for degree $\geq 3$ if we replace relative hyperbolic components by hyperbolic components (see the first bullet point above and Remark \ref{rmk:ibc}).


\subsection{Techniques and strategies in the proofs}
Let us briefly discuss the techniques and strategies in our proofs of Theorems \ref{thm:A} and \ref{thm:B}.

First, the proof of $(1) \Longleftrightarrow (4)$ in Theorem \ref{thm:A} relies on a standard study of {\em simple cycles} in the directed graph for the Markov map $f$ on its Hubbard tree (see \S \ref{sec:hd}).
We remark that the implication (4) $\implies$ (1) also directly follows from the equation $\Hdim \mathcal{B}(f) = \frac{h(f)}{\log d}$ \cite{Tio15, BS17}.

To prove the implication (1) $\implies$ (3), we define an operation, called {\em simplicial quotient}, which produces a \pcf polynomial $g$ from any \pcf polynomial $f$.
The simplicial quotient is an inverse process of {\em simplicial tuning}.
More precisely, the map $f$ is obtained from its simplicial quotient $g$ via some simplicial tuning.
The key step is to prove that if the core entropy of $f$ is zero, then there exists a finite sequence $f_0 = f, f_1, ..., f_k(z) = z^d$ so that $f_{i+1}$ is the simplicial quotient of $f_i$ (see Theorem \ref{prop:fb}).

We prove (3) $\implies$ (2) by using the theory of {\em quasi \pcf degeneration} developed in \cite{L21a, L21b}.
The theory allows us to relate the combinatorial operation of simplicial tuning with a bifurcation on a \shc.

Finally, to prove (2) $\implies$ (1), we analyze how the external rays landing at the same point on the Julia set change as we perturb polynomials. This concludes the proof of Theorem \ref{thm:A}.

To get the second main result (Theorem \ref{thm:B}), we relate the complexity measures $\mathscr{C}_{gr}(f), \mathscr{C}_{c}(f)$ and $\mathscr{C}_{t}(f)$ with the {\em depths} of simple cycles in the directed graph for the Markov map $f$ on its Hubbard tree (see \S \ref{sec:hd}).
This establishes the equality $\mathscr{C}_{gr}(f)= \mathscr{C}_{c}(f)=\mathscr{C}_{t}(f)$.

We then use the theory developed in \cite{L21a} to show that if $\mathcal{H}_f$ bifurcates to $\mathcal{H}_g$, then $g$ can be constructed from $f$ by at most $d-1$ simplicial tunings.
This allows us to prove the bound $\mathscr{C}_{b}(f) \leq \mathscr{C}_{c}(f) \leq (d-1)\mathscr{C}_{b}(f)$.
We also construct explicit examples to show the sharpness of our bound, which completes the proof of Theorem \ref{thm:B}.

\subsection{Notes and discussions}
The study of topological entropy for real quadratic polynomials goes back to the work of Milnor-Thurston \cite{MT88}, where the continuity and the monotonicity of entropy were proved.

The continuity of core entropies is proved in the quadratic case in \cite{Tio16, DS20}, and for higher degrees in \cite{GT21}.
For quadratic polynomials, the core entropy is increasing from the center of the Mandelbrot set to the tips \cite{Tio15, Zen20}.
Core entropies and related concepts for other maps are studied in \cite{Tsu00, Li07, Jun14, BvS15, Gao19, FP20, Fil21, BDLW21, LTW21}.

The topology in higher degree of the connected locus of $\mathcal{P}_d$ are studied in \cite{Lav89, Milnor92, EY99}.
The topologies of the main hyperbolic components are studied in \cite{BOPT14, PT09, L21a}.

The idea of using the Cantor-Bendixson rank was suggested by K.\@ Pilgrim to the second-named author in the study of crochet maps \cite{Par21}. It was then further discussed with L.\@ Bartholdi, D.\@ Dudko, M.\@ Hlushchanka, V.\@ Nekrashevych, D.\@ Thurston. Polynomials with core entropy zero can be considered as crochet maps relative to the external Fatou component.

%
%

\subsection{Open questions}
Here are some more questions about the structure of the main molecule $\mathscr{M}_d$ which are not answered in this article.

\begin{question}
	Is $\mathscr{M}_d$ simply connected?
\end{question}

\begin{question}\label{quest:DenseHyp}
	Are the hyperbolic maps in $\mathscr{M}_d$ dense?
\end{question}

See Figure \ref{fig:BHM} and Figure \ref{fig:BL} for examples of non-hyperbolic \pcf polynomials $f$ with core entropy zero that are limits of hyperbolic \pcf polynomials in $\mathscr{M}_d$.

To understand the combinatorial structure of $\mathscr{M}_d$, we consider a graph $\mathscr{G}_d$ whose vertices are \shcs in $\Mol_d$ with two vertices joined by an edge if they are adjacent. 

There exists another natural complexity $\mathscr{C}_{a}(f)$ defined as the distance between $\mathcal{H}_f$ and $\mathcal{H}_d$ with respect to the graph metric of $\mathscr{G}_d$.
Since our bifurcation complexity is defined in terms of dynamically meaningful transitions, for any \pcf polynomial in $\Mol_d$, we have
$$
\mathscr{C}_a(f) \leq \mathscr{C}_b(f).
$$

We do not know whether the shortest path connecting $\mathcal{H}_f$ and the main hyperbolic component $\mathcal{H}_d$ in $\mathscr{G}_d$ can be realized as a sequence of bifurcations.
It would be interesting to know:

\begin{question}\label{quest:AdjBif}
Let $f \in \Mol_d$ be a \pcf polynomial.
Do we have $\mathscr{C}_a(f) = \mathscr{C}_b(f)$?
\end{question}

In degree $2$, the graph $\mathscr{G}_2$ is a tree.
In higher degrees $d>2$, it is easy to see that the graph $\mathscr{G}_d$ contains many closed loops.
One would like to know:

\begin{question}
What is the combinatorial structure of the graph $\mathscr{G}_d$?
\end{question}
The adjacent vertices of $\mathcal{H}_d$ in $\mathscr{G}_d$ are studied in \cite{L21a}, providing some partial answers to this question.

\subsection*{Acknowledgement}
The authors thank Laurent Bartholdi, Dzmitri Dudko, Mikhail Hlushchanka, Kevin Pilgrim, Dylan Thurston for helpful discussions.

This material is based upon work supported by the National Science Foundation under Grant No.\@ DMS-1928930 while the authors participated in a program hosted by the Mathematical Sciences Research Institute in Berkeley, California, during the Spring 2022 semester. The first-named author is partially supported by NSF Grant DMS-2349929. The second-named author was supported by the Simons Foundation Institute Grant Award ID 507536 while he was in residence at the Institute for Computational and Experimental Research in Mathematics in Providence, RI.

\section{Directed graphs and Cantor-Bendixson decomposition}\label{sec:hd}
Let $f$ be a \pcf polynomial.
Recall that the Hubbard tree is the regulated hull of the critical and post-critical points, and let $\mathcal{T}_f$ be the Hubbard tree of $f$.
Let $\mathcal{V} \subseteq \mathcal{T}_f$ be a finite vertex set so that 
\begin{itemize}
\item $\mathcal{V}$ is forward invariant, i.e., $f(\mathcal{V})\subset \mathcal{V}$, and
\item $\mathcal{V}$ contains all {\em branch points}, which are points with degree$>2$, and all critical points.
\end{itemize}

We remark that there might be many different sets $\mathcal{V}$'s satisfying these properties, though there is a unique minimal vertex set. 
An edge of $\mathcal{T}_f$ is the closure of a component $\mathcal{T}_f - \mathcal{V}$.
Let $\mathcal{E}$ be the collection of edges of $\mathcal{T}_f$.
Let $E \in \mathcal{E}$ be an edge of $\mathcal{T}_f$.
Then $f$ is injective on $E$ and $f(E)$ is a union of edges of $\mathcal{T}_f$. We can thus associate a directed graph $\mathcal{G}_f$ as follows.
\begin{itemize}
\item Vertices of $\mathcal{G}_f$ are in bijective correspondence with edges of $\mathcal{T}_f$.
\item There is a directed edge from $a$ to $b$ if the corresponding edges $E_a, E_b$ of the Hubbard tree $\mathcal{T}_f$ satisfy $E_b \subseteq f(E_a)$.
\end{itemize}
We remark that $f(E_a)$ cannot cover $E_b$ more than once because every critical point is a vertex.

A {\em simple cycle} of $\mathcal{G}_f$ is a simple closed directed path. Two simple cycles are {\em distinct} if they are not same up to cycle reordering. Two distinct simple cycles are said to be 
\begin{itemize}
\item {\em disjoint} if they have no vertices in common, and 
\item {\em intersecting} otherwise.
\end{itemize}
The directed graph $\mathcal{G}_f$ is said to have {\em no intersecting cycles} if any distinct simple cycles are disjoint.

In this section, we will show:
\begin{theorem}\label{thm:nic}
Let $f$ be a \pcf polynomial. Let $\mathcal{G}_f$ be a directed graph associated to $f$.
Then $f$ has core entropy zero if and only if $\mathcal{G}_f$ has no intersecting cycles.
\end{theorem}

\begin{theorem}\label{thm:cs}
Let $f$ be a \pcf polynomial. 
Then $f$ has core entropy zero if and only if $\Julia_f \cap \mathcal{T}_f$ is a countable set.
\end{theorem}

As an immediate corollary, we have:
\begin{cor}\label{cor:cs}
Let $f$ be a \pcf polynomial whose Julia set is a dendrite.
Then $f$ has positive core entropy.
\end{cor}

To give a more quantitative version of Theorem \ref{thm:cs}, we investigate the Cantor-Bendixson decomposition of the intersection $\Julia_f \cap \mathcal{T}_f$ for arbitrary \pcf polynomials $f$.
This gives the notion of topological complexity $\mathscr{C}_{t}(f)$ for core entropy zero polynomials $f$.

We prove the following theorem, which immediately implies Theorem \ref{thm:cs} (see \S \ref{subsec:dg} and \S \ref{subsec:cbd} for terminologies used in the statement). 

\begin{theorem}[Cantor-Bendixson decomposition of $\Julia_f\cap \mathcal{T}_f$]\label{thm:CBDecompJuliaIntersectGraph}
	Let $f$ be a \pcf polynomial and $\mathcal{T}_f$ the Hubbard tree. The Cantor-Bendixson decomposition of the intersection $\Julia_f \cap \mathcal{T}_f$ is given by
	\begin{equation}\label{eqn:CBdecompositionof Julia Hubbard}
		\{x\in \Julia_f \cap \mathcal{T}_f~|~\depth(x)=\infty\} \sqcup \{x\in \Julia_f \cap \mathcal{T}_f~|~\depth(x)\le d\}
	\end{equation}
	where $d=\max \{ \depth(e)~|~ e\in \Edge(\mathcal{T}_f),~\depth(e)<\infty \} $. 
	
	In particular, $f$ has core entropy zero if and only if $\Julia_f \cap \mathcal{T}_f$ is a countable set if and only if $\depth(x)< \infty$ for every $x\in \Julia_f \cap \mathcal{T}_f$. 
	
	Moreover, if $f$ has core entropy zero, then we have
	\begin{equation}\label{eqn:TopComplexity and depth}
		\mathscr{C}_{t}(f):= \CBrank(\Julia_f \cap \mathcal{T}_f)=1+\max_{e\in \Edge(\mathcal{T}_f)} \depth(e).
	\end{equation}
\end{theorem}

Let $f \in \mathcal{P}_d$ be a \pcf polynomial with core entropy zero. Suppose $f(z) \neq z^d$. Let $A_f$ be the Markov matrix associated to the dynamics of $f$ on $\mathcal{T}_f$, which equals the adjacency matrix of the corresponding directed graph $\mathcal{G}_f$.
The {\em growth rate complexity} is defined by $\mathscr{C}_{gr}(f):= 1+ \alpha$, where $\alpha \in \Z$ satisfies $\| A_f^n\| \asymp  n^\alpha$. If $f(z) = z^d$, we define $\mathscr{C}_{gr}(f) = 0$.

We will show in Proposition \ref{prop:qnic} that $\alpha = \max_{e\in \Edge(\mathcal{T}_f)} \depth(e)$. 
Thus, by Theorem \ref{thm:CBDecompJuliaIntersectGraph}, we have:

\begin{cor}\label{cor:gr=t}
Let $f$ be a \pcf polynomial with core entropy zero.
Then $\mathscr{C}_{gr}(f) = \mathscr{C}_{t}(f)$.
\end{cor}

\subsection{Directed graph}\label{subsec:dg}
In this subsection, we shall record some basic facts about directed graphs. We then prove Theorem \ref{thm:nic}.

Let $\mathcal{G}$ be a finite directed graph.
The graph $\mathcal{G}$ is determined by its edge and vertex sets together with the map
$$
\Edge(\mathcal{G}) \longrightarrow \Vertex(\mathcal{G}) \times \Vertex(\mathcal{G})
$$
which sends an edge $e$ running from $a$ to $b$ to the ordered pair $[e] = (a,b)$.
We allow $e = (a,a)$, and we may have $e_1 = e_2$ even if $e_1 \neq e_2$. That is, we allow self-loops and multi-edges.

\subsection*{Paths, simple cycles and growth rates}
A {\em path} of {\em length n} is a sequence of edges $p = (e_1,..., e_n)$, with $e_i = (a_i, b_i)$ and $a_{i+1} = b_i$.
It is {\em closed} if $a_1 = b_n$.
A {\em simple cycle} is a closed path that never visits the same vertex twice.

For $n\ge0$, possibly $n=\infty$, we denote by $\Path_\mathcal{G}(v,n)$ the set of paths in $\mathcal{G}$ which start from $v$ and have length $n$. 
We denote by $\Path_\mathcal{G}(n)$ the set of all paths of length $n$ and by $\Path_\mathcal{G}^0(n)$ the set of all closed paths of length $n$.
We use $|\cdot |$ to denote the cardinality of sets.

\subsection*{Adjacency matrix and spectral radius}
Let $\mathcal{G}$ be a finite directed graph.
The {\em adjacency matrix} is a $|\Vertex(\mathcal{G})| \times |\Vertex(\mathcal{G})|$ matrix $A$, with entries
$$
A_{ab} = \mathrm{the~number~of~edges~from}~a~\mathrm{to}~b.
$$
Note that $A \geq 0$, i.e., every entry is nonnegative.
We define a norm by
$$
\| A \| = \sum_{a,b} | A_{ab} |.
$$
The spectral radius $\rho(A)$, defined as the maximum modulus of the complex eigenvalues of $A$, satisfies
$$
\rho(A) = \lim_{n\to \infty} \| A^n \|^{1/n}.
$$
Note that $\| A^n \| = |\Path_\mathcal{G}(n)|$, and it follows from \cite[Lemma 3.1]{McM14} that for a finite directed graph $\mathcal{G}$,
$$
\lim_{n\to \infty}|\Path_\mathcal{G}(n)|^{1/n} = \lim_{n\to \infty}|\Path_\mathcal{G}^0(n)|^{1/n}.
$$
Therefore, we have:
\begin{lem}\label{lem:srp}
Let $\mathcal{G}$ be a finite directed graph.
Then the spectral radius $\rho(A)$ of the adjacency matrix $A$ satisfies
$$
\rho(A) =  \lim_{n\to \infty}|\Path_\mathcal{G}(n)|^{1/n} = \lim_{n\to \infty}|\Path_\mathcal{G}^0(n)|^{1/n}.
$$
\end{lem}

\subsection*{Strongly connected component}
Given two vertices $v, w \in \Vertex(\mathcal{G})$,
we write $v \ge w$ if there is a path from $v$ to $w$.
This defines a preorder on $\Vertex(\mathcal{G})$.
We write $v\simeq w$ if $v\ge w$ and $v \le w$.
Note that $v\simeq w$ if and only if $v, w$ lie on a closed directed path.

We say the finite directed graph $\mathcal{G}$ is {\em strongly connected} if $v \simeq w$ for any $v, w \in  \Vertex(\mathcal{G})$. A subgraph $\mathcal{C} \subseteq \mathcal{G}$ is a {\it strongly connected component} if it is a maximal strongly connected subgraph.
If a strongly connected component $\mathcal{C}$ is a simple cycle, then we call $\mathcal{C}$ a {\it cyclic} component. Otherwise, we call $\mathcal{C}$ a {\it non-cyclic} component.

\subsection*{Graphs having no intersecting cycles}
\begin{defn}
Let $\mathcal{G}$ be a finite directed graph. 
It is said to {\it have no intersecting cycles} if any distinct simple cycles are disjoint.
\end{defn}	

The following lemma is straightforward.

\begin{lem}\label{lem:nicscc}
Let $\mathcal{G}$ be a finite directed graph. 
Then $\mathcal{G}$ has no intersecting cycles if and only if  every strongly connected component is cyclic.
\end{lem}

\subsection*{Computing growth rates}

Let $\{a_n\}_{n\ge1}$ be a sequence of positive real numbers. We say that the sequence $\{a_n\}_{n\ge1}$ grows {\it exponentially fast} with $n$ if $\liminf_{n\to \infty} \sqrt[n]{a_n}>1$ and {\it polynomially fast} with $n$ if $a_n=O(n^d)$ for some $d>0$. In particular, if there exists $C>1$ and $d\ge0$ such that $n^d/C< a_n < Cn^d$ for any sufficiently large $n>0$, we say that $\{a_n\}_{n\ge1}$ is {\it asymptotic} to $n^d$ and write $a_n \asymp n^d$.

\begin{prop}\label{prop:qnic}
	Let $\mathcal{G}$ be a finite directed graph and $v,w\in \Vertex(\mathcal{G})$.
	\begin{enumerate}
		\item If $v \ge w$, then there exists $k>0$, which can be chosen as the length of a path from $v$ to $w$, such that $|\Path_\mathcal{G}(v,n+k)| \ge |\Path_\mathcal{G}(w,n)|$ for every $n\ge0$.
		\item $|\Path_\mathcal{G}(v,n)|$ grows either exponentially or polynomially fast with $n$.
		\item $|\Path_\mathcal{G}(v,n)|$ grows exponentially fast with $n$ if and only if there is a path from $v$ to a strongly connected component that is non-cyclic.
		\item Suppose $|\Path_\mathcal{G}(v,n)| \asymp n^d$. Then $d+1$ is equal to the maximal number of disjoint cycles which can be contained in a directed path starting from $v$. Here we use the convention that $|\Path_\mathcal{G}(v,n)| \asymp n^{-1}$ if $|\Path_\mathcal{G}(v,n)|$ is eventually zero as $n$ tends to $\infty$. 
	\end{enumerate}
\end{prop}
\begin{proof}
	(1) is immediate. See \cite[Theorem 3.6]{Par20} for (2)-(4).
\end{proof}

We remark that for a vertex $v$ of $G$, $|\Path_\mathcal{G}(v,n)| \asymp n^{-1}$ if and only if every directed path starting from $v$ eventually arrives at a vertex with no outgoing edges.
Thus, if every vertex has an outgoing edge, then the spectral radius $\rho(A) \geq 1$.
As a corollary, we have:
\begin{cor}\label{cor:nic}
Let $\mathcal{G}$ be a finite directed graph.
Let $\rho(A)$ be the spectral radius of the adjacency matrix $A$.
Suppose that every vertex of $\mathcal{G}$ has at least one outgoing edge so that $\rho(A)\ge 1$.
Then $\rho(A) = 1$ if and only if $\mathcal{G}$ has no intersecting cycles.
\end{cor}
\begin{proof}
Suppose that $\rho(A) > 1$. Then by Lemma \ref{lem:srp}, $|\Path_\mathcal{G}(n)|$ grows exponentially fast.
Thus there exists a vertex $v \in \Vertex(\mathcal{G})$ so that $|\Path_\mathcal{G}(v,n)|$ grows exponentially fast.
By the statement (3) in Proposition \ref{prop:qnic}, there exists a non-cyclic strongly connected component.
By Lemma \ref{lem:nicscc}, $\mathcal{G}$ has intersecting cycles.

Conversely, suppose that $\rho(A) = 1$. Then by Lemma \ref{lem:srp}, $|\Path_\mathcal{G}(n)|$ grows sub-exponentially fast. 
Thus for any vertex $v \in \Vertex(\mathcal{G})$, $|\Path_\mathcal{G}(v,n)|$ grows  sub-exponentially fast.
By the statement (3) in Proposition \ref{prop:qnic}, all strongly connected components are cyclic.
By Lemma \ref{lem:nicscc}, $\mathcal{G}$ has no intersecting cycles.
\end{proof}

\subsection*{Level structures}
Proposition \ref{prop:qnic} allows us to give a natural level structure on vertices of a finite directed graph $\mathcal{G}$.

\begin{defn}[Depth]\label{defn:DepthDiGraph}
Let $\mathcal{G}$ be a finite directed graph. 
We define a depth function $\depth:\Vertex(\mathcal{G}) \to \Z_{\ge -1}\cup\{\infty\}$ by
\begin{itemize}
	\item $\depth(v)=d$ if $|\Path_\mathcal{G}(v,n)| \asymp n^d$ and
	\item $\depth(v)=\infty$ if $|\Path_\mathcal{G}(v,n)|$ grows exponentially fast.
\end{itemize}
\end{defn}
By definition $v \ge w$ implies $\depth(v)\ge \depth(w)$. Hence the depth function is constant on each strongly connected component.

Let $\mathcal{C}_1, \mathcal{C}_2$ be two distinct simple cycles.
We write $\mathcal{C}_1 \geq \mathcal{C}_2$ (resp. $v \geq \mathcal{C}$) if there exists a path from $\mathcal{C}_1$ to $\mathcal{C}_2$ (resp. from $v$ to $\mathcal{C}$).
By Proposition \ref{prop:qnic}-(4), we have the following lemma, which is useful to compute the depth of a vertex.

\begin{lem}\label{lem:mnsc}
Let $\mathcal{G}$ be a finite directed graph. 
Let $v \in \mathcal{G}$ be a vertex. If $\depth(v)<\infty$, then $\depth(v)$ equals the maximal number $k$ so that there exist disjoint simple cycles $\mathcal{C}_0,..., \mathcal{C}_k$ with
$$
v \geq \mathcal{C}_0 \geq \mathcal{C}_1 \geq ... \geq \mathcal{C}_k.
$$
\end{lem}
\begin{rmk}
In Lemma \ref{lem:mnsc}, we use the convention that if there is no simple cycle $\mathcal{C}$ with $v \geq \mathcal{C}$, then the maximal number is $-1$. 
\end{rmk}

\subsection*{Quotient graph}
We can construct a quotient directed graph $\mathcal{G}'$ by collapsing every strongly connected component to a point.
More precisely, we define two vertices $v,w \in \Vertex(\mathcal{G})$ to be equivalent $v \simeq w$ if and only if they are in the same strongly connected component. 
Then a new vertex set $\Vertex(\mathcal{G}')$ is defined as the quotient $\Vertex(\mathcal{G})/\simeq$.
Two distinct classes $[v], [w] \in \Vertex(\mathcal{G}')$ are connected by a directed edge if there exists an edge $e=(v', w')$ and $v'\in[v],w'\in[w]$.
By definition there are no directed edges connecting $[v]$ with itself, and there is at most one direct edge connecting $[v]$ to $[w]$.
This induces a quotient map
$$
	\Phi: \mathcal{G} \longrightarrow \mathcal{G}',
$$
which sends each edge to a vertex or an edge.
We say a vertex $v \in \Vertex(\mathcal{G}')$ is {\em regular} if $\Phi^{-1}(v)$ is a single point; and {\em singular} otherwise.

\begin{lem}\label{lem:nc}
For a finite directed graph $\mathcal{G}$, the quotient directed graph $\mathcal{G}'$ has no cycles.
\end{lem}
\begin{proof}
	Suppose there exists a cycle $p' = (e_1,..., e_n)$ in $\mathcal{G}'$. Then there is a cycle $p$ in $\mathcal{G}$ so that $\Phi(p)=p'$. Since every cycle belongs to a strongly connected component, $p'=\Phi(p)$ has to be a vertex. Therefore $\mathcal{G}'$ has no cycles.
\end{proof}

\subsection*{Ending component of infinite paths}
We denote by $\Comp(\mathcal{G})$ the set of strongly connected components. Define a map
\[
	\End:\Path_{\mathcal{G}}(\infty) \to \Comp(\mathcal{G})
\]
in such a way that for every infinite path $p$, $\End(p)$ is the unique strongly connected component where $p$ is eventually supported. We call $\End(p)$ the {\it ending component of $p$}. If $\End(p)$ is cyclic, then $p$ is {\em eventually periodic}. The following lemma is straightforward.

\begin{lem}\label{lem:dip}
Let $\mathcal{G}$ be a finite directed graph with no intersecting cycles. Then every infinite path in $\mathcal{G}$ is eventually periodic.
\end{lem}

\subsection*{Application to the core entropy}
Let $f$ be a \pcf polynomial.
Recall that the core entropy $h(f)$ is the topological entropy of the restriction of f on its Hubbard tree $\mathcal{T}_f$:
$$
	h(f):= h_{top}(f|_{\mathcal{T}_f}).
$$

Let $\mathcal{G}_f$ be the directed graph associated to $f$, with respect to some invariant vertex set $\mathcal{V}$.
Let $A_f$ be the adjacency matrix for $\mathcal{G}_f$.
Then, by \cite{MisSzl90}, the core entropy satisfies
\[
	h(f)=\left\{ \begin{array}{cl}		
		\log \rho(A_f) & \mathrm{if}~\rho(A_f)\ge 1 \\
		0 & \mathrm{if}~\rho(A_f)=0
	\end{array}\right.
\]

\begin{proof}[Proof of Theorem \ref{thm:nic}]
Note that every vertex of $\mathcal{G}_f$ has at least one outgoing edge.
Thus, by Corollary \ref{cor:nic}, we have $h(f) = 0$ if and only if $\mathcal{G}_f$ has no intersecting cycles.
\end{proof}

\subsection{Cantor-Bendixson decomposition}\label{subsec:cbd}
In this subsection, we briefly record some results in the Cantor-Bendixson theory, and refer the readers to \cite[\S 6]{Kec95} for more details.
We then use them to prove Theorem \ref{thm:CBDecompJuliaIntersectGraph}.

\begin{defn}[Cantor-Bendixson rank]
	Let $X$ be a topological space. The {\em Cantor-Bendixson derivative} $X'$ of $X$ is the complement of the isolated points, i.e., the set of accumulation points. Let $X^{0}:=X$, $X^{(1)}:=X'$, and $X^{(\lambda+1)}:=(X^{(\lambda)})'$ for any ordinal $\lambda$. The {\em Cantor-Bendixson rank} of $X$, denoted by $\CBrank(X)$, is the smallest ordinal $\lambda$ with $X^{(\lambda+1)}=X^{(\lambda)}$.
\end{defn}

\begin{theorem}[Cantor-Bendixson theorem {\cite[(6.4) and (6.11)]{Kec95}}]
	Let $X$ be a Polish space, i.e., a separable completely metrizable space. Then there exist unique disjoint subsets $C,P \subset X$ with $X=C \cup P$ such that $P$ is perfect and $C$ is countable. More precisely, $P=X^{(\lambda)}$ where $\lambda=\CBrank(X)$.
\end{theorem}

Hence Cantor-Bendixson ranks measure the complexity of the countable components in the decomposition.

\begin{cor}
	Let $X$ be a Polish space. If $X$ is countable, then $X^{(\lambda)}=\emptyset$ for $\lambda=\CBrank(X)$.
\end{cor}

\begin{example}
	Here are some elementary examples.
	\begin{itemize}
		\item $\CBrank(\emptyset)=0$ and $\CBrank(\mathrm{a~discrete~set})=1$.
		\item Consider $X:=\{0\} \cup \{1/n\}_{n\ge1}$. Then $X^{(1)}=\{0\}$ and $X^{(k)}=\emptyset$ for $k\ge2$. Hence $\CBrank(X)=2$.
	\end{itemize}
\end{example}

\subsection*{Semi-conjugacy $\pi:\Path_{\mathcal{G}_f}(\infty) \to \Julia_f \cap \mathcal{T}_f$}

Let $f$ be a \pcf polynomial, $\mathcal{T}_f$ be its Hubbard tree, and $\mathcal{G}_f$ be the corresponding directed graph. We equip $\Edge(\mathcal{T}_f)$ with the discrete topology, $\Edge(\mathcal{T}_f)^{\Z_{\ge0}}$ with the product topology, and $\Path_{\mathcal{G}_f}(\infty)$ with the subspace topology as a subset of $\Edge(\mathcal{T}_f)^{\Z_{\ge0}}$. Let $\sigma$ be the one-side shift on $\Path_{\mathcal{G}_f}(\infty)$, i.e., $\sigma((e_1,e_2,\dots))=(e_2,e_3,\dots)$.

\begin{prop}\label{prop:PathtoJulia}
	Let $f$ be a \pcf polynomial and $\mathcal{T}_f$ be the Hubbard tree. Let $\mathcal{G}_f$ be the associated directed graph. There is a continuous semi-conjugacy 
	\[
	\pi:(\Path_{\mathcal{G}_f}(\infty),\sigma) \to (\Julia_f \cap \mathcal{T}_f,f)
	\] 
	such that $\pi((e_0,e_1,\dots))=v$ if and only if $f^n(v)\in e_n$ for any $n\ge0$.

	For any $v\in \Julia_f \cap \mathcal{T}_f$, the fiber $\pi^{-1}(v)$ is not a singleton if and only if 
	\begin{itemize}
	\item $v\in \bigcup_{n\ge0} f^{-n}(\Vertex(\mathcal{T}_f))$, and 
	\item $\val_{\mathcal{T}_f}(f^k(v))>1$ where $k$ is the least integer with $k\ge0$ with $f^k(v)\in \Vertex(\mathcal{T}_f)$ and $\val$ denotes the valence.
	\end{itemize}
	Moreover, $|\pi^{-1}(v)|=\val_{\mathcal{T}_f}(f^k(v))$.
\end{prop}

\begin{proof}
	We first define a continuous map $\pi$ with $\pi \circ \sigma=f \circ \pi$.
	
	Consider $\mathcal{T}_f$ as a simplicial complex whose $0$-skeleton is $\Vertex(\mathcal{T}_f)$. For $n\ge0$, we define $\mathcal{T}_f(n)$ to be the simplicial complex whose underlying space is homeomorphic to that of $\mathcal{T}_f$ and the $0$-skeleton is $f^{-n}(\Vertex(\mathcal{T}_f)) \subseteq \mathcal{T}_f$. For $n>m$, $\mathcal{T}_f(n)$ is a subdivision of $\mathcal{T}_f(m)$. We call every edge of $\mathcal{T}_f(n)$ a {\em level-$n$ edge}. 
	
	Using induction, we can show that for any path $p_n=(e_0,e_1,\dots,e_{n-1})$ of $\mathcal{G}_f$ of length $n$, there is a unique level-$n$ edge $e(p_n)$ that is injectively mapped into $e_i$ by $f^{i}$ for all $i=0, ..., n-1$. 
	
	
	Let $p_\infty=(e_0,e_1,\dots)\in \Path_{\mathcal{G}_f}(\infty)$ and let $p_n=(e_0,e_1,\dots,e_{n-1})$ be its initial subpath of length $n$. 
	With respect to a conformal metric on $(\hCbb,P_f)$, $f$ is uniformly expanding on a neighborhood of $\Julia_f$ \cite[\S 19]{MilnorBook}. Then we have
	\[
		\diam(\Julia_f \cap e(p_n)) \to 0.
	\]
	Hence the set
	\[
		\pi(p_\infty):=\Julia_f \cap \bigcap_{n\ge0} e(p_n)
	\]
	is a singleton, and the map $\pi:\Path_{\mathcal{G}_f}(\infty)\to \Julia_f \cap \mathcal{T}_f$ is continuous. The equation $f \circ \pi = \pi \circ \sigma$ is immediate from the definition.
	\medskip
	
	Next, we describe $\pi^{-1}(v)$ for $v\in \Julia_f \cap \mathcal{T}_f$ whose non-emptyness implies the surjectivity of $\pi$.
	
	If $f^n(v)\notin \Vertex(\mathcal{T}_f)$ for every $n\ge0$, then there is a unique $e_n\in \Edge(\mathcal{T}_f)$ for each $n\ge0$ so that $f^n(v) \in \interior(e_n)$. Then $\pi^{-1}(v)=\{(e_0,e_1,\dots)\}$.
	
	If $f^n(v) \in \Vertex(\mathcal{T}_f)$ for some $n\ge0$, there may be many edges having $f^n(v)$ as their endpoints. To investigate the ambiguity we consider $\Edge_w(\mathcal{T}_f)$ the set of edges incident to $w\in \Vertex(\mathcal{T}_f)$. 
	We consider any pair $(w,e)\in \Vertex(\mathcal{T}_f) \times \Edge(\mathcal{T}_f)$ with $e\in \Edge_w(\mathcal{T}_f)$ as a tangent direction at $w$. 
	Then $f$ induces a natural map on tangent directions
	\[
		Df: \{w\} \times \Edge_w(\mathcal{T}_f) \to \{f(w)\} \times \Edge_{f(w)}(\mathcal{T}_f).
	\]
	
	Let us first consider the case that $v \in \Vertex(\mathcal{T}_f)$. For any edge $e$ in $\Edge_v(\mathcal{T}_f)$, we define $e_0:=e$ and $e_i$ as the unique edge satisfying 
	$$
		Df^i((v,e_0))=(f^i(v),e_i).
	$$ 
	Then we have $\pi((e_0,e_1,\dots))=v$. This gives rise to an injection $\Edge_v(\mathcal{T}_f) \hookrightarrow \pi^{-1}(v)$. The surjectivity is trivial.

	Suppose that $v \notin \Vertex(\mathcal{T}_f)$. Let $v_i:=f^i(v)$, and let $k>0$ be the least number with $f^k(v)\in \Vertex(\mathcal{T}_f)$. Let $e_0 \in \Edge(\mathcal{T}_f)$ with $v\in e_0$ and define $e_i$ to be the unique edge containing $v_i$ for any $i<k$. 
	Let $e_k$ be any edge in $\Edge_{v_k}(\mathcal{T}_f)$. For $i>k$, define $e_i$ to be the unique edge with 
	$$
	Df^{i-k}((v_k,e_k))=(v_i,e_i).
	$$
	Then $\pi((e_0,e_1,\dots))=v$.
	Hence, there is a bijection between $\Edge_{v_k}(\mathcal{T}_f)$ and $\pi^{-1}(v)$.
\end{proof}

Now we investigate the Cantor-Bendixson rank of $\Julia_f \cap \mathcal{T}_f$. We consider $\Path_{\mathcal{G}_f}(\infty)$ first and then consider the intersection $\Julia_f \cap \mathcal{T}_f$.

\begin{defn}[Depth of edges and paths]\label{defn:dd}
	Let $f$ be a \pcf polynomial and $\mathcal{T}_f$ be the Hubbard tree. Let $\mathcal{G}_f$ be the directed graph of the Markov map $f:\mathcal{T}_f\righttoleftarrow$. We extend the depth function on the vertices of $\mathcal{G}_f$ in Definition \ref{defn:DepthDiGraph} to other objects as follows.
	\begin{itemize}
		\item (Edges) Let $e\in \Edge(\mathcal{T}_f)$, which corresponds to a vertex $v$ for $\mathcal{G}_f$. We define $\depth(e)$ as the depth of the vertex $v$.		
		\item (Paths) Let $p \in \Path_{\mathcal{G}_f}(\infty)$ be an infinite path, and let $v \in \End(p)$. We define
		\[
			\depth(p):=\depth(v),
		\]
		which is independent of the choice of $v\in \End(p)$ because $\depth$ is constant on each strongly connected component of $\mathcal{G}_f$.
		\item (Julia points) Let $x\in \Julia_f \cap \mathcal{T}_f$. We define
		\[
			\depth(x):=\max_{p\in \pi^{-1}(x)} \depth(p),
		\]
		 where $\pi:\Path_{\mathcal{G}_f}\to \Julia_f\cap \mathcal{T}_f$ is the semi-conjugacy discussed above.
	\end{itemize}
\end{defn}

\begin{example}
	To illustrate Proposition \ref{prop:PathtoJulia} and Definition \ref{defn:dd}, let us consider a polynomial $f(z) \approx - 1.3513 z^3 -2.73903 z^2$ (see Figure \ref{fig:BasilicawithBubble}). The critical point $0$ is fixed and the other critical point $- 1.3513$ is in a period 2 cycle $\{- 1.3513, -1.66717\}$.
	\begin{figure}[ht]
		\centering
			  \resizebox{0.7\linewidth}{!}{
    \def\svgwidth{\columnwidth}
\begingroup%
  \makeatletter%
  \providecommand\color[2][]{%
    \errmessage{(Inkscape) Color is used for the text in Inkscape, but the package 'color.sty' is not loaded}%
    \renewcommand\color[2][]{}%
  }%
  \providecommand\transparent[1]{%
    \errmessage{(Inkscape) Transparency is used (non-zero) for the text in Inkscape, but the package 'transparent.sty' is not loaded}%
    \renewcommand\transparent[1]{}%
  }%
  \providecommand\rotatebox[2]{#2}%
  \newcommand*\fsize{\dimexpr\f@size pt\relax}%
  \newcommand*\lineheight[1]{\fontsize{\fsize}{#1\fsize}\selectfont}%
  \ifx\svgwidth\undefined%
    \setlength{\unitlength}{578.57148418bp}%
    \ifx\svgscale\undefined%
      \relax%
    \else%
      \setlength{\unitlength}{\unitlength * \real{\svgscale}}%
    \fi%
  \else%
    \setlength{\unitlength}{\svgwidth}%
  \fi%
  \global\let\svgwidth\undefined%
  \global\let\svgscale\undefined%
  \makeatother%
  \begin{picture}(1,0.34285714)%
    \lineheight{1}%
    \setlength\tabcolsep{0pt}%
    \put(0,0){\includegraphics[width=\unitlength,page=1]{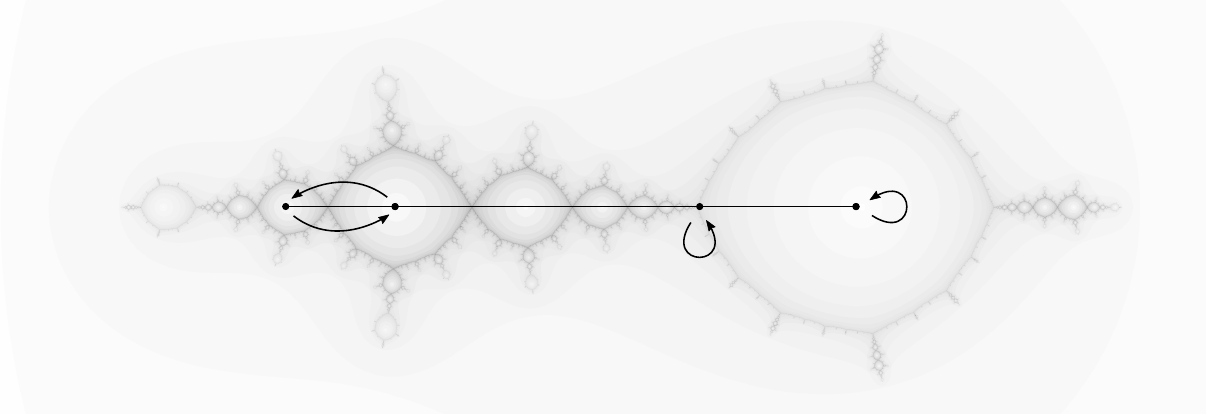}}%
    \put(0.25767755,0.1982892){\color[rgb]{0,0,0}\makebox(0,0)[lt]{\lineheight{1.25}\smash{\begin{tabular}[t]{l}$E_0$\end{tabular}}}}%
    \put(0.43251956,0.1982892){\color[rgb]{0,0,0}\makebox(0,0)[lt]{\lineheight{1.25}\smash{\begin{tabular}[t]{l}$E_1$\end{tabular}}}}%
    \put(0.62411154,0.1982892){\color[rgb]{0,0,0}\makebox(0,0)[lt]{\lineheight{1.25}\smash{\begin{tabular}[t]{l}$E_2$\end{tabular}}}}%
    \put(0.56966716,0.18185586){\color[rgb]{0,0,0}\makebox(0,0)[lt]{\lineheight{1.25}\smash{\begin{tabular}[t]{l}$v$\end{tabular}}}}%
  \end{picture}%
\endgroup%

			}
		\caption{The Julia set of $- 1.3513 z^3 -2.73903 z^2$}
		\label{fig:BasilicawithBubble}
	\end{figure}
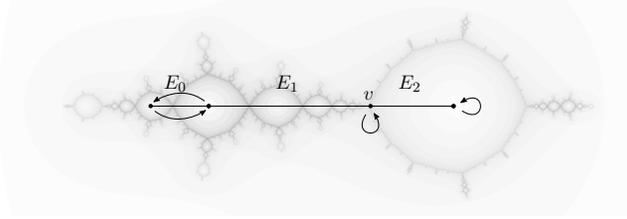

	Let us add the Julia fixed point, denoted by $v$, which is not in the post-critical set, as a vertex in the Hubbard tree. Then we have three edges $E_0,E_1,E_2$. The associated directed graph is given in Figure \ref{fig:DGBB}.
\begin{figure}[ht]
	\begin{tikzpicture}[node distance=2
		cm,gnode/.style={draw, circle, inner sep=0pt, minimum size=.15cm, fill=black}]
		\node[gnode, label={below: $E_0$}](0) {};
		\node[gnode, label={below: $E_1$}](1) [right of=0] {};
		\node[gnode, label={below: $E_2$}](2) [right of=1] {};

		\path[-stealth,thick]
		(0) edge [out=140,in=220,looseness=15] node {} (0)
		(1) edge [out=320,in=40,looseness=15] node {} (1)
		(1) edge node {} (0)
		(2) edge [out=320,in=40,looseness=15] node {} (2);
	\end{tikzpicture}
	\caption{ The directed graph associated to the Hubbard tree in Figure \ref{fig:BasilicawithBubble}.}
	\label{fig:DGBB}
\end{figure}
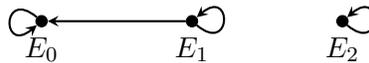

There are four possible forms of infinite paths:
\begin{itemize}
	\item $(E_0,E_0,\dots)$,
	\item $(E_1,E_1,\dots)$,
	\item $(E_1,E_1,\dots,E_1,E_0,E_0,E_0,\dots)$, and
	\item $(E_2,E_2,\dots)$.
\end{itemize}
The semi-conjugacy 	$\pi:\Path_{\mathcal{G}_f}(\infty) \to \Julia_f \cap \mathcal{T}_f$ is bijective except for
\[
	\pi((E_1,E_1,\dots))=\pi((E_2,E_2,\dots))=v.
\]
The depth of $(E_1,E_1,\dots)$ is one and the depth of $(E_2,E_2,\dots)$ is zero, so the depth of $v$ is one (see the vertex $v$ in Figure \ref{fig:BasilicawithBubble}). From the left side, we can see the depth 1 property, i.e., $v$ is a limit point of the boundaries of bounded Fatou components. From the right side, however, we can see the depth 0 property, i.e., $v$ is on the boundary of a bounded Fatou component.
\end{example}

\begin{prop}\label{prop:CBDecompofPaths}
	Let $f$ be a \pcf polynomial and $\mathcal{T}_f$ be the Hubbard tree. For the directed graph $\mathcal{G}_f$ of the Markov map $f:\mathcal{T}_f\righttoleftarrow$, let $\Pcal:=\Path_{\mathcal{G}_f}(\infty)$. Then
	\begin{equation}\label{eqn:1st of PropCBDecompofPaths}
		\Pcal^{(n)}=\{p\in \Pcal~|~ \depth(p)\ge n\},
	\end{equation}
	where $\Pcal^{(n)}$ is the $n^{th}$ Cantor-Bendixson derivative. Hence, the Cantor-Bendixson rank of $\Pcal$ is $d+1$ where
	\begin{equation}
		d:=\max \{\depth(e)~|~e\in \Edge(\mathcal{T}_f),~\depth(e)<\infty\}.
	\end{equation}
	More precisely, the Cantor-Bendixson decomposition of $\Pcal$ is given by
	\begin{equation}\label{eqn:2nd of PropCBDecompofPaths}
		\Pcal= \{ p \in \Pcal~|~ \depth(p) < \infty\} \sqcup \{p \in\Pcal~|~ \depth(p)=\infty \},
	\end{equation}
	which is also equivalent to
	\begin{equation}\label{eqn:3rd of PropCBDecompofPaths}
	\Pcal= \{ p \in \Pcal~|~ \depth(p) \le d\} \sqcup \{p \in \Pcal~|~ \depth(p)>d \}.
	\end{equation}
%
\end{prop}
\begin{proof}
	The equivalence between equations \eqref{eqn:2nd of PropCBDecompofPaths} and \eqref{eqn:3rd of PropCBDecompofPaths} follows from the fact that for each $e\in \Edge(\mathcal{T}_f)$ either $\depth(e)=\infty$ or $\depth(e)\le d$. 
	
	For any infinite directed path $p \in \Pcal$, there is a unique decomposition $p=p_1 \sqcup p_2$, which we call the {\it head-tail} decomposition of $p$, so that $p_1$ is disjoint from the ending component $\End(p)$ and $p_2$ is contained in $\End(p)$. We call $p_1$ the {\it head} and $p_2$ the {\it tail} of $p$ respectively. 
	
	 For $p \in \Pcal$, we define a subset $Z(p) \subset \Pcal$ as the collection of paths $p'$ with properties that (i) the heads of $p$ and $p'$ are the same and (ii) $\End(p)=\End(p')$. 
	
	If $\depth(p)=\infty$, then $Z(p)$ is a Cantor set containing $p$ so that $Z(p)$ survives forever when we iteratively take Cantor-Bendixson derivatives for $\Pcal$. Hence the set $\{p \in \Pcal~|~ \depth(p)=\infty\}$ is contained in the perfect set component in the Cantor-Bendixson decomposition. Then \eqref{eqn:3rd of PropCBDecompofPaths} follows from \eqref{eqn:1st of PropCBDecompofPaths}.
	
	Let us show \eqref{eqn:1st of PropCBDecompofPaths}. If $\depth(p)<\infty$, then $Z(p)=\{p\}$. 
	Suppose $\depth(p)>0$. Let $C=\End(p)$ be the ending component of $p$, which is cyclic. Then there is a cyclic strongly connected component $C'$ of $\mathcal{G}$ so that $\depth(C)=\depth(C')+1$ and there is a path $\delta$ from $C$ to $C'$. Let $p=p_1\cup p_2$ be the decomposition as above. Then $p_2$ is a path which infinitely rotates along $C$. Define $q_n$ be the concatenation $p_1 \cup C^n \cup \delta \cup C'^\infty$ where $C^n$ is the $n$-times rotation along $C$ and $C'^\infty$ is the infinite rotations along $C'$. Then $\{q_n\}$ converges to $p$ with $\depth(q_n)=\depth(p)-1$. 
	
	It is easy to show that every $p \in \Pcal$ with $\depth(p)=0$ is an isolated point of $\Pcal$. Then the equations \eqref{eqn:1st of PropCBDecompofPaths} follow from an induction argument.
\end{proof}

As a corollary of Proposition \ref{prop:CBDecompofPaths}, we have the following theorem.

\begin{theorem}[Cantor-Bendixson rank of $\Julia_f \cap \mathcal{T}_f$]\label{thm:CBrankofJuliaIntersectGraph}
	Let $f$ be a \pcf polynomial and $\mathcal{T}_f$ be the Hubbard tree. For each $k\ge 0$, we have
	\begin{equation}\label{eqn:Depth of Pts in Julia Hubb tree}
	(\Julia_f \cap \mathcal{T}_f)^{(k)}=\{x\in \Julia_f \cap \mathcal{T}_f~|~\depth(x)\ge k\}
	\end{equation}
	where $(\Julia_f \cap \mathcal{T}_f)^{(k)}$ is the $k$-th Cantor-Bendixson derivative of $\Julia_f \cap \mathcal{T}_f$. Then we have 
	\begin{multline}\label{eqn:CBRank of Julia and Hubb}
		\CBrank(\Julia_f \cap \mathcal{T}_f) = \\
			\max\{ \{0\}\cup\{1+\depth(e)~|~\depth(e)<\infty,~ e\in\Edge(\mathcal{T}_f)\} \},
	\end{multline}
	and for any $e\in \Edge(\mathcal{T}_f)$ we have
	\begin{multline}\label{eqn:CBrank of Julia and Hubb edge}
		\CBrank(\Julia_f \cap e) = \\
			\max\{ \{0\}\cup\{1+\depth(e')~|~\depth(e')<\infty,~ e\ge e'\} \}.
	\end{multline}
\end{theorem}
\begin{proof}[Proof of Theorem \ref{thm:CBrankofJuliaIntersectGraph}]
	For the semi-conjugacy $\pi:\Path_{\mathcal{G}_f}(\infty)\to \Julia_f \cap \mathcal{T}_f$ and $x\in \Julia_f \cap \mathcal{T}_f$, let $\{p_1,p_2,\dots,p_k\}=\pi^{-1}(x)$. Suppose that $\depth(x)=\max_i \depth(p_i)$ is finite. Since $\pi$ is continuous, it follows from Proposition \ref{prop:CBDecompofPaths} that $x$ is removed at the $(\depth(x)+1)^{th}$ Cantor-Bendixson derivative of $\Julia_f \cap \mathcal{T}_f$. Then the equation \eqref{eqn:Depth of Pts in Julia Hubb tree} follows. 
	
	The equation \eqref{eqn:CBrank of Julia and Hubb edge} follows from the same argument restricted to the minimal $f$-invariant subgraph of $\mathcal{T}_f$ containing $e$. 
\end{proof}

Now we are ready to prove Theorem \ref{thm:CBDecompJuliaIntersectGraph}.

\begin{proof}[Proof of Theorem \ref{thm:CBDecompJuliaIntersectGraph}]
	For $d:=\max_{e\in \Edge(\mathcal{T}_f)} \depth(e)$, it follows from Proposition \ref{prop:CBDecompofPaths} and Theorem \ref{thm:CBrankofJuliaIntersectGraph} that for any $k\ge d+1$, 
	\[
		(\Julia_f \cap \mathcal{T}_f)^{(k)}= \{x \in \Julia_f \cap \mathcal{T}_f~|~ \depth(x)=\infty\}.
	\]
	Hence we obtain the equation \eqref{eqn:CBdecompositionof Julia Hubbard}.
	The equation \eqref{eqn:TopComplexity and depth} also follows.
	
	Since $h(f)=0$ if and only if $\depth(e)<\infty$ for every $e\in \Edge(\mathcal{T}_f)$, we have that $h(f)=0$ if and only if the perfect set component in the Cantor-Bendixson decomposition is the empty set. Thus we have the equivalences in the statements.
\end{proof}

\section{The main molecule $\Mol_d$}\label{sec:mm}
For quadratic polynomials, the {\em main molecule} is the closure of the union of all hyperbolic components that can be obtained from the main cardioid through a finite chain of bifurcations.
In this section, we will define the {\em main molecule} for higher degree polynomials and discuss some subtleties that occur in higher degrees.

Recall that a degree $d$ polynomial $f(z) = c_dz^d+c_{d-1}z^{d-1}+...+c_0$ is called 
\begin{itemize}
\item {\bf monic} if $c_d = 1$, and
\item {\bf centered} if $c_{d-1} = 0$.
\end{itemize}

Let $\mathcal{P}_d \cong \C^{d-1}$ be the space of degree $d$ monic and centered polynomials.
Note that every degree $d$ polynomial is affine conjugate to a monic and centered one.

The space $\mathcal{P}_d$ is regarded as the space of {\em marked polynomials}.
More precisely, if $f \in \mathcal{P}_d$ has connected Julia set, then it has a unique B\"ottcher map whose derivative is $1$ at infinity.
We call the external ray of angle $0$ under this B\"ottcher coordinate the {\em marked external ray}.
So generically, there are $d-1$ monic and centered polynomials that are affine conjugate.
Thus, $\mathcal{P}_d$ is a branched covering of the moduli space of degree $d$ polynomials.

We use the notion of $J$-conjugacy following \cite[\S3]{McM88}.

\begin{defn}[$J$-conjugacy]\label{defn:jc}
Let $f, g\in \mathcal{P}_d$ be two polynomials. They are {\em $J$-conjugate} if there exists a map $\phi: \C \longrightarrow \C$ so that
\begin{itemize}
\item $\phi$ is quasiconformal on $\C$ and preserves the marked external rays;
\item $\phi(J_f) = J_g$, where $J_f, J_g$ are Julia sets of $f, g$; and
\item $\phi \circ f(z) = g \circ \phi(z)$ for all $z\in J_f$.
\end{itemize}
We say that 
\begin{itemize}
	\item $f$ and $g$ are {\em weakly $J$-conjugate} if $\phi: \C \longrightarrow \C$ is only assumed to be a homeomorphism, and
	\item $f$ is {\em $J$-semi-conjugate} to $g$ if $\phi: \C \longrightarrow \C$ is only assumed to be a surjective continuous map.
\end{itemize}
\end{defn}

A rational map $f$ is {\em sub-hyperbolic} if any critical point is either (i) in an attracting basin or (ii) in the Julia set $\Julia_f$ and preperiodic. A rational map is {\em hyperbolic} if every critical point is in an attracting basin.

\begin{defn}[Relative hyperbolic components]
	For $f \in \mathcal{P}_d$ a \pcf polynomial, the {\em relative hyperbolic component $\mathcal{H}_f$ of $f$} is defined by
	$$
		\mathcal{H}_f:= \{g \in \mathcal{P}_d: g \text{ is $J$-conjugate to } f\}.
	$$ 
\end{defn}
We omit the subscript $f$ when we do not need to specify $f$. Following are some properties on relative hyperbolic components.
\begin{enumerate}
\item Each \shc $\mathcal{H}$ contains a unique \pcf polynomial $f$ \cite{McM88}, and we call $f$ the {\em center} of $\mathcal{H}$.
\item If $f\in \mathcal{P}_d$ is a hyperbolic post-critically finite polynomial, then by \cite{MSS83}, the \shc $\mathcal{H}_f$ of $f$ is the usual hyperbolic component of $\mathcal{P}_d$ containing $f$.
\item If $f\in \mathcal{P}_d$ is a non-hyperbolic post-critically finite polynomial, then the \shc $\mathcal{H}_f$ consists of those sub-hyperbolic polynomials obtained by deforming the dynamics in the bounded Fatou components (see \cite{MS98}.)
In this case, the dimension of $\mathcal{H}_f$ is smaller than the dimension of $\mathcal{P}_d$ (see \S \ref{subsec:lm} for a model space of $\mathcal{H}_f$).
\item The \shc $\mathcal{H}_f$ is a singleton set if and only if the Julia set $J_f$ is a dendrite, i.e., $f$ has no bounded Fatou component. 
\item In degree $2$, if $f\in \mathcal{P}_2$ is a non-hyperbolic post-critically finite polynomial, then the Julia set $J_f$ is a {\em dendrite}, and thus $\mathcal{H}_f = \{f\}$.
\end{enumerate}

Two \shcs $\mathcal{H}$ and $\mathcal{H}'$ are said to be {\em adjacent} if $\partial{\mathcal{H}} \cap \partial{\mathcal{H}'} \neq \emptyset$.
We say a \shc $\mathcal{H}'$ has a {\em finite distance} from $\mathcal{H}$ if there exists a finite sequence of \shcs $\mathcal{H}_0 = \mathcal{H}, \mathcal{H}_1, ..., \mathcal{H}_k = \mathcal{H}'$ so that $\mathcal{H}_i$ is adjacent to $\mathcal{H}_{i+1}$.

\begin{defn}[Main molecule]
Let $\mathcal{H}_d \subseteq \mathcal{P}_d$ be the main hyperbolic component, i.e., the hyperbolic component containing $z^d$.
Let $\mathfrak{S}$ be the set of \shcs of finite distance from $\mathcal{H}_d$.
We define the degree $d$ {\em main molecule} as
$$
\Mol_d := \overline{\bigcup_{\mathcal{H}\in \mathfrak{S}} \mathcal{H}}.
$$
\end{defn}

\begin{rmk}
We remark that one can naturally replace \shcs in the definition of the main molecule by hyperbolic components and ask whether the closure is still the same as $\Mol_d$ (see Question \ref{quest:DenseHyp}).

For quadratic polynomials, this is true as $\Mol_2$ contains no non-hyperbolic \pcf polynomial, so any \shc $\mathcal{H} \subseteq \Mol_d$ is in fact a hyperbolic component.
We do not know the answer in higher degrees (see \S \ref{subsec:eg} for a discussion on some subtleties in higher degrees).

We remark that the discussion in \S \ref{subsec:eg} suggests it is more natural to consider \shcs in higher degrees (see Remark \ref{rmk:ibc}).
\end{rmk}

\subsection{A model of \shc}\label{subsec:lm}
In this subsection, we will discuss how to model relative hyperbolic components using Blaschke products. This idea was introduced in \cite{MilPoi92}, which is not published. The published article \cite{Milnor12} is a revised version of \cite{MilPoi92}.

\begin{defn}[Mapping scheme]
	A {\em mapping scheme} $\mathcal{S} = (|\mathcal{S}|, \Phi, \delta)$ consists of a finite set $|\mathcal{S}|$, whose elements are called vertices, together with a map $\Phi = \Phi_{\mathcal{S}}: |\mathcal{S}| \longrightarrow |\mathcal{S}|$, and a degree function $\delta : |\mathcal{S}| \longrightarrow \Z_{\geq 1}$, satisfying two conditions:
	\begin{itemize}
		\item (Minimality) Any vertex of degree $1$ is the iterated forward image of some vertex of degree $\geq 2$, and
		\item (Hyperbolicity) Every periodic orbit under $\Phi$ contains at least one vertex of degree $\geq 2$.
	\end{itemize}
	We define the degree of the scheme as
	$\deg(\mathcal{S}) = 1+ \sum_{s\in |\mathcal{S}|} (\delta(s)-1)$.
\end{defn}

Let $f: \D\longrightarrow \D$ be a proper holomorphic map of degree $d\geq 1$. 
It can be uniquely written as a {\em Blaschke product}
$$
f(z) = e^{i\theta}\prod_{i=0}^d \frac{z-a_i}{1-\overline{a_i}z},
$$
where $|a_i|<1$.

By the Denjoy-Wolff theorem, there is a unique non-repelling fixed point of $f$ on $\overline{\D}$, which classifies a Blaschke product $f$ into exactly three categories:
\begin{itemize}
\item $f$ is {\bf \em interior-hyperbolic} or simply {\bf \em hyperbolic} if $f$ has an attracting fixed point in $\D$,
\item $f$ is {\bf \em parabolic} if $f$ has a parabolic fixed point on $\mathbb{S}^1$, and
\item $f$ is {\bf \em boundary-hyperbolic} if $f$ has an attracting fixed point on $\mathbb{S}^1$.
\end{itemize}
The parabolic Blaschke products can be further divided into {\bf \em singly parabolic} or {\bf \em doubly parabolic} depending on the multiplicities for the parabolic fixed points.
The Julia set of a hyperbolic or a doubly parabolic Blaschke product is the circle $\mathbb{S}^1$, while the Julia set of a singly parabolic or a boundary-hyperbolic Blaschke product is a Cantor set on $\mathbb{S}^1$.

Following \cite{Milnor12}, we say that a Blaschke product $f$ is
\begin{itemize}
\item {\bf \em 1-anchored} if $f(1) = 1$,
\item {\bf \em fixed point centered} if $f(0) = 0$, and
\item {\bf \em zeros centered} if the sum
$$
a_1+... + a_d
$$
of the points of $f^{-1}(0)$ (counted with multiplicity) is equal to $0$.
\end{itemize}
We define $\BP_{d, \fc}$ and $\BP_{d, \zc}$ as the space of all 1-anchored Blaschke products of degree $d$ which are respectively fixed point centered or zeros centered.
When $d=1$, $\BP_{1, \fc} = \BP_{1, \zc}$ consists of only the identity map.

\begin{defn}[Blaschke model space]\label{defn:bms}
Let $\mathcal{S}=(|\mathcal{S}|, \Phi, \delta)$ be a mapping scheme. 
We associate the {\em Blaschke model space} $\BP^\mathcal{S}$ consisting of all proper holomorphic maps
$$
	\bp: |\mathcal{S}| \times \D \longrightarrow |\mathcal{S}| \times \D
$$
such that $\mathcal{F}$ carries each $\{s\}\times \D$ onto $\{\Phi(s)\} \times \D$ by an 1-anchored Blaschke product 
$$
	(s, z) \mapsto (\Phi(s), \bp_s(z))
$$
of degree $\delta(s)$ which is either fixed point centered or zero-centered according to whether $s$ is periodic or aperiodic under $\Phi$.

\end{defn}

\begin{defn}[Mapping schemes of sub-hyperbolic maps]
	Let $\mathcal{H}_f \subseteq \mathcal{P}_d$ be a \shc with the \pcf center $f$.
	We define the {\em mapping scheme $\mathcal{S}_f$ associated to $f$} in the following way.
	Vertices $s_U \in |\mathcal{S}_f|$ are in correspondence with bounded Fatou components $U$ which contain a critical or post-critical point.
	The associated map $F = F_f: |\mathcal{S}_f| \longrightarrow |\mathcal{S}_f|$ carries $s_U$ to $s_{f(U)}$ and the degree $\delta(s_U)$ is defined to be the degree of $f: U \longrightarrow f(U)$.
\end{defn}

	We remark that we exclude the fixed Fatou component containing $\infty$ from the mapping scheme $\mathcal{S}_f$.

\begin{example}
	See Figure \ref{fig:BHM}. For a sub-hyperbolic cubic polynomial $f(z)=z^3-\frac{3}{2}z+\frac{1}{\sqrt{2}}$, the associated mapping scheme $(|\mathcal{S}|,\Phi, \delta)$ is defined as follows:
	\begin{itemize}
		\item $|S|=\{0, \frac{1}{\sqrt{2}}\}$,
		\item $\Phi(0)=\frac{1}{\sqrt{2}}$ and $\Phi(\frac{1}{\sqrt{2}})=0$, and
		\item $\delta(0)=1$ and $\delta(\frac{1}{\sqrt{2}})=2$.
	\end{itemize}
\end{example}

	Let $f \in \mathcal{H}$ be a sub-hyperbolic polynomial and $U$ be a bounded critical or post-critical Fatou component. We can uniformize the dynamics of $f:U\longrightarrow f(U)$ to an 1-anchored fixed point or zero centered Blaschke product (see \cite[\S 5]{Milnor12} for more details).

\begin{theorem}\label{thm:HomeoToModelSp}
Let $\mathcal{H} \subseteq \mathcal{P}_{d}$ be a \shc and $\mathcal{S} = \mathcal{S}_f$ be the corresponding mapping scheme. 
Then there is a diffeomorphism $\Psi: \mathcal{H} \to \BP^\mathcal{S}$ that maps $f\in \mathcal{H}$ to its Blaschke product model $\Fcal$.
\end{theorem}

Theorem \ref{thm:HomeoToModelSp} is proven for the hyperbolic case in \cite[Theorem 5.1]{Milnor12}. Since the same argument also works for the case of relative hyperbolic components, we omit its proof. We also refer the reader to \cite[\S 6]{McM88} for the idea of uniformization of more general types of Fatou components.

There are finitely many different choices of diffeomorphisms between $\mathcal{H}_f$ and $\BP^\mathcal{S}$. These different diffeomorphisms correspond to different choices of boundary markings.

\begin{defn}[Boundary marking of sub-hyperbolic polynomials]\label{defn:BdryMarking}
	Let $f$ be a sub-hyperbolic polynomial with connected Julia set. 
	A {\em boundary marking} of $f$ is a choice of a point $q(U) \in \partial U$ for any $U \in |\mathcal{S}_f|$, i.e., any bounded critical or post-critical Fatou component, that satisfies $f(q(U)) = q(f(U))$. 
	We call $q(U)$ the {\em marked point} of $U$.
	
	Given a boundary marking $q$, we have a unique family of continuous maps $\phi_U:\mathbb{S}^1 \to \partial U$ associated to $U\in |\mathcal{S}_f|$ satisfying $f\circ \phi_U= \phi_{f(U)} \circ z^{\delta(U)}$ and $\phi_U(e^{i\cdot 0})=q(U)$. 
	By abusing notations, we also call such a family of continuous maps $\{\phi_U\}_{U\in |\mathcal{S}_f|}$ a boundary marking of $f$.
\end{defn}

Once a choice of boundary marking is made for the center $f \in \mathcal{H}_f$, we get a boundary marking for any map $g \in \mathcal{H}_f$ by tracing the marked points through quasiconformal conjugacies. Given a boundary marking $q$, the diffeomorphism $\Psi: \mathcal{H}_f \to \BP^\mathcal{S}$ is defined by uniformizing the dynamics on each post-critical Fatou component to the disk $\D$ in such a way that the marked boundary point $q(U)$ is sent to $1 \in \partial \D_{s_U}$ (see \cite[\S 5]{Milnor12} for more details).

We define markings for external Fatou components, the Fatou components containing $\infty$, in a similar manner with a slight modification.

\begin{defn}[$\infty$-markings of polynomials]\label{defn:InfMarking}
	Let $f$ be a degree-$d$ polynomial with connected Julia set. Let $\partial_\infty \C$ be the boundary at infinity of the complex plane that is homeomorphic to the circle. The boundary $\partial_\infty \C$ parametrizes the set of external angles so that $f$ induces a degree $d$ endomorphism $f:\partial_\infty \C \righttoleftarrow$. An {\em $\infty$-marking of a polynomial $f$} of degree $d$ is a homeomorphism $\phi:\mathbb{S}^1 \to \partial_\infty \C$ satisfying $\phi \circ z^d= f \circ \phi$.
	The polynomial $f$ is called {\em $\infty$-marked}, or simply {\em marked} if an $\infty$-marking is chosen.
\end{defn}

	There are ($d-1$) choices of $\infty$-markings, which correspond to the ambiguity of B\"{o}ttcher coordinates of the infinity. 
	For a monic and centered polynomial, there is a {\em canonical B\"{o}ttcher coordinate} whose derivative at the point at infinity is the identity. 
	Therefore, we can regard a monic and centered polynomial $f$ with connected Julia set a {\em marked polynomial}.


\subsection{An illustrating example}\label{subsec:eg}
In this subsection, we discuss some differences between quadratic and higher degree main molecules using the following example.
Consider a \pcf polynomial 
$$
f(z) = z^3-\frac{3}{2}z+\frac{1}{\sqrt{2}}.
$$
It has two critical points $c_1 = -\frac{1}{\sqrt{2}}$ and $c_2 = \frac{1}{\sqrt{2}}$ with the dynamics
$$
c_1 \to \sqrt{2} \to \sqrt{2} \; \; \text{ and } \;\; c_2 \to 0 \to c_2.
$$
Its Julia set is depicted in Figure \ref{fig:BHM}.
The following three properties of this example are not held for quadratic polynomials. 

\begin{figure}[ht]
  \centering
  \resizebox{0.6\linewidth}{!}{
    \def\svgwidth{\columnwidth}
\begingroup%
  \makeatletter%
  \providecommand\color[2][]{%
    \errmessage{(Inkscape) Color is used for the text in Inkscape, but the package 'color.sty' is not loaded}%
    \renewcommand\color[2][]{}%
  }%
  \providecommand\transparent[1]{%
    \errmessage{(Inkscape) Transparency is used (non-zero) for the text in Inkscape, but the package 'transparent.sty' is not loaded}%
    \renewcommand\transparent[1]{}%
  }%
  \providecommand\rotatebox[2]{#2}%
  \newcommand*\fsize{\dimexpr\f@size pt\relax}%
  \newcommand*\lineheight[1]{\fontsize{\fsize}{#1\fsize}\selectfont}%
  \ifx\svgwidth\undefined%
    \setlength{\unitlength}{388.5bp}%
    \ifx\svgscale\undefined%
      \relax%
    \else%
      \setlength{\unitlength}{\unitlength * \real{\svgscale}}%
    \fi%
  \else%
    \setlength{\unitlength}{\svgwidth}%
  \fi%
  \global\let\svgwidth\undefined%
  \global\let\svgscale\undefined%
  \makeatother%
  \begin{picture}(1,0.95559846)%
    \lineheight{1}%
    \setlength\tabcolsep{0pt}%
    \put(0,0){\includegraphics[width=\unitlength,page=1]{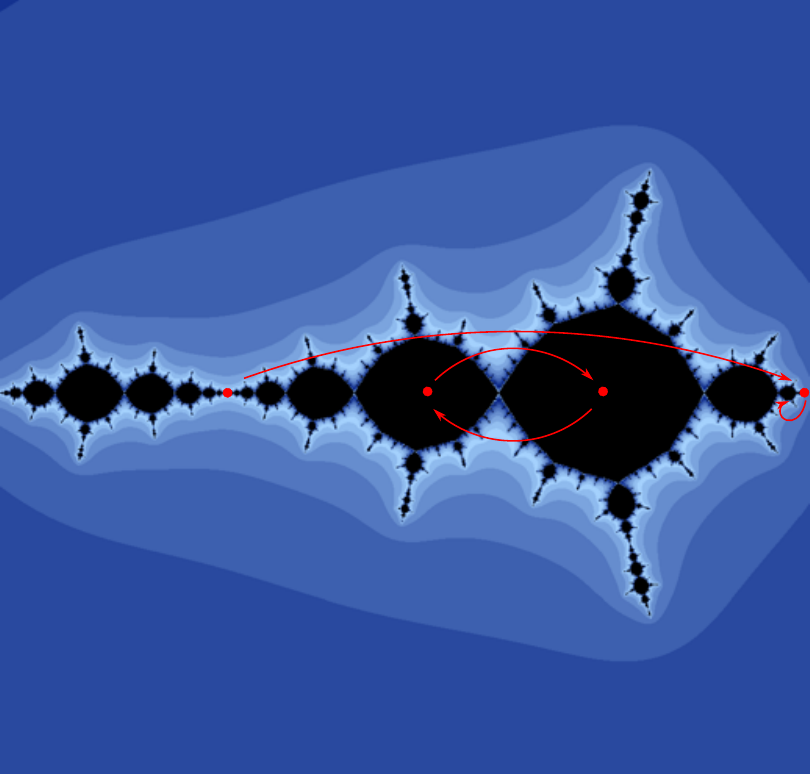}}%
    \put(0.75620744,0.46106711){\color[rgb]{0,0,0}\makebox(0,0)[lt]{\lineheight{1.25}\smash{\begin{tabular}[t]{l}$c_2$\end{tabular}}}}%
    \put(0.28569163,0.44947067){\color[rgb]{1,0,0}\makebox(0,0)[lt]{\lineheight{1.25}\smash{\begin{tabular}[t]{l}$c_1$\end{tabular}}}}%
    \put(0.75232982,0.45159653){\color[rgb]{1,0,0}\makebox(0,0)[lt]{\lineheight{1.25}\smash{\begin{tabular}[t]{l}$c_2$\end{tabular}}}}%
    \put(0,0){\includegraphics[width=\unitlength,page=2]{BHM.pdf}}%
    \put(0.61545218,0.83973613){\color[rgb]{0,0,0}\makebox(0,0)[lt]{\lineheight{1.25}\smash{\begin{tabular}[t]{l}$\mathcal{R}_1$\end{tabular}}}}%
    \put(0.59855316,0.11692592){\color[rgb]{0,0,0}\makebox(0,0)[lt]{\lineheight{1.25}\smash{\begin{tabular}[t]{l}$\mathcal{R}_2$\end{tabular}}}}%
  \end{picture}%
\endgroup%

  }
  \caption{The Julia set of $f(z) = z^3-\frac{3}{2}z+\frac{1}{\sqrt{2}}$.}
  \label{fig:BHM}
\end{figure}

{\bf Property (1):} $f$ is not on the closure of any hyperbolic component with connected Julia set.

Suppose for contradiction that there exists some hyperbolic component $\mathcal{H}$ so that $f \in \overline{\mathcal{H}}$.
Since $f$ is not hyperbolic, $f \in \partial \mathcal{H}$.
Then there exists a sequence $f_n \in \mathcal{H} \subseteq \mathcal{P}_d$ with $f_n \to f$.
Let $c_{1,n}, c_{2,n}$ be the critical points of $f_n$ with $c_{1,n} \to c_1$ and $c_{2,n} \to c_2$.

Since attracting Fatou components are stable under perturbation, $c_{2,n}$ is contained in a period $2$ Fatou component $U_n$ of $f_n$ for all $n$.
Since there is only one non-repelling periodic cycle for $f$, the Fatou component $V_n$ containing $c_{1,n}$ is eventually mapped to $U_n$. 
Let $l$ be the smallest integer so that $f_n^l (V_n) = U_n$.
Note that the sequence $f_n$ is contained in a single hyperbolic component, so $l$ does not depend on $n$.

On the other hand, we note that the external rays $\mathcal{R}_1$ and $\mathcal{R}_2$ in Figure \ref{fig:BHM} land at the same repelling fixed point, and their union separates $c_1$ from $c_2$.
This condition is stable under perturbation, and we conclude that $c_{1,n} \notin U_n$. 
By considering the $l$-times pull back of the external rays, we can similarly conclude that $c_{1,n} \notin f_n^{-l}(U_n)$.
This is a contradiction, and Property (1) follows.

{\bf Property (2):} $f$ is an accumulation point of hyperbolic \pcf polynomials.

Note that there exists a sequence of points $p_k \to \sqrt{2}$ on the real line so that $f^k(p_k) = c_2$.
One can construct a perturbation $f_k$ of $f$ satisfying $f_k(c_{1,k}) = p_{k,k}$ where $c_{1,k}$ and $p_{k,k}$ are the corresponding perturbation of the critical point $c_1$ and the point $p_k$.
Note that by construction, $f_k$ is a hyperbolic \pcf polynomial.
This phenomenon is illustrated in Figure \ref{fig:BL}.

\begin{figure}[ht]
  \centering
  \resizebox{0.6\linewidth}{!}{
    \def\svgwidth{\columnwidth}
\begingroup%
  \makeatletter%
  \providecommand\color[2][]{%
    \errmessage{(Inkscape) Color is used for the text in Inkscape, but the package 'color.sty' is not loaded}%
    \renewcommand\color[2][]{}%
  }%
  \providecommand\transparent[1]{%
    \errmessage{(Inkscape) Transparency is used (non-zero) for the text in Inkscape, but the package 'transparent.sty' is not loaded}%
    \renewcommand\transparent[1]{}%
  }%
  \providecommand\rotatebox[2]{#2}%
  \newcommand*\fsize{\dimexpr\f@size pt\relax}%
  \newcommand*\lineheight[1]{\fontsize{\fsize}{#1\fsize}\selectfont}%
  \ifx\svgwidth\undefined%
    \setlength{\unitlength}{388.5bp}%
    \ifx\svgscale\undefined%
      \relax%
    \else%
      \setlength{\unitlength}{\unitlength * \real{\svgscale}}%
    \fi%
  \else%
    \setlength{\unitlength}{\svgwidth}%
  \fi%
  \global\let\svgwidth\undefined%
  \global\let\svgscale\undefined%
  \makeatother%
  \begin{picture}(1,0.95559846)%
    \lineheight{1}%
    \setlength\tabcolsep{0pt}%
    \put(0,0){\includegraphics[width=\unitlength,page=1]{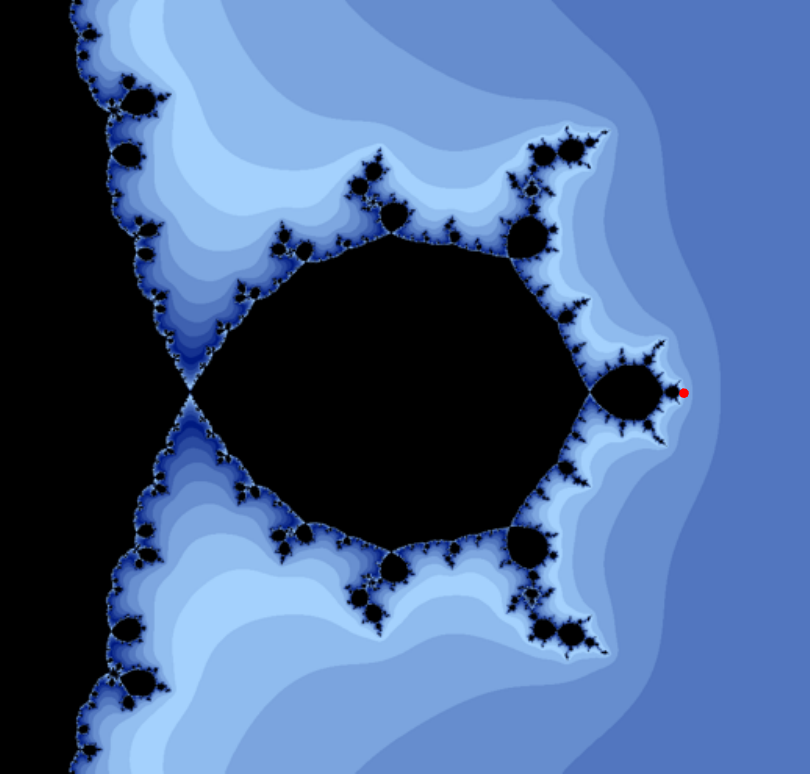}}%
    \put(0.85674348,0.45813454){\color[rgb]{1,0,0}\makebox(0,0)[lt]{\lineheight{1.25}\smash{\begin{tabular}[t]{l}$f$\end{tabular}}}}%
  \end{picture}%
\endgroup%

  }
  \caption{A 1-dimensional slice of the parameter space. A sequence of hyperbolic components converges to the `tip', which corresponds to the \pcf polynomial $f(z) = z^3-\frac{3}{2}z+\frac{1}{\sqrt{2}}$.}
  \label{fig:BL}
\end{figure}

{\bf Property (3):} The \shc $\mathcal{H}_f$ is adjacent to the main hyperbolic component $\mathcal{H}_d$.

Note that there exists a unique geometrically finite polynomial $g \in \partial \mathcal{H}_f$ that is weakly $J$-conjugate to $f$. Recall that a polynomial is {\em geometrically finite} if every critical point on its Julia set is preperiodic.
This polynomial $g$ has a unique parabolic fixed point corresponding the continuation of the common landing points of $\mathcal{R}_1$ and $\mathcal{R}_2$ in Figure \ref{fig:BHM}.
The associated pointed Hubbard tree is {\em pointed iterated-simplicial} (see \cite{L21a} or \S \ref{subsec:bc} for the definition). 
So $g\in \partial \mathcal{H}_d$ by \cite[Theorem 1.1]{L21a}.

\subsection{Bifurcation}
Let $\mathcal{H}_f$ be a relative hyperbolic component.
Unlike the quadratic case, it is known that geometrically finite polynomials $h \in \partial \mathcal{H}_f$ may have simpler Julia set, that is, the lamination $\lambda(h)$ satisfies $\lambda(h) \subsetneq \lambda(f)$ (see \cite[Appendix B]{GM93} or \cite[\S 8]{IK12}).

Thus, to describe dynamically meaningful transitions from one \shc to another, we give the following definition.

Let $\mathcal{H}$ be a \shc with center $f$.
Let $g \in \partial{\mathcal{H}}$ be a geometrically finite polynomial.
We say $g$ is a {\em root} of the \shc $\mathcal{H}$ if $g$ is weakly $J$-conjugate to $f$.
\begin{defn}[Bifurcation]\label{defn:bif}
A \shc $\mathcal{H}_f$ is said to {\em bifurcate} to another \shc $\mathcal{H}_g$ if there exists a geometrically finite polynomial $h \in \partial \mathcal{H}_f \cap \partial \mathcal{H}_g$ so that
\begin{itemize}
\item $h$ is a root of $\mathcal{H}_g$, and
\item $f$ is $J$-semi-conjugate to $h$.
\end{itemize}
\end{defn}
We remark that since the Julia set of a geometrically finite polynomial is locally connected, the second condition of Definition \ref{defn:bif} can be replaced by
$\lambda(f) \subseteq \lambda(g)$, where $\lambda(f)$ and $\lambda(g)$ are the polynomial laminations of $f$ and $g$.

\begin{defn}[Bifurcation complexity]
Let $f \in \Mol_d$ be a \pcf polynomial.
The {\em bifurcation complexity} $\mathscr{C}_b(f)$ of $f$ is the minimum number $k$ of \shcs $\mathcal{H}^0,..., \mathcal{H}^k$ so that
\begin{itemize}
\item $\mathcal{H}^0 = \mathcal{H}_d$,
\item $\mathcal{H}^k = \mathcal{H}_f$, and
\item $\mathcal{H}^i$ bifurcates to $\mathcal{H}^{i+1}$.
\end{itemize}
\end{defn}

\begin{rmk}\label{rmk:ibc}
We remark that a priori, $\mathscr{C}_b(f)$ may be infinite, i.e., $\mathcal{H}_f$ may not be obtaind by a finite chain of bifurcations from $\mathcal{H}_d$.
We will show that any \pcf polynomial $f \in \Mol_d$ has finite bifurcation complexity (see Theorem \ref{thm:A} and Theorem \ref{thm:bcbound}).
On the other hand, if we used bifurcations through hyperbolic components (instead of relative hyperbolic components) to define the bifurcation complexity, the map $f(z) = z^3-\frac{3}{2}z+\frac{1}{\sqrt{2}}$ would have bifurcation complexity $+\infty$.
Therefore, for our purposes, \shcs are more natural to consider in higher degrees.
\end{rmk}

\section{Simplicial tuning and simplicial quotient}\label{sec:ss}
In this section, we study two operations for \pcf polynomials, {\em simplicial tuning} and {\em simplicial quotient}.
The main result we will prove in this section is the following.
\begin{theorem}\label{thm:fb}
Let $f \in \mathcal{P}_d$ be a \pcf polynomial.
Then $f$ has core entropy zero if and only if $f$ is obtained by a finite sequence of simplicial tunings from $z^d$.
\end{theorem}

Let $f$ be a \pcf polynomial with core entropy zero.
Recall that the {\em combinatorial complexity} $\mathscr{C}_{c}(f)$ is the smallest number of simplicial tunings needed to get $f$, and the {\em growth rate complexity} $\mathscr{C}_{gr}(f):= \alpha+1$ where the Markov matrix $A_f$ has growth rate $\| A_f^n\| \asymp  n^\alpha$.
We will also prove:
\begin{theorem}\label{thm:cgrb}
Let $f\in \mathcal{P}_d$ be a \pcf polynomial with core entropy zero.
Then
$$
\mathscr{C}_{c}(f) = \mathscr{C}_{gr}(f).
$$
\end{theorem}

\subsection{Abstract Hubbard tree and abstract Hubbard forests}
We will define the simplicial tuning in a combinatorial way by using the notions of abstract Hubbard tree and abstract Hubbard forests. We refer the readers to \cite{Poi10, Poi13} for more details on abstract Hubbard trees and forests.

Recall that for a \pcf polynomial $f$, the Hubbard tree is the regulated hull of the post-critical set (see \cite{DH85, Poi10}).
It is a finite invariant tree that gives a combinatorial description for the dynamics of $f$.
Conversely, one can define the combinatorial notion of an abstract Hubbard tree.
The following definition is from \cite{Poi10}.
\begin{defn}[Abstract Hubbard trees]\label{defn:aht1}
Let $T$ be a finite tree with vertex set $V$.
An {\em angle function} is a function $\angle_v(l, l') \in \Q/\Z$ which assigns a rational number modulo $1$ to each pair of edges $l, l'$ meeting at a common vertex $v$, so that
\begin{itemize}
\item $\angle_v(l, l') = -\angle_v(l', l)$,
\item $\angle_v(l, l') = 0$ if and only if $l = l'$, and
\item $\angle_v(l, l') + \angle_v(l', l'') = \angle_v(l, l'')$ whenever three edges $l,l',$ and $l''$ are adjacent at a vertex $v$.
\end{itemize}
A finite tree with an angle function is called and {\em angled tree}.
A {\em local degree function} on $T$ is a function $\delta(v) \in \Z_{\geq 1}$ that assigns a positive integer to each vertex.
A vertex is {\em critical} if $\delta(v) > 1$, and {\em non-critical} otherwise.

An {\it abstract Hubbard tree} is a finite angled tree $T$ with a local degree function and a continuous map $f: T \longrightarrow T$ so that
\begin{itemize}
\item $f(V) \subseteq V$,
\item $f$ sends an edge to a finite union of edges,
\item $f$ is angle preserving, i.e., $\angle_{f(v)}(f(l), f(l')) = \delta(v) \angle_v(l, l')$,
\item $f$ is expanding, i.e., if two Julia vertices $v, w \in T$ are endpoints of an edge, then $f^n(v)$ and $f^n(w)$ are not endpoints of an edge for some $n>0$, and
\item $T$ is minimal, i.e., $T$ is the minimal forward invariant tree that contains all points $s$ with $\delta(s)\geq 2$.
\end{itemize}
The {\em degree of $f: T \longrightarrow T$}, denoted by $\deg(f)$, is defined by
\begin{equation}\label{eqn:DegTreeMap}
	\deg(f) =\sum_{\delta(v)>1} \left(\delta(v)-1\right)
\end{equation}
\end{defn}

More generally, we also consider dynamical systems defined by families of polynomials.
\begin{defn}[Families of polynomials, $\infty$-markings]\label{defn:FamilyPolynomialandMarking}
Let $\mathcal{S}=(|\mathcal{S}|, \Phi, \delta)$ be a mapping scheme as in Definition \ref{defn:bms}.
A {\it family of polynomials over $\mathcal{S}$} is a map 
\[
	\Fcal:\bigcup_{s\in|\mathcal{S}|}\C(s)\to \bigcup_{s\in|\mathcal{S}|}\C(s)
\]
satisfying
	\begin{itemize}
		\item $\C(s)$ is a copy of the complex plane, and
		\item $\Fcal_s:=\Fcal|_{\C(s)}$ is a polynomial from $\C(s)$ to $\C(\Phi(s))$ of degree $\delta(s)$.
	\end{itemize}


The map $\Fcal_s$ extends to the boundaries at infinity $\partial_\infty \C(s) \to \partial_\infty \C(\Phi(S))$. An {\em $\infty$-marking of $\Fcal$} is a collection of homeomorhpisms
\[
	\left\{ \phi_s:\mathbb{S}^1 \to \partial_\infty \C(s) \right\}_{s\in|\mathcal{S}|}
\]
satisfying $\phi_{\Fcal(s)} \circ z^{\delta(s)}= \Fcal_s \circ \phi_s$. 
We will call a family of polynomials $\Fcal$ {\em marked} if an $\infty$-marking is chosen.
\end{defn}

Critical points of $\Fcal$ are defined to be points where the derivatives vanish. The post-critical set $P_\Fcal$ is the union of forward orbits of critical points. We say that $\Fcal$ is {\it \pcf}if the post-critical set $P_\Fcal$ is finite. Objects associated to complex polynomials, such as Julia sets or external rays, can be naturally extended to families of polynomials. For $s\in |\mathcal{S}|$, we denote by $\Julia_\Fcal(s)$ (resp.\@ $\mathcal{K}_\Fcal(s)$) the Julia set (resp.\@ filled Julia set) in the plane $\C(s)$.

Similar as \pcf polynomials, for a \pcf family of polynomials $\Fcal$ over $\mathcal{S}$, we define the {\em Hubbard forest $H$} as the union of regulated hull $H(s)$ of $P_\Fcal \cap \C(s)$ in $\mathcal{K}_\Fcal(s)$ where the union is taken over $s\in |\mathcal{S}|$ (see \cite{Poi13}).
That is, the Hubbard forest $H$ is an invariant set consisting of a finite union of finite trees, which gives a combinatorial description of the dynamics of $\Fcal$.

\begin{rmk}
Instead of taking the regulated hull of the post-critical set, we can also take the regulated hulls generated by any finite $\Fcal$-invariant set $S$ containing $P_\Fcal$.
The same construction gives an invariant set $H_S$ consisting of a finite union of finite trees. The Hubbard forest $H$ is minimal in a sense that any such invariant forest $H_S$ contains $H$.
We will use this construction later when we define simplicial tunings.
\end{rmk}


We define abstract Hubbard forests following \cite{Poi13}.
\begin{defn}[Abstract Hubbard forests]\label{defn:aht}
Let $\mathcal{S}=(|\mathcal{S}|, \Phi, \delta)$ be a mapping scheme. An {\it abstract Hubbard forest} is a collection of angled trees
$$
	H = \bigcup_{s\in |\mathcal{S}|} H(s),
$$
with an angle preserving map $\Fcal: H \longrightarrow H$, where each $\Fcal_s:=\Fcal|_{H(s)}=H(s) \to H(\Phi(s))$ is of degree $\delta(s)$. 
The degree $\deg(\Fcal_s)$ is defined as the sum of $(\delta(v)-1)$ as the equation \eqref{eqn:DegTreeMap}.
The angle structures of Hubbard forests are defined in the same way as the angle structures of Hubbard trees in Definition \ref{defn:aht1}.
We additionally require that 
\begin{itemize}
\item $\Fcal$ is expanding, i.e., if two Julia vertices $v, w \in H$ are endpoints of an edge, then $\Fcal^n(v)$ and $\Fcal^n(w)$ are not endpoints of one edge for some $n>0$, and
\item $H$ is minimal, i.e., $H$ is the minimal forward invariant forest that contains all points $v$ with $\delta(v)\geq 2$.
\end{itemize}
\end{defn}
We remark that an abstract Hubbard tree is an abstract Hubbard forest with one connected component.

Abstract Hubbard trees and abstract Hubbard forests give combinatorial models for \pcf or \pcf families of polynomials. More precisely, by \cite[Theorem 1.1]{Poi10} and \cite[Theorem 1.2]{Poi13}, we have the following theorem.
\begin{theorem}\label{thm:ra}
An abstract Hubbard tree (resp.\@ forest) is the Hubbard tree (resp. forest) of a post-critically finite polynomial (resp. a family of post-critically finite polynomials). Moreover, the polynomial (resp. a family of polynomials) is unique up to affine conjugacy.
\end{theorem}



There is a notion of markings of abstract Hubbard trees that emulate parameterizations of external rays landing at points on polynomial Hubbard trees \cite[\S 9]{Poi10}. It also has ($d-1$)-choices of ambiguity. The use of marking allows us to have the one-to-one correspondence in the following theorem.
\begin{theorem}\label{thm:ram}
A marked abstract Hubbard tree (resp.\@ forest) is the Hubbard tree (resp.\@ forest) of a unique marked \pcf polynomial $f$ (resp.\@ marked \pcf family of polynomials) up to affine conjugacy.
\end{theorem}

\subsection{Simplicial tuning}\label{subsec:sthf}
{\em Tuning} is an operation introduced by Douady and Hubbard in \cite{DH85} for quadratic polynomials.
The concept can be easily generalized to higher degree polynomials or even rational maps (see \cite[\S 3]{Pil94} and \cite{IK12}).

\begin{defn}[Simplicial Hubbard trees and forests]
Let $f: T \longrightarrow T$ be a map on a finite tree $T$ and $\Fcal: H \longrightarrow H$ be a map on a finite union of finite trees $H = \bigcup T_i$.
We say $f$ (resp.\@ $\Fcal$) is {\em simplicial} if there exists a finite simplicial structure on $T$ (resp.\@ $H$) so that $f$ (resp.\@ $\Fcal$) sends an edge of $T$ (resp.\@ $H$) to an edge of $T$ (resp.\@ $H$).

Let $f$ be a \pcf polynomial.
We say $f$ has a {\em simplicial Hubbard tree} if $f: \mathcal{T}_f \longrightarrow \mathcal{T}_f$ is simplicial.

Let $\Fcal$ be a \pcf family of polynomials over $\mathcal{S}$.
We say $\Fcal$ has a {\em simplicial Hubbard forest} if $\Fcal: H_\Fcal \longrightarrow H_\Fcal$ is simplicial.
\end{defn}


\subsection*{Tuning}
We first define tunings for abstract Hubbard trees in a combinatorial way. Then, by Theorem \ref{thm:ra}, we define the tuning of polynomials as the polynomials corresponding to the tuned abstract Hubbard trees. Figure \ref{fig:Tuning} describes how we tune Hubbard trees.

\begin{figure}[h!]
	\centering
	\begin{subfigure}[b]{0.3\textwidth}
		\centering
		\def\svgwidth{0.8\textwidth}
\begingroup%
  \makeatletter%
  \providecommand\color[2][]{%
    \errmessage{(Inkscape) Color is used for the text in Inkscape, but the package 'color.sty' is not loaded}%
    \renewcommand\color[2][]{}%
  }%
  \providecommand\transparent[1]{%
    \errmessage{(Inkscape) Transparency is used (non-zero) for the text in Inkscape, but the package 'transparent.sty' is not loaded}%
    \renewcommand\transparent[1]{}%
  }%
  \providecommand\rotatebox[2]{#2}%
  \newcommand*\fsize{\dimexpr\f@size pt\relax}%
  \newcommand*\lineheight[1]{\fontsize{\fsize}{#1\fsize}\selectfont}%
  \ifx\svgwidth\undefined%
    \setlength{\unitlength}{393.75bp}%
    \ifx\svgscale\undefined%
      \relax%
    \else%
      \setlength{\unitlength}{\unitlength * \real{\svgscale}}%
    \fi%
  \else%
    \setlength{\unitlength}{\svgwidth}%
  \fi%
  \global\let\svgwidth\undefined%
  \global\let\svgscale\undefined%
  \makeatother%
  \begin{picture}(1,1.01333332)%
    \lineheight{1}%
    \setlength\tabcolsep{0pt}%
    \put(0,0){\includegraphics[width=\unitlength,page=1]{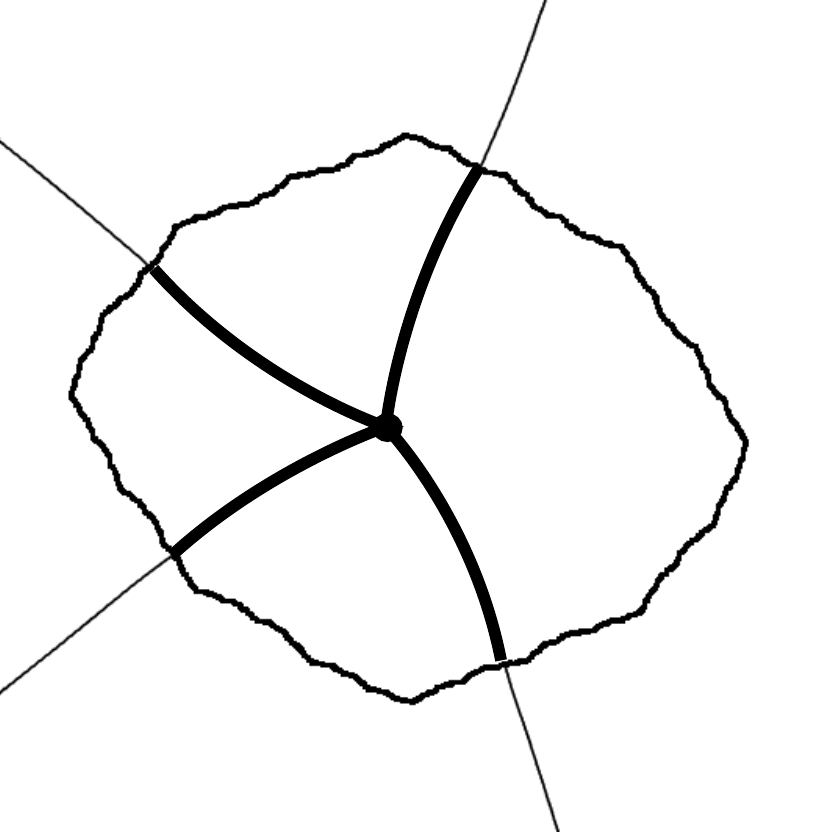}}%
    \put(0.49154543,0.48917405){\color[rgb]{0,0,0}\makebox(0,0)[lt]{\lineheight{1.25}\smash{\begin{tabular}[t]{l}$v$\end{tabular}}}}%
  \end{picture}%
\endgroup%

		\caption*{$\mathcal{T}_v$}
	\end{subfigure}
	\begin{subfigure}[b]{0.3\textwidth}
		\centering
		\def\svgwidth{0.8\textwidth}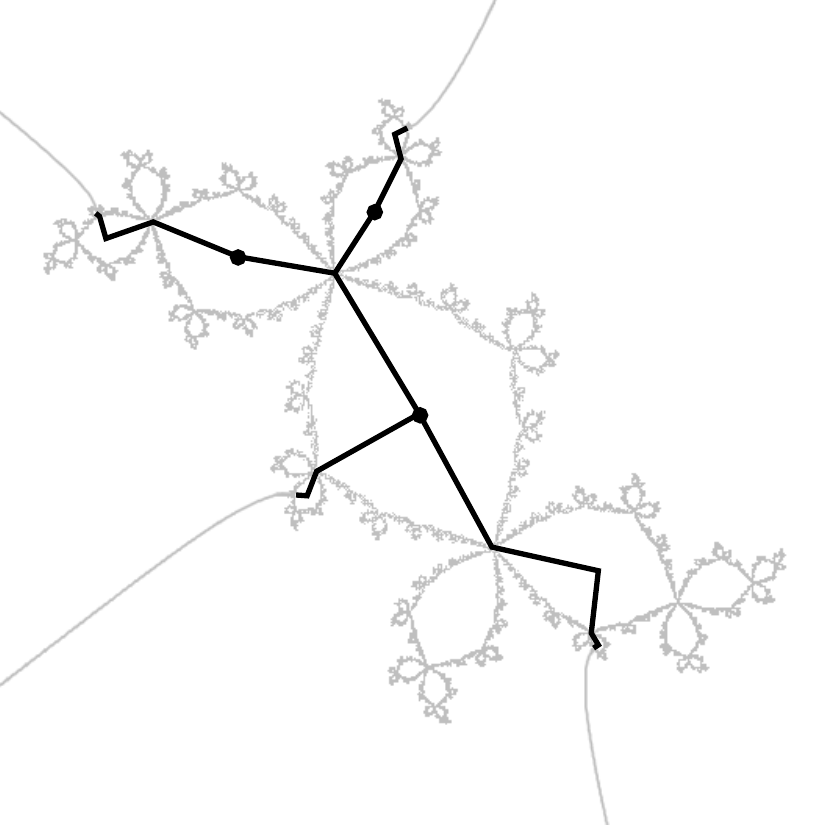
		\caption*{$\mathcal{T}_{v,new}$}
	\end{subfigure}
	\begin{subfigure}[b]{0.3\textwidth}
		\centering
		\def\svgwidth{0.8\textwidth}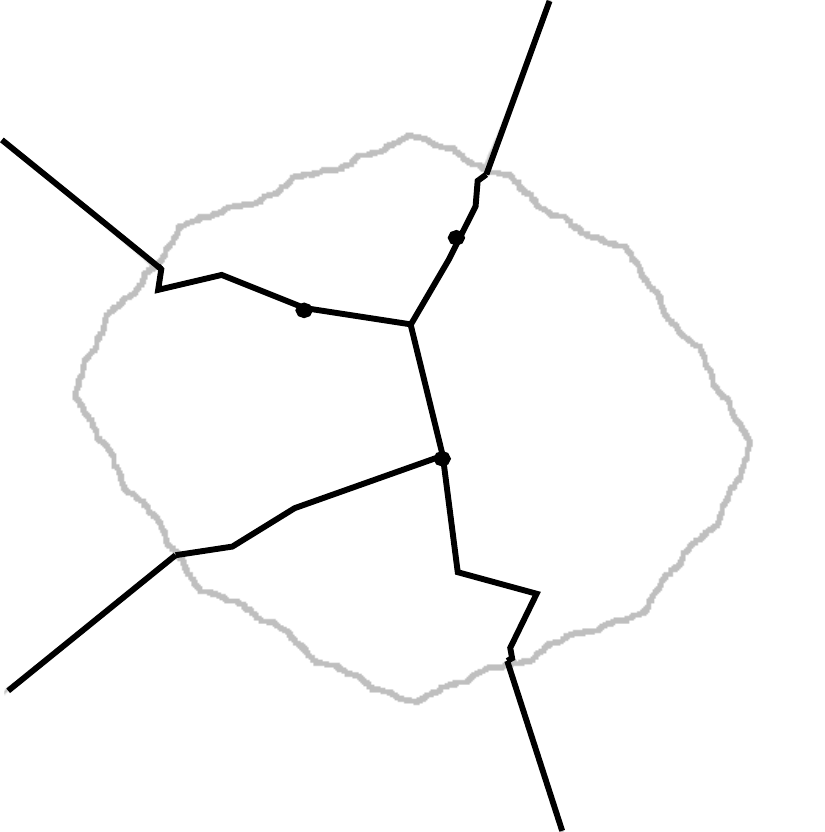
		\caption*{$\mathcal{T}_f$}
	\end{subfigure}
	\caption{Local picture of tuning a \pcf polynomial $g$ with a \pcf family of polynomials $\Fcal$ whose first return map is the Rabbit polynomial, where $\phi_v^{-1}(\partial \mathcal{T}_v)=\{e^{2\pi i \cdot 1/5},e^{2\pi i \cdot 2/5},e^{2\pi i \cdot 3/5},e^{2\pi i \cdot 4/5}\}$}
	\label{fig:Tuning}
\end{figure}

Suppose $g$ is a post-critically finite polynomial having at least one periodic critical point. Let $\mathcal{S}=(|\mathcal{S}|, \Phi, \delta)$ be a mapping scheme associated to $g$. Assume that we have a boundary marking $\phi$ for $g$ (see Definition \ref{defn:BdryMarking}). Suppose that $\Fcal$ is a \pcf family of polynomials over $\mathcal{S}$ and an $\infty$-marking of $\Fcal$ is given.

Let $\mathcal{V}^F$ be the set of branch Fatou points, critical and post-critical Fatou points on $\mathcal{T}_g$. The superscript F stands for Fatou, not the map $\Fcal$.

First let us suppose that $v \in \mathcal{V}^F$ is critical or post-critical.
Let $\mathcal{T}_v = U_v \cap \mathcal{T}_g$, where $U_v$ is the Fatou component of $g$ containing $v$.
Note that the boundary marking determines a map 
\[
	\phi_{v}: \mathbb{S}^1 \longrightarrow \partial U_v
\]
which sends $e^{i\cdot 0} \in \mathbb{S}^1$ to the marked point.

Recall that the $\infty$-marking of $\Fcal$ uniquely determines the external angles in the complex plane $\C(v)$ (see  Definition \ref{defn:FamilyPolynomialandMarking}). 
We define $\mathcal{T}_{v, new}$ by the regulated hull generated by the union of the Hubbard tree $H(v)$ and the landing points of external rays of angles in $\phi_v^{-1}(\partial \mathcal{T}_v)$. Recall that the boundary $\partial \mathcal{T}_{v}$ is the set of endpoints of the tree $\mathcal{T}_{v}$. Then we define
\[
	\Phi_v:\partial T_v \to T_{v,new}
\]
in such a way that for any $x\in \partial T_v$, $\Phi_v(x)$ is the landing point of the external ray of angle $\phi_v^{-1}(x)$ in the complex plane $\C(v)$.

Now suppose that $v \in \mathcal{V}^F$ is not critical or post-critical.
Then $v$ is a branch Fatou point, and $v$ has local degree $1$.
Let $l$ be the smallest integer so that $g^l(v)$ is a critical point. 
Then $g^l: U_v \longrightarrow U_{g^l(v)}$ is a biholomorphism.
It follows from $g^l(v)$ being a critical Fatou point that there is a boundary marking $\phi_{g^l(v)}: \mathbb{S}^1 \longrightarrow \partial U_{g^l(v)}$.
Since $g^l: U_v \longrightarrow U_{g^l(v)}$ has degree $1$, we have the inverse $g^{-l}: U_{g^l(v)} \longrightarrow U_v$, which extends to the boundaries. Then we define a marking $\phi_v: \mathbb{S}^1 \to \partial U_v$ for $v$ by $\phi_v= g^{-l} \circ \phi_{g^l(v)}$, i.e.,
\begin{align*}
\phi_v: \mathbb{S}^1 &\longrightarrow \partial U_v\\
x &\mapsto  g^{-l} \circ \phi_{g^l(v)}(x).
\end{align*}

Since $v \notin |S_g|$, we construct $\mathcal{T}_{v,new}$ using the complex plane $\C(g^l(v))$. More precisely, for $\mathcal{T}_v:=U_v \cap \mathcal{T}_g$, we define $\mathcal{T}_{v,new}$ as the regulated hull generated by the landing points of rays of angles in $\phi_v^{-1}(\partial \mathcal{T}_v)$ in the complex plane $\C(g^l(v))$. The map $\Phi_v : \partial T_v \longhookrightarrow T_{v, new}$ is again defined as the landing points of the external rays of corresponding angles.

Define a tree
$$
	\mathcal{T}_f = \left( \mathcal{T}_g - \bigcup_{v\in\mathcal{V}^F} \Int(\mathcal{T}_v) \right) \coprod_{\{\Phi_v|v\in\mathcal{V}^F \}} \left(\bigcup_{v\in\mathcal{V}^F} \mathcal{T}_{v,new}\right)
$$ 
by removing $\mathcal{T}_v$'s and gluing $\mathcal{T}_{v,new}$'s using the identifications $\Phi_v : \partial T_v \longhookrightarrow T_{v, new}$. See Figure \ref{fig:Tuning}. It is easy to verify the following:
\begin{itemize}
	\item the dynamics of $g$ and $\Fcal$ induce a dynamics on $\mathcal{T}_f$, which we denote by $f$. More precisely 
	$$
	f(x) = \begin{cases}
	g(x) & \text{for } x \in \mathcal{T}_g - \bigcup_{v\in\mathcal{V}^F} \Int(\mathcal{T}_v)\\
	\Fcal(x) & \text{for } x\in \bigcup_{v\in\mathcal{V}^F} \mathcal{T}_{v,new}
	\end{cases},
	$$
	\item the angle structures of $\mathcal{T}_g$ and $H_\mathcal{F}$ induce the angle structure on $\mathcal{T}_f$ and $f$ is angle preserving, and
	\item $f$ is expanding, in the sense of Definition \ref{defn:aht}.
\end{itemize}
Lastly, we replace $\mathcal{T}_f$ by the minimal $f$-invariant subtree, which could be $\mathcal{T}_f$, that contains all the critical points. By Theorem \ref{thm:ra} we may assume $f$ is a \pcf polynomial and $\mathcal{T}_f$ is the Hubbard tree of $f$. 

\begin{defn}[Tuning, Simplicial tuning]
We call $f$ in the above construction the {\em tuning of $g$ with $\Fcal$}. 
If $\Fcal$ is simplicial, then we say $f$ is a {\em simplicial tuning} of $g$.
\end{defn}

\begin{defn}[Internal and external edges]\label{defn:IntExtEdges}
Let $f$ be a tuning of $g$ with $\Fcal$.
An edge $E$ of $\mathcal{T}_g$ (resp.\@ of $\mathcal{T}_f$) is called {\em internal} if $E \subseteq \mathcal{T}_v$ (resp.\@ $E \subseteq \mathcal{T}_{v,new}$) for some $v \in \mathcal{V}^F$, and {\em external} otherwise.
\end{defn}

By construction, external edges of $\mathcal{T}_g$ and $\mathcal{T}_f$ are in bijective correspondence with each other.

\begin{rmk}
In many cases, the Hubbard tree $\mathcal{T}_f$ can be simply constructed from $\mathcal{T}_g$ by replacing each branch Fatou point, or critical and post-critical Fatou point by a tree. However this is not always the case because some additional identification may be required (see Figure \ref{fig:TuneByBasilicas}).

\begin{figure}[h!]
	\centering
	\def\svgwidth{0.5\textwidth}
\begingroup%
  \makeatletter%
  \providecommand\color[2][]{%
    \errmessage{(Inkscape) Color is used for the text in Inkscape, but the package 'color.sty' is not loaded}%
    \renewcommand\color[2][]{}%
  }%
  \providecommand\transparent[1]{%
    \errmessage{(Inkscape) Transparency is used (non-zero) for the text in Inkscape, but the package 'transparent.sty' is not loaded}%
    \renewcommand\transparent[1]{}%
  }%
  \providecommand\rotatebox[2]{#2}%
  \newcommand*\fsize{\dimexpr\f@size pt\relax}%
  \newcommand*\lineheight[1]{\fontsize{\fsize}{#1\fsize}\selectfont}%
  \ifx\svgwidth\undefined%
    \setlength{\unitlength}{206.6150382bp}%
    \ifx\svgscale\undefined%
      \relax%
    \else%
      \setlength{\unitlength}{\unitlength * \real{\svgscale}}%
    \fi%
  \else%
    \setlength{\unitlength}{\svgwidth}%
  \fi%
  \global\let\svgwidth\undefined%
  \global\let\svgscale\undefined%
  \makeatother%
  \begin{picture}(1,0.38676585)%
    \lineheight{1}%
    \setlength\tabcolsep{0pt}%
    \put(0,0){\includegraphics[width=\unitlength,page=1]{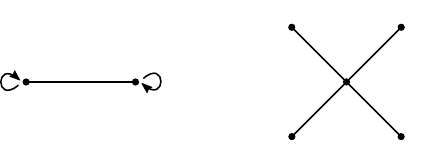}}%
    \put(-0.00391789,0.11926576){\color[rgb]{0,0,0}\makebox(0,0)[lt]{\lineheight{1.25}\smash{\begin{tabular}[t]{l}$3$\end{tabular}}}}%
    \put(0.32484144,0.11927034){\color[rgb]{0,0,0}\makebox(0,0)[lt]{\lineheight{1.25}\smash{\begin{tabular}[t]{l}$3$\end{tabular}}}}%
    \put(0,0){\includegraphics[width=\unitlength,page=2]{TuneByBasilicas.pdf}}%
    \put(0.65075808,0.01037222){\color[rgb]{0,0,0}\makebox(0,0)[lt]{\lineheight{1.25}\smash{\begin{tabular}[t]{l}$2$\end{tabular}}}}%
    \put(0.90485542,0.01037903){\color[rgb]{0,0,0}\makebox(0,0)[lt]{\lineheight{1.25}\smash{\begin{tabular}[t]{l}$2$\end{tabular}}}}%
    \put(0.90485876,0.34148133){\color[rgb]{0,0,0}\makebox(0,0)[lt]{\lineheight{1.25}\smash{\begin{tabular}[t]{l}$2$\end{tabular}}}}%
    \put(0.65076083,0.34148437){\color[rgb]{0,0,0}\makebox(0,0)[lt]{\lineheight{1.25}\smash{\begin{tabular}[t]{l}$2$\end{tabular}}}}%
  \end{picture}%
\endgroup%

	\caption{The quartic polynomial $f$ for the right Hubbard tree is obtained by replacing each of the cubic fixed points of the quartic polynomial $g$ for the left Hubbard tree by the bitransitive cubic Basilica. }
	\label{fig:TuneByBasilicas}
\end{figure}
\end{rmk}

We remark that if $\mathcal{T}_f$ is simplicial, then any edge of $\mathcal{T}_f$ has depth $0$ in the sense of Definition \ref{defn:dd}.

	

\begin{lem}\label{lem:0or1}
	Let $f$ be a \pcf polynomial with a simplicial Hubbard tree $\mathcal{T}_f$. Then the following hold.
	\begin{enumerate}
		\item Every biaccessible point in the Julia set $\Julia_f$ is on the boundary of a bounded Fatou component.
		\item Let $M$ be a finite $f$-invariant subset of the Julia set and $\mathcal{T}'$ be the regulated hull of $\mathcal{T}_f \cup M$. Then any edge of $\mathcal{T}'$ has depth 0 or 1. Moreover, an edge $e$ of $\mathcal{T}'$ has depth $1$ if and only if the boundary of $e$ contains an accumulation point, i.e., a point that is not on the boundary of a bounded Fatou component.
	\end{enumerate}
\end{lem}
\begin{proof}
	
	(1) By the expanding property of Hubbard trees, every edge of $\mathcal{T}_f$ is either an internal ray or the union of two internal rays of bounded Fatou components. This is also true for $f^{-n}(\mathcal{T}_f)$ for any $n>0$. Then for any two bounded Fatou components $U$ and $V$, there is a chain of bounded Fatou components $U_0=U, U_1,U_2,\dots,U_k=V$ so that every consecutive pair $U_i$ and $U_{i+1}$ is adjacent at a pre-periodic biaccessible point.
	
	Suppose $v \in \Julia_f$ is a biaccessible point. Let $D_1,D_2,\dots,D_k$ be connected components of the complement of external rays landing at $v$. Note that every $D_i$ contains bounded Fatou components. Then the existence of an adjacent chain in the previous paragraph implies that $v$ is also on the boundary of a bounded Fatou component. 
	
	(2) Note that any edge of $\mathcal{T}_f$ has depth $0$.
	Let $e = [v,w]$ be an edge of $\mathcal{T}' - \mathcal{T}_f$.
	
	Suppose $e=[v, w]$ does not contain any accumulation points. Then $e$ is a finite union of internal rays of bounded Fatou components, which are eventually periodic. So $e$ has depth $0$.
	
	Suppose $e$ contains an accumulation point. By (1), only endpoints of $e$ can be accumulation points. Suppose $v$ is an accumulation point. (1) also implies that $v$ is an endpoint of $\mathcal{T}'$. Then $w$ is not an endpoint of $\mathcal{T}'$, so $w$ is not an accumulation point.
	Without loss of generality, we assume that $v$ is a periodic point.
	Note that there exists a sequence $v_n \in [v,w]$ where $v_n$ is a boundary point of a bounded Fatou component with $v_n \to v$.
	Thus $[v_n, w]$ is eventually mapped into $\mathcal{T}_f$ for all $n$.
	Therefore, every pre-periodic point in $e$ other than $v$ is eventually mapped into $\mathcal{T}_f$, so $e$ has depth $1$.
\end{proof}

	We also remark that by Lemma \ref{lem:0or1}, even if $\mathcal{F}$ is simplicial, the tree $\mathcal{T}_{v,new}$ we defined above may not be simplicial. We may obtain edges of depth one when we enlarge the simplicial tree $H(v)$ by taking the regulated hull of $H(v)$ union the landing points of external rays of angles in $\phi_v^{-1}(\partial\mathcal{T}_v)$.

\subsection*{Core entropy of simplicial tuning}
\begin{lem}\label{lem:fbgr}
Let $g$ be a \pcf polynomial with core entropy zero.
Let $f$ be a simplicial tuning of $g$.
Then $f$ has core entropy zero. Moreover, we have $\mathscr{C}_{gr}(g) \leq \mathscr{C}_{gr}(f) \leq \mathscr{C}_{gr}(g)+1$.
\end{lem}
\begin{proof}
Let $\mathcal{G}_f$ and $\mathcal{G}_g$ be the corresponding directed graphs for $f$ and $g$ respectively.
Since $g$ has core entropy zero, $\mathcal{G}_g$ has no intersecting cycles.
By the construction of simplicial tuning, new edges are mapped to new edges. The dynamics among new edges is equivalent to the dynamics on a Hubbard tree of $g$, which has zero entropy. Hence $\mathcal{G}_f$ also has no intersecting cycles, so $f$ has core entropy zero by Theorem \ref{thm:nic}.

Let $\mathcal{G}_f^{ext}$ (resp.\@ $\mathcal{G}_f^{int}$) be the subgraph of $\mathcal{G}_f$ consisting of external (resp.\@ internal) edges of $\mathcal{T}_f$. Subgraphs $\mathcal{G}_g^{ext}$ and $\mathcal{G}_g^{int}$ of $\mathcal{G}_g$ are similarly defined. 

From the bijective correspondence between external edges of $\mathcal{T}_f$ and external edges of $\mathcal{T}_g$, we know that two subgraphs $\mathcal{G}_f^{ext}$ and $\mathcal{G}_g^{ext}$ are isomorphic as directed graphs.
On the other hand, any vertex in $\mathcal{G}_g^{int}$ has depth 0, and any vertex in $\mathcal{G}_f^{int}$ has depth 0 or 1 by Lemma \ref{lem:0or1}. Hence we have
\[
\mathscr{C}_{gr}(f)=\mathscr{C}_{gr}(g)~\mathrm{or}~ \mathscr{C}_{gr}(g)+1.
\]
\end{proof}

As a corollary, we have one implication of Theorem \ref{thm:fb} and one inequality of Theorem \ref{thm:cgrb}.
\begin{cor}\label{cor:fcz}
Let $f \in \mathcal{P}_d$ be a \pcf polynomial.
Suppose that $f$ is obtained by a finite sequence of simplicial tunings from $z^d$.
Then $f$ has core entropy zero.
Moreover, $\mathscr{C}_{c}(f) \geq \mathscr{C}_{gr}(f)$.
\end{cor}

\subsection{The simplicial quotient}\label{subsec:simp}
In this subsection, we will discuss simplicial quotients, which can be thought of as an inverse operation of simplicial tunings.
Our goal is to prove: 
\begin{theorem}\label{prop:fb}
Let $f$ be a \pcf polynomial with core entropy zero.
Then there exists a finite sequence $f_0 = f, f_1,..., f_k$ such that 
\begin{itemize}
\item $f_k(z) = z^d$, and
\item $f_{i+1}$ is the simplicial quotient of $f_i$.
\end{itemize}
\end{theorem}
Simplicial quotients will be defined in Definition \ref{defn:sq}.

\begin{figure}[h!]
	\centering
	\begin{subfigure}[b]{0.45\textwidth}
		\centering
		\includegraphics[width=\textwidth]{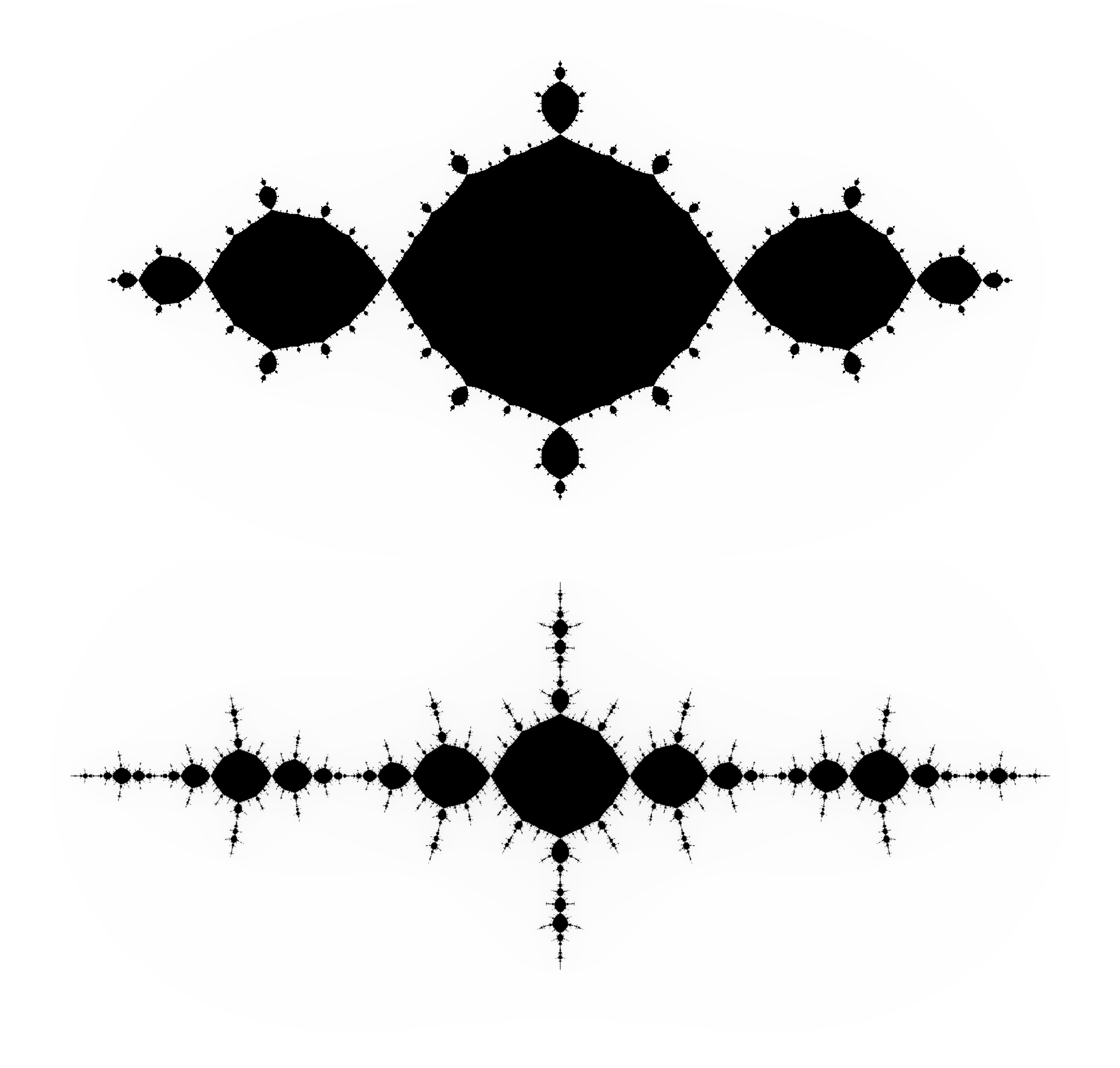}
		\caption{}
	\end{subfigure}
	\hspace{5pt}
	\begin{subfigure}[b]{0.45\textwidth}
		\centering
		\includegraphics[width=\textwidth]{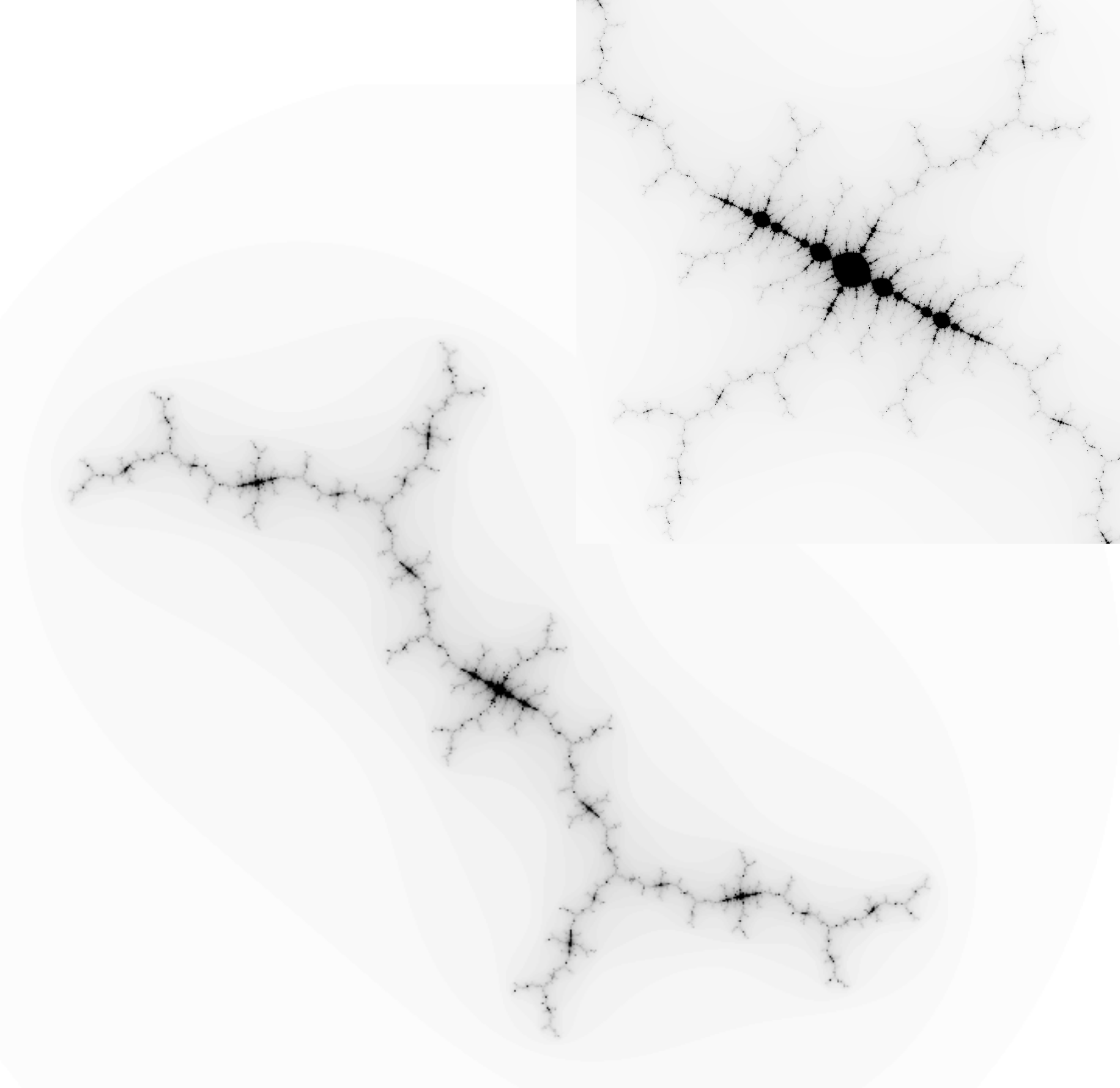}
		\caption{}
	\end{subfigure}
	\newline
	\begin{subfigure}[b]{0.45\textwidth}
		\centering
		\includegraphics[width=\textwidth]{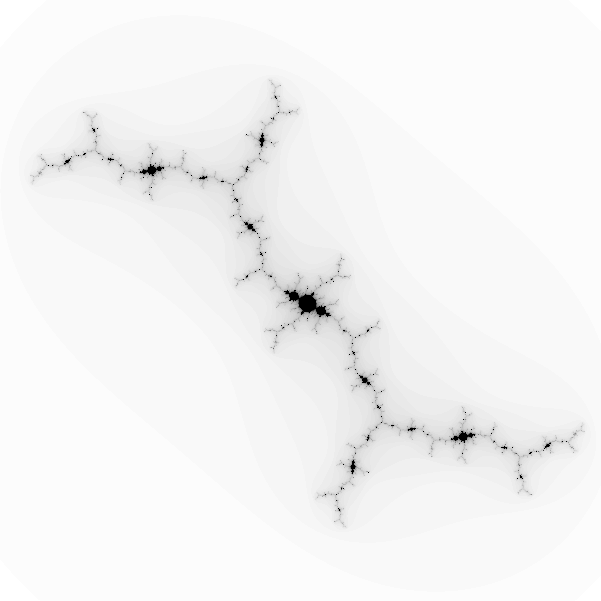}
		\caption{}
	\end{subfigure}
	\hspace{5pt}
	\begin{subfigure}[b]{0.45\textwidth}
		\centering
		\includegraphics[width=\textwidth]{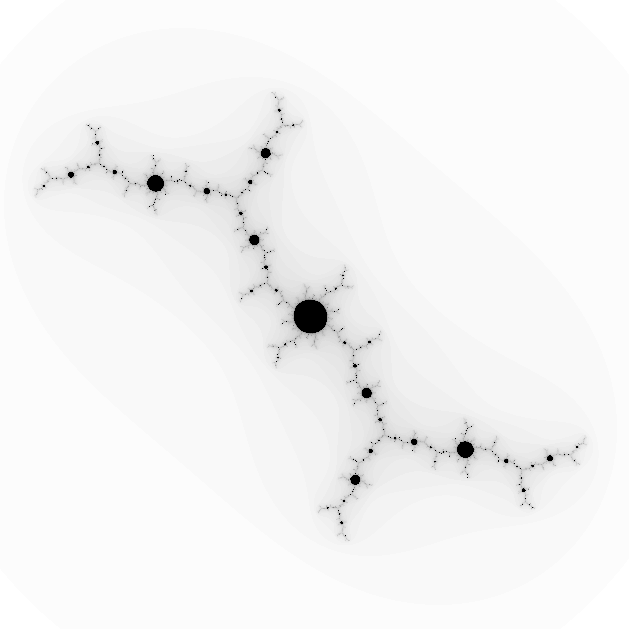}
		\caption{}
	\end{subfigure}
	\caption{(A) Basilica (above, $z^2-1$) and the doubled Basilica(below, $z^2-1.3107$). (B) Kokopelli tuned with the doubled Basilica ($z^2-0.15652 + 1.03225 i$) and the zoomed-in image near 0. (C) Kokopelli tuned with Basilica ($z^2-0.16267+1.037123$). (D) Kokopelli ($z^2-0.15652 + 1.03224$).\protect\linebreak
	Cores of (B) and (C) are both homeomorphic to the Julia set of the Basilicas. (C) is the simplicial quotient of (B), and (D) is the simplicial quotient of (C).}
\end{figure}

\subsection*{Cores of periodic Fatou components}
Let $f$ be a \pcf polynomial and $\mathcal{K}^0$ be the closure of the union of periodic Fatou components. Let
$\mathcal{K}^0_1,\mathcal{K}^0_2,\dots,\mathcal{K}^0_l$ be the connected components of $\mathcal{K}^0$. We inductively define $\mathcal{K}^n_i$ to be the connected component of $f^{-1}(\mathcal{K}^{n-1})$ containing $\mathcal{K}^{n-1}_i$ and define $\mathcal{K}^n$ to be the union of $\mathcal{K}^n_1,\mathcal{K}^n_2,\dots,\mathcal{K}^n_l$. It is also inductively obtained that $\mathcal{K}^n_1,\mathcal{K}^n_2,\dots,\mathcal{K}^n_l$ are pairwise disjoint for any $n\ge0$.

By construction, for any $i\in\{1,2,\dots,l\},$ $\{\mathcal{K}^n_i\}_{n\ge0}$ is an increasing sequence. 
\begin{defn}[Cores]
We use the notations defined above. For any $i\in \{1,2,\dots,l\}$, we call $\mathcal{K}^\circ_i := \bigcup_n \mathcal{K}^n_i$ a {\em periodic pre-core} and define a {\em periodic core} by
$$
\mathcal{K}_i:= \overline{\mathcal{K}^\circ_i}.
$$ 
We define the {\em accumulation set} of a core $\mathcal{K}_i$ by
$$
\mathcal{K}^\omega_i = \mathcal{K}_i - \mathcal{K}^\circ_i.
$$

For any periodic core $\mathcal{K}_i$ and any $n\ge1$, we call a connected component $\mathcal{K}^\circ$ of $f^{-n}(\mathcal{K}_i^\circ)$ that is not a periodic pre-core a {\em pre-periodic pre-core} of $f$. Then we call $\mathcal{K}:=\overline{\mathcal{K}^\circ}$ a pre-periodic core and  $\mathcal{K}^\omega:= \mathcal{K} - \mathcal{K}^\omega$ the accumulation set of $\mathcal{K}$.

By a {\em core}, we mean a periodic or pre-periodic core.
\end{defn}

The following proposition is immediate from the construction.
\begin{prop}\label{prop:sqc}
	Let $f$ be a \pcf polynomial. Then two bounded Fatou components $U$ and $V$ are in the same core if and only if there is a finite sequence of bounded Fatou components $U=U_0,U_1,\dots,U_k=V$ so that $\overline{U}_{i-1} \cap \overline{U}_{i}\neq \emptyset$ for any $i\in\{1,\dots,k\}$.
\end{prop}

In the above notations, there are $l$ periodic cores $\mathcal{K}_1, \mathcal{K}_2,\dots,\mathcal{K}_l$ such that for $i \neq j$, $\mathcal{K}^\circ_i$ is disjoint from $\mathcal{K}^\circ_j$ (though $\mathcal{K}_i$ may intersect $\mathcal{K}_j$).
Thus, for every Fatou component $U$, there exists a unique core $\mathcal{K}$ containing $U$.

Let $\mathcal{K}$ be a core.
We define the degree of $\mathcal{K}$ by
$$
\deg(\mathcal{K}) := 1+ \sum_{c \in \mathcal{K}^{\circ}} (\deg(f|_c) -1).
$$

One can show that $\mathcal{K}$ is locally connected. Puzzle pieces defined by equipotential curves and eventually periodic external rays shrink, and the intersection of each puzzle piece with $\mathcal{K}$ is connected. We omit details and refer the reader to \cite{OvS09}.

Hence there is a Riemann mapping $\phi:\hCbb\setminus \D \to \hCbb\setminus \mathcal{K}$, which continuously extends to the boundary $\phi:\partial \D \to \partial \mathcal{K}$ by the Carath\'{e}odory theorem \cite[\S 17]{MilnorBook}. This domain circle $\partial \D$ is called the {\em ideal boundary} of $\C \setminus \mathcal{K}$ and we denote it by $I(\C\setminus \mathcal{K})$.

The map $f: \mathcal{K} \longrightarrow f(\mathcal{K})$ induces a degree $\deg(\mathcal{K})$ covering map
$$
f_*:I(\C \setminus \mathcal{K}) \longrightarrow I(\C \setminus f(\mathcal{K})).
$$

We also remark that $\mathcal{K}$ may not be a connected component of $f^{-1}(f(\mathcal{K}))$.
This happens if there exists a critical point of $f$ on the accumulation set
$\mathcal{K}^\omega$.
By construction, if $\mathcal{K}, \mathcal{K}'$ are two distinct cores, then $\mathcal{K}^\circ$ and ${\mathcal{K}'}^\circ$ are disjoint.


\subsection*{Construction of simplicial quotients}
For a \pcf polynomial $f$, the {\em simplicial quotient} of $f$ is roughly speaking a \pcf polynomial $g$ obtained by replacing each core $\mathcal{K}$ of $f$ with the closure of a new Fatou component $\overline{U}$ of degree $\deg(\mathcal{K})$.
A precise construction can be achieved by constructing a desired Hubbard tree in a combinatorial way and then applying Theorem \ref{thm:ra} as follows.

Let $\mathcal{T}_f$ be the Hubbard tree of $f$.
Let $v$ be a Fatou vertex of $\mathcal{T}_f$.
We define a {\em core subtree} $\mathcal{T}_v$ for $v$ by
$$
\mathcal{T}_v = \mathcal{T}_f \cap \mathcal{K}_v,
$$
where $\mathcal{K}_v$ is the core containing the Fatou component $U_v$ associated to $v$. Since $\mathcal{T}_f$ and $\mathcal{K}_f$ are convex with respect to regulated arcs, the intersection $\mathcal{T}_v$ is indeed a tree.
For different Fatou vertices $v$ and $w$, the core subtrees $\mathcal{T}_v$ and $\mathcal{T}_w$ are same, disjoint, or intersecting at a single point.

Let $\mathcal{T}'$ be the closure of a connected component $\mathcal{T}_f - \mathcal{T}_v$.
Then $\mathcal{T}'$ determines a point on the ideal boundary $x_{\mathcal{T}'} \in I(\C \setminus \mathcal{K}_v)$.
Similarly, if $a\in \partial \mathcal{T}_v \cap \partial \mathcal{T}_f$, then it determines a point on the ideal boundary $x_{a} \in I(\C \setminus \mathcal{K}_v)$ as well. Let $X_v \subseteq I(\C \setminus \mathcal{K}_v)$ be the collection of all these ideal boundary points,
and let $\nu_v = \#X_v$ be its cardinality.

A tree is called a {\em star} if only one vertex, called the center, has valence greater than one.
A {\em $k$-star} is a star whose center has valence $k$.
We construct a new tree by removing $\Int(\mathcal{T}_v)$ and glue back a $\nu_v$-star $\mathcal{T}_{v, new}$ using the ideal boundary information.

We obtain a new tree $\mathcal{T}_g$ by performing the above surgery at all Fatou vertices of $\mathcal{T}_f$. The map of $f$ on $\mathcal{T}_f$ naturally induces a map $g$ on $\mathcal{T}_g$.
Indeed,
\begin{itemize}
\item if $w \notin \bigcup \mathcal{T}_{v, new}$, then $g(w) = f(w)$,
\item if $w$ is the center of $\mathcal{T}_{v, new}$, then we define $g(w)$ to be the center of $\mathcal{T}_{f(v), new}$, and
\item if $w\in \partial \mathcal{T}_{v,new}$, then $g(w) \in \partial \mathcal{T}_{f(v), new}$ is determined by the induced map between the ideal boundaries $f_*: I(\C \setminus \mathcal{K}_v) \longrightarrow I(\C \setminus \mathcal{K}_{f(v)})$.
\end{itemize}
The map $g$ is then extended on the edges $[w,w']$ so that 
$$
g([w,w']) =[g(w), g(w')].
$$ 
Inherited from the local degree on $\mathcal{T}_f$, we remark that the local degrees of a vertex of $\mathcal{T}_g$ can be naturally defined, and a vertex is called critical if it has local degree $\geq 2$.
By removing some endpoints and the corresponding edges of $\mathcal{T}_g$ if necessary, we may assume $\mathcal{T}_g$ is the convex hull of critical and post-critical points.
The ideal boundary also gives the angle structure on the new tree $\mathcal{T}_g$. More precisely, angles at the center of stars are defined from the ideal boundaries and angles at the other vertices are induced from angles of $\mathcal{T}_f$.
Therefore, $\mathcal{T}_g$ is an abstract Hubbard tree.
The marking for $\mathcal{T}_f$ induces a marking for $\mathcal{T}_g$.
Thus, by Theorem \ref{thm:ram}, there exists a unique \pcf polynomial $g \in \mathcal{P}_d$ realizing this abstract Hubbard tree $\mathcal{T}_g$. 

\begin{defn}[Simplicial quotients]\label{defn:sq}
Let $f$ be a \pcf polynomial with Hubbard tree $\mathcal{T}_f$.
Let $\mathcal{T}_g$ be the marked abstract Hubbard tree constructed as above.
We call the corresponding \pcf polynomial $g \in \mathcal{P}_d$ the {\em simplicial quotient} of $f$.
\end{defn}

It is clear from the construction that $f$ is a simplicial tuning of $g$. We have the following bijection
\[
	\{\mathrm{cores~of~} f\} \longleftrightarrow  \{\mathrm{bounded~Fatou~components~of~} g\}.
\]
This induces a natural surjection between the sets of critical and post-critical points 
\[
	q:\Crit(f) \cup P_f \rightarrow \Crit(g) \cup P_g
\]
which sends all the critical or post-critical points in $\mathcal{K}^\circ$ for each core $\mathcal{K}$ of $f$ to the center of the corresponding Fatou component of $g$. 
The following lemma is immediate from these correspondences.

\begin{lem}\label{lem:PostCritPtisAccPt}
	Let $g$ be the simplicial quotient of a \pcf polynomial $f$ and $q:\Crit(f) \cup P_f \rightarrow \Crit(g) \cup P_g$ be the induced surjection on the sets of critical and post-critical points. If $x \in \Julia_g \cap (\Crit(g)\cup P_g)$, then every point in $q^{-1}(x)$ is in the accumulation set of a core of $f$. 
\end{lem}

By considering $f$ as a simplicial tuning of $g$, we can use the definitions of internal and external edges in Definition \ref{defn:IntExtEdges}. Internal edges of simplicial quotients are periodic or preperiodic.

\begin{lem}[Depths of internal edges of $g$]\label{lem:IntEdgeG}
	Let $g$ be the simplicial quotient of a \pcf polynomial $f$. 
	Then the following are equivalent:
	\begin{itemize}
		\item an edge $e$ of $\mathcal{T}_g$ is an internal edge.
		\item $e$ is periodic or preperiodic.
		\item $\depth(e)=0$.
	\end{itemize}
\end{lem}
\begin{proof}
	Every internal edge is an edge of a star, so it is periodic or preperiodic. By definition every periodic or preperiodic edge has depth 0. Suppose $e$ is an external edge. The endpoints of $e$ are Julia points. If $e$ has depth 0, then $e$ consists of preimages of periodic edges. It follows that $e$ has to be contained in a core of $f$, which contradicts the assumption that $e$ is an external edge.
\end{proof}
The same proof as Lemma \ref{lem:0or1} gives:
\begin{lem}[Depths of internal edges of $f$]\label{lem:IntEdgeF}
	Let $f$ be a \pcf polynomial and $\mathcal{T}_f$ be its Hubbard tree. Let $v$ be a Fatou vertex in $\mathcal{T}_f$. Let $\mathcal{T}_{v}$ be the core subtree of $\mathcal{T}_f$ whose core is $\mathcal{K}$, and let $e$ be an edge of $\mathcal{T}_v$. If any endpoint of $e$ is not a boundary point of $\mathcal{T}_v$, then $\depth(e)=0$. If $e$ has an endpoint $w$ that is a boundary point of $\mathcal{T}_v$, then
	\begin{itemize}
		\item $\depth(e)=0$ if and only if $w\in \mathcal{K}^\circ$, and
		\item $\depth(e)=1$ if and only if $w\in \mathcal{K}^\omega$.
	\end{itemize}
\end{lem}

\begin{prop}\label{prop:SimpliPoly}
	Let $f(z)\neq z^d$ be a \pcf polynomial. Then the following are equivalent:
	\begin{enumerate}
		\item $\mathcal{T}_f$ is simplicial.
		\item $\mathscr{C}_{gr}(f)=1$.
		\item $f$ has only one core.
		\item The simplicial quotient of $f$ is $z^d$.
	\end{enumerate}
\end{prop}
\begin{proof}
Suppose $\mathcal{T}_f$ is simplicial. Then by adding some vertices if necessary, every edge is either periodic or eventually mapped to a periodic edge. Therefore, $\mathscr{C}_{gr}(f)=1$. Conversely, if $\mathscr{C}_{gr}(f)=1$, then by adding more vertices if necessary, every edge is either periodic or eventually mapped to a periodic edge. Therefore, $\mathcal{T}_f$ is simplicial. So (1) and (2) are equivalent.

If $\mathcal{T}_f$ is simplicial, then the simplicial quotient of $f$ is $z^d$. Thus (1) implies (4). If the simplicial quotient of $f$ is $z^d$, then by definition, $f$ has unique core. Therefore (4) implies (3). Finally if $f$ has unique core $\mathcal{K}$, then the filled Julia set of $f$ equals to $\mathcal{K}$. We claim that every critical point is contained is contained in $\mathcal{K}^\circ$. Indeed, otherwise, let $c \in \mathcal{K}^\omega$ be a critical point. Since $f$ has local degree $\geq 2$ on $c$, we see that there must be a strictly pre-periodic core attached to $\mathcal{K}$ at $c$. This is a contradiction to $f$ has only one core, and proves the claim. Therefore, every critical point, and so does the Hubbard tree $\mathcal{T}_f$, is contained in $\mathcal{K}^N$ for some sufficiently large $N$. Thus, $\mathcal{T}_f$ is simplicial. This proves (3) implies (1) and concludes the proof of the proposition.
\end{proof}

\subsection*{Non-triviality of simplicial quotient}
A core $\mathcal{K}$ of a \pcf polynomial $f$ is said to be {\em trivial} if it contains only one Fatou component.

\begin{lem}\label{lem:nt}
Let $f \in \mathcal{P}_d$ be a \pcf polynomial with core entropy zero.
If $f(z) \neq z^d$, then it has at least one non-trivial core.
\end{lem}

Before proving Lemma \ref{lem:nt}, let us discuss its consequences first.
Let $f \in \mathcal{P}_d$ be a \pcf polynomial, and let
\begin{itemize} 
\item $k_f$ be the number of periodic bounded Fatou components, and 
\item $m_f$ be the number of critical points in periodic bounded Fatou components counted with multiplicities.
\end{itemize}
Note that $m_f \leq d-1$.

\begin{lem}\label{lem:dq}
Let $f(z) \neq z^d \in \mathcal{P}_d$ be a \pcf polynomial with core entropy zero.
Let $g$ be the simplicial quotient of $f$. Then either
\begin{itemize}
\item $k_g < k_f$, or
\item $m_g > m_f$.
\end{itemize}
\end{lem}
\begin{proof}
Recall that Fatou components of $f$ which can be joined by finite sequences of adjacent Fatou components are merged to one Fatou component of $g$ (Proposition \ref{prop:sqc}). In this process, some pre-periodic Fatou components containing critical points may also be merged into periodic Fatou components. Hence $k_g \leq k_f$ and $m_g \geq m_f$. Suppose $k_g = k_f$. We will show $m_g > m_f$.

By Lemma \ref{lem:nt}, there exists a non-trivial periodic core $\mathcal{K}$ of $f$.
Without loss of generality, we assume that $\mathcal{K}$ is fixed and there exists a fixed Fatou component $U$ in $\mathcal{K}$.

Since $k_g = k_f$, $\overline{U}$ does not intersect with the closures of other periodic Fatou components. Recall the construction of cores for periodic Fatou components. Suppose for contradiction that $\partial U$ does not contain any critical points of $f$. The connected component $\mathcal{K}^1_U$ of $f^{-1}(\overline{U})$ that contains $U$ is $\overline{U}$.
Thus, by induction $\mathcal{K}^n_U = \overline{U}$, so $\mathcal{K} = \overline{\bigcup_n \mathcal{K}^n_U} = \overline{U}$. This means that $\mathcal{K}$ contains only one bounded Fatou component, so $\mathcal{K}$ is trivial. This contradicts to the assumption that $\mathcal{K}$ is non-trivial.  
Therefore, $\mathcal{K}$ contains at least one additional critical point of $f$.

By the construction of simplicial quotients, the corresponding fixed Fatou component of $g$ has a critical fixed point whose multiplicity is strictly greater than the multiplicity of the center of $U$.
Since the other periodic Fatou components of $g$ contain at least the same number of critical points of $g$ as the corresponding periodic Fatou components of $f$ counted with multiplicities, we have $m_g > m_f$.
\end{proof}

Theorem \ref{prop:fb} now follows from Lemma \ref{lem:dq}.
\begin{proof}[Proof of Theorem \ref{prop:fb}]
Let $f_0 = f$, and let $f_1$ be the simplicial quotient of $f_0$.
It is easy to verify that the core entropy of the simplicial quotient $f_1$ is still zero.
Thus, if $f_1(z) \neq z^d$, the above argument also applies to $f_1$.
Since $k_f \geq 0$ and $m_f \leq d-1$ are both integers, the process must terminate in finite steps.
Theorem \ref{prop:fb} now follows.
\end{proof}

We now prove Lemma \ref{lem:nt}.
\begin{proof}[Proof of Lemma \ref{lem:nt}]
Suppose for contradiction that all periodic cores are trivial.
For each Fatou vertex $v$, let $\mathcal{T}_v \subseteq \mathcal{T}_f$ be the core subtree. Since every periodic core is trivial, $\mathcal{T}_v$'s are stars.
We consider a quotient map
$$
\Phi: \mathcal{T}_f \longrightarrow \widetilde{\mathcal{T}_f} = \mathcal{T}_f/\sim,
$$
where two points are identified if they are in the same core subtree $\mathcal{T}_v$.
We remark that $\widetilde{\mathcal{T}_f}$ may not be a Hubbard tree.

The dynamics of $f: \mathcal{T}_f \longrightarrow \mathcal{T}_f$ induces a map $\tilde f: \widetilde{\mathcal{T}_f} \longrightarrow \widetilde{\mathcal{T}_f}$.
Since $h(f) = 0$, we have $h(\tilde f) = 0$.
By Lemma \ref{lem:nicscc}, the directed graph of the adjacency matrix of $\tilde f$ contains a cycle. Thus, there exists a periodic edge $\widetilde E \subseteq \widetilde{\mathcal{T}_f}$, i.e., $\tilde f^k: \widetilde E \longrightarrow \widetilde E$ is a homeomorphism, where $k$ is the period of $\widetilde E$.
Without loss of generality, we assume $\widetilde E = [\tilde v, \tilde w]$ is fixed, and the homeomorphism is orientation preserving.

Let $E' = \Phi^{-1}(\widetilde E)$.
Since $\tilde{f}(\widetilde{E})=\widetilde{E}$, we have $f(E') \subseteq E'$.
Let $E$ be an edge in $\overline{E' - \bigcup_v \mathcal{T}_v}$, which exists because $T_v$'s are  disjoint.
The edge $E$ is an exterior edge of $\mathcal{T}_f$. Denote the endpoints of $E$ by $a_1$ and $a_2$.

We show that $f: E \longrightarrow E$ is a homeomorphism.
It suffices to prove $f(a_i) = a_i$ for $i\in \{1,2\}$.
If $a_i \notin \mathcal{T}_v$ for any Fatou vertex $v$, then $\Phi^{-1}(\Phi(a_i))=\{a_i\}$. Then it follows from $\Phi: \widetilde{E} \longrightarrow \widetilde{E}$ being homeomorphic that $f(a_i) = a_i$.
If $a_i \in \mathcal{T}_v$ for some Fatou vertex, there is a fixed Fatou component $U_v$ with $a_i \in \partial U_v$.
Suppose for contradiction that $f(a_i) \neq a_i$. 
Since $f(E) \subseteq E'$, this forces $a_i$ to be a critical point.
This is a contradiction to the core of $U_v$ being trivial. Hence $f(a_i)=a_i$ and $f:E\to E$ is a homeomorphism.

Since $f: \mathcal{T}_f \longrightarrow \mathcal{T}_f$ is expanding (Definition \ref{defn:aht1}), $E$ must contain a fixed Fatou point.
This is a contradiction as $E$ is an exterior edge.
\end{proof}

\subsection{Bounding the combinatorial complexity}
\begin{proof}[Proof of Theorem \ref{thm:fb}]
Suppose $f$ has core entropy zero.
Then by Theorem \ref{prop:fb}, there exists a finite sequence $f_0 = f, f_1,..., f_k(z) = z^d$ such that $f_{i+1}$ is the simplicial quotient of $f_i$.
Since $f_i$ is a simplicial tuning of $f_{i+1}$, $f$ is obtained from $z^d$ by a finite sequence of simplicial tunings.
The reverse direction follows from Corollary \ref{cor:fcz}.
\end{proof}

To prove Theorem \ref{thm:cgrb}, we need a bound of $\mathscr{C}_{c}(f)$ in the other direction.
To achieve this, we prove the following lemma (cf. Lemma \ref{lem:fbgr}).
\begin{lem}\label{lem:sbgr}
Let $f\in\mathcal{P}_d$ be a \pcf polynomial with core entropy zero, with $f(z) \neq z^d$.
Let $g$ be the simplicial quotient of $f$.
Then
$$
\mathscr{C}_{gr}(f) = \mathscr{C}_{gr}(g) + 1.
$$
\end{lem}

\begin{proof}
By Lemma \ref{lem:fbgr}, it suffices to show $\mathscr{C}_{gr}(f) \geq \mathscr{C}_{gr}(g) + 1$.

Let $\mathcal{T}_f$ and $\mathcal{T}_g$ be the Hubbard trees for $f$ and $g$.
Note that $f$ is a simplicial tuning of $g$ (see \S \ref{subsec:sthf}). 
The Hubbard tree $\mathcal{T}_f$ is constructed from $\mathcal{T}_g$ by removing $\mathcal{T}_v$ and gluing back $\mathcal{T}_{v,f}$, where $v$ is a branch Fatou point, or a critical or post-critical Fatou point.
Here we use the notation $\mathcal{T}_{v,f}$ instead of $\mathcal{T}_{v,new}$, which was used in \S \ref{subsec:sthf}, to emphasize the dependence on $f$.

We define the vertex set $\mathcal{V}_g$ as the smallest forward invariant set of $\mathcal{T}_g$ containing the critical points, branch points, and $\partial \mathcal{T}_v$.
Let $\mathcal{G}_f$ and $\mathcal{G}_g$ be the corresponding directed graphs.

Let $a \in \mathcal{G}_g$ be a vertex with maximal depth so that $\mathscr{C}_{gr}(g) = 1+\depth(a)$.
Let $m = \depth(a)$.
Then by Lemma \ref{lem:mnsc}, we have a maximal sequence
$$
a = a_0 \geq a_1 \geq ... \geq a_{m}
$$
where $a_i$'s are in disjoint simple cycles $\mathcal{C}_i$.

Consider the case $\mathscr{C}_{gr}(g)\le 1$. We use properties polynomials with simplicial Hubbard trees (see Proposition \ref{prop:SimpliPoly}).
If $\mathscr{C}_{gr}(g)=0$, then $g=z^d$. Then the Hubbard tree $\mathcal{T}_f$ of $f$ is simplicial so that $\mathscr{C}_{gr}(f)=1$.
If $\mathscr{C}_{gr}(g)=1$, then $\mathcal{T}_g$ is simplicial. If $\mathscr{C}_{gr}(f)=1$ as well, then $\mathcal{T}_f$ is also simplicial. Then $g=z^d$ because $g$ is the simplicial quotient of $f$, which contradicts $\mathscr{C}_{gr}(g)=1$.
Hence we may assume $m=\mathscr{C}_{gr}(g)-1 \ge 1$.

Let $E_i$ be the corresponding edge of $\mathcal{T}_g$ for $a_i$. Let $l_i$ be the length of the cycle $\mathcal{C}_i$. By the maximality of the sequence $(a_i)$, the map
$$
	g^{l_m}: E_m \longrightarrow E_m
$$ 
is a homeomorphism.
It follows from Lemma \ref{lem:IntEdgeG} that $E_m$ is an internal edge and $E_i$ is an external edge for all $i < m$.

Since external edges of $\mathcal{T}_f$ and $\mathcal{T}_g$ are in bijective correspondence, we have an external edge $\widetilde{E_i}$ of $\mathcal{T}_f$ that corresponds to the external edge $E_i$ of $\mathcal{T}_g$ for $i<m$. Let $\tilde a_0 \geq \tilde a_1 \geq ... \geq \tilde a_{m-1}$ be the vertices in $\mathcal{G}_f$ corresponding to $\widetilde{E_i}$. Note that $\tilde a_i$ are in disjoint simple cycles.

It suffices to show that 
$\tilde a_{m-1}$ has depth at least 2.
To prove this, we denote the internal edge $E_m$ by $E_{m} = [x, y]$, where $x$ is the center of a periodic Fatou component $U$ and $y \in \partial U_{v}$.
By the definition of simplicial quotient, the boundary point $y$ corresponds to a point $\tilde y\in \partial \mathcal{T}_{v,f}\subset\mathcal{T}_f$. 

\vspace{5pt}
{\em \noindent Claim: There exists a point $y' \in \Int(E_{m-1})$ with $g^k(y') = y$ for some $k$. Therefore, either
	\begin{itemize}
		\item[(1)] $y$ is in the forward orbit of a critical point of $g$, or
		\item[(2)] there exists another edge $E_m'$ with $y\in \partial E_m'$ so that $E_m' \subseteq g^k(E_{m-1})$.
	\end{itemize}}
\begin{proof}[Proof of the claim]
Since $a_{m-1} \ge a_m$ and $a_{m-1}$ and $a_m$ are in different cycles, we have
	\[
		\limsup_{n\to \infty}\#\{\mathrm{paths~from~} a_{m-1} \mathrm{~to~} a_m\}=\infty.
	\]
Hence there exists $k>0$ so that at least three small segments of $E_{m-1}$
are homeomorphically mapped onto $E_m$ by $f^k$. Then we have $y'\in \Int(E_{m-1})$ with $f^k(y')=y$.
	
The cases (1) and (2) are determined according to whether $f^k$ folds $E_{m-1}$ at $y'$.
\end{proof}

We now show $\depth(\tilde a_{m-1}) \ge 2$ for each of the cases (1) and (2).

Case (1): Suppose that $y$ is in the forward orbit of a critical point of $g$. By Lemma \ref{lem:PostCritPtisAccPt}, $\tilde{y}$ is in the accumulation set of the core $\mathcal{K}_v$. Thus, $\tilde{y}$ is an endpoint of $\mathcal{T}_{v,f}$.
Hence, there is a unique edge $\widetilde{E_m}$ of $\mathcal{T}_{v,f}$ that has $\tilde{y}$ as an endpoint so that $\widetilde{E_{m}} \subseteq f^n(\widetilde{E_{m-1}})$ for some $n>0$. Let $\tilde a_m$ be the vertex of $\mathcal{G}_f$ corresponding to $\widetilde{E_m}$.
It follows from Lemma \ref{lem:IntEdgeF} that $\depth(\tilde a_m)=1$. 
Since $\tilde a_{m-1}\ge \tilde a_m$ and they are in disjoint simple cycles, $\depth(\tilde a_{m-1}) \ge 2$.

Case (2): Suppose that there exists another edge $E_m'$ with $y\in \partial E_m'$ so that $E_m' \subseteq g^k(E_{m-1})$. By the maximality assumption, $\depth(E_m')=0$.
Thus, $E_m'$ is also an internal edge by Lemma \ref{lem:IntEdgeG}. 
Hence, $E_m' \subseteq \mathcal{T}_{v'}$ for some $v'$. 
By Proposition \ref{prop:sqc}, $\tilde{y}$ is in the accumulation set of either $\mathcal{K}_v$ or $\mathcal{K}_{v'}$.
Without loss of generality, we may assume that $\tilde{y}$ is in the accumulation set of the core $\mathcal{K}_v$.
Let $\widetilde{E_m}$ be the edge of $\mathcal{T}_{v,f}$ having $\tilde y$ as an endpoint and let $\tilde a_m$ be the corresponding vertex in $\mathcal{G}_f$. By Lemma \ref{lem:IntEdgeF}, $\tilde a_m$ has depth $1$. 
Since $\tilde a_{m-1}\ge \tilde a_m$ and they are in disjoint simple cycles, we have $\depth(\tilde a_{m-1}) \ge 2$.
\end{proof}

Theorem \ref{thm:cgrb} now follows immediately.
\begin{proof}[Proof of Theorem \ref{thm:cgrb}]
By Corollary \ref{cor:fcz}, $\mathscr{C}_{c}(f) \geq \mathscr{C}_{gr}(f)$.

On the other hand, let $f_0 = f, f_1,..., f_k(z) = z^d$ with $f_{i+1}$ is the simplicial quotient of $f_i$.
By Lemma \ref{lem:sbgr}, $\mathscr{C}_{gr}(f) = k$.
Thus $\mathscr{C}_{c}(f) \leq \mathscr{C}_{gr}(f)$.
\end{proof}

\begin{rmk}
We remark that from the proof, the sequence of simplicial quotients $f_0 = f, f_1,..., f_k(z) = z^d$ with $f_{i+1}$ is the simplicial quotient of $f_i$ gives the minimal number of simplicial tunings to obtain $f$ from $z^d$.
\end{rmk}

\section{Bifurcation from simplicial tuning}\label{subsec:zim}
In this section, we use the theory developed in \cite{L21a, L21b} to study the relations between bifurcations and simplicial tunings.
The main result we will prove in this section is the following theorem.

\begin{theorem}\label{prop:gbf}
Let $g \in \mathcal{P}_d$ be a \pcf polynomial.
Let $f \in \mathcal{P}_d$ be a simplicial tuning of $g$.
Then $\mathcal{H}_g$ bifurcates to $\mathcal{H}_f$.
\end{theorem}

The idea is to construct a sequence $g_n$ of {\em quasi \pcf degenerations} in $\mathcal{H}_g$ that converges to a root $g_\infty$ of $\mathcal{H}_f$, which will be defined in this section.
These quasi \pcf degenerations are combinatorially parameterized by quasi invariant forests and allow us to change the dynamics on the bounded Fatou component of $g$.
As a corollary, we have:

\begin{cor}\label{cor:gbf}
Let $f \in \mathcal{P}_d$ be obtained by a finite sequence of simplicial tunings from $z^d$.
Then $f \in \Mol_d$.
\end{cor}

Recall that the {\em bifurcation complexity} $\mathscr{C}_{b}(f)$ is the smallest number of \shcs one needs to bifurcate to arrive at $\mathcal{H}_f$ from $\mathcal{H}_d$, and the {\em combinatorial complexity} $\mathscr{C}_{c}(f)$ is the smallest number of simplicial tunings needed to get $f$. 
Here we allow $\mathscr{C}_{b}(f)$ and $\mathscr{C}_{c}(f)$ to be $+\infty$ if there are no finite chains of bifurcation or simplicial tunings. There is an inequality between these two complexities.

\begin{theorem}\label{thm:bcbound}
Let $f \in \mathcal{P}_d$ be a \pcf polynomial.
Then
$$
\mathscr{C}_{b}(f) \leq \mathscr{C}_{c}(f).
$$
Moreover, the bound is sharp, i.e., there exists a degree $d$ \pcf polynomial $f$ with $\mathscr{C}_{c}(f) = \mathscr{C}_{b}(f)$.
\end{theorem}

\subsection*{Quasi \pcf degeneration}\label{sec:bst}
Let us briefly summarize the theory of quasi \pcf degeneration. We refer the reader to \cite{L21a, L21b} for more details.

Let $\mathcal{S}=(|\mathcal{S}|, \Phi, \delta)$ be a mapping scheme.
Consider the Blaschke model space $\BP^\mathcal{S}$, which is defined in Definition \ref{defn:bms}.
We first define an extended metric, a metric allowing $\infty$ distances, $d_{\mathcal{S}}$ on $|\mathcal{S}| \times \D$ by
\begin{itemize}
\item $d_{\mathcal{S}}(x,y) = \infty$ if $x \in s\times \D$ and $y\in t\times\D$ with $s \neq t$, and
\item $d_{\mathcal{S}}(x,y) = d_{s\times \D}(x, y)$ if $x, y\in s\times \D$,
\end{itemize}
where $d_{s\times \D}$ is the hyperbolic metric on the unit disk.

We are interested in a particular degeneration in $\BP^\mathcal{S}$, called {\em quasi \pcf degeneration}.
\begin{defn}[Quasi \pcf degeneration]\label{defn:qpcfbm}
Let $\mathcal{S}$ be a mapping scheme of degree $d$ and $\BP^\mathcal{S}$ be the corresponding Blaschke model space.
Let $\{\bp_n\}$ be a sequence in $\BP^\mathcal{S}$.
For $K>0$, $\bp_n$ is said to be {\em $K$-quasi post-critically finite} if we can label the critical points by $c_{1,n},..., c_{2d-2, n}$, and for any sequence $\{c_{i,n}\}_n$ of critical points, there exist $l_i$ and $q_i$, called {\em quasi pre-periods} and {\em quasi periods} respectively, such that for any $n>0$ we have
$$
d_{\mathcal{S}}(\bp_n^{l_i}(c_{i,n}), \bp_n^{l_i+q_i}(c_{i,n})) \leq K.
$$
We say $\{\bp_n\}$ is {\em quasi post-critically finite} if it is $K$-quasi post-critically finite for some $K>0$.
\end{defn}

\subsection*{Quasi invariant forests}
A {\em ribbon structure} on a finite tree is a choice of planar embedding up to isotopy.
A ribbon structure can also be defined as the assignment of a cyclic ordering of the edges incident to each vertex.
A {\em ribbon finite tree} is a finite tree with a ribbon structure, and an isomorphism between ribbon finite trees is an isomorphism between finite trees that preserves the ribbon structures.
A {\em marked finite tree} $(\mathcal{T}, \mathcal{P})$ is a finite tree with a subset $\mathcal{P} \subseteq \mathcal{V}$ of the vertex set $\mathcal{V}$ for $\mathcal{T}$.
%

In \cite[\S 3]{L21b}, for any quasi post-critically finite sequence $\{\bp_n \in \BP^\mathcal{S}\}$, a sequence of marked ribbon forests
$$
(\mathcal{T}_n, \mathcal{P}_n):= \bigcup_{s\in |\mathcal{S}|} (\mathcal{T}_{s,n}, \mathcal{P}_{s,n})
$$
is constructed, where each $\mathcal{T}_{s,n} \subseteq s\times \D$ is a finite ribbon tree. 
We call $(\mathcal{T}_n, \mathcal{P}_n)$'s the {\em quasi-invariant forests} for the sequence $\{\bp_n\}$.
They capture all interesting dynamical features of the sequence $\{\bp_n\}$, which can be modeled by simplicial maps, as described in the following theorem.

\begin{theorem}[\cite{L21b} Theorem 3.2]\label{thm:qit}
Let $\mathcal{S} = (|\mathcal{S}|, \Phi, \delta)$ be a mapping scheme and $\{\bp_n\}$ be a quasi post-critically finite sequence in $\BP^\mathcal{S}$. 
After passing to a subsequence, there exist a constant $K > 0$, a marked ribbon forest $(\mathcal{T}, \mathcal{P})=\bigcup_{s\in |\mathcal{S}|} (\mathcal{T}_s, \mathcal{P}_s)$ with vertex set $\mathcal{V}=\bigcup_{s\in |\mathcal{S}|}\mathcal{V}_s$, where $(\mathcal{T}_s, \mathcal{P}_s)$ is a marked ribbon finite tree, a simplicial map 
$$
F = \bigcup_{s\in |\mathcal{S}|} F_s:  (\mathcal{T}, \mathcal{P})=\bigcup_{s\in |\mathcal{S}|} (\mathcal{T}_s, \mathcal{P}_s) \longrightarrow (\mathcal{T}, \mathcal{P})
$$ with
$$
F_s: (\mathcal{T}_s, \mathcal{P}_s) \longrightarrow (\mathcal{T}_{\Phi(s)}, \mathcal{P}_{\Phi(s)}),
$$
and a sequence of isomorphisms
$$
\phi_{n}: (\mathcal{T}, \mathcal{P}) \longrightarrow (\mathcal{T}_n, \mathcal{P}_n)
$$
where $(\mathcal{T}_n, \mathcal{P}_n)$'s are the quasi-invariant forests for $\{\bp_n\}$
such that the following properties hold.
\begin{itemize}
\item (Degenerating vertices.) If $v_1\neq v_2 \in \mathcal{V}$, then 
$$
d_\mathcal{S}(\phi_n(v_1), \phi_n(v_2)) \to \infty.
$$
\item (Geodesic edges.) If $E =[v_1,v_2] \subseteq \mathcal{T}$ is an edge, then the corresponding edge $\phi_n(E) \subseteq \mathcal{T}_n$ is a hyperbolic geodesic segment connecting $\phi_n(v_1)$ and $\phi_n(v_2)$.
\item (Critically approximating.) Any critical points of $\bp_n$ are within $K$ distance from the vertex set $\mathcal{V}_n:= \phi_n(\mathcal{V})$ of $\mathcal{T}_n$.
\item (Quasi-invariance on vertices.)
If $v\in \mathcal{V}$, then
$$
d_\mathcal{S}(\bp_n(\phi_n(v)), \phi_n(F(v))) \leq K \text{ for all } n.
$$

\item (Quasi-invariance on edges.) If $E\subseteq \mathcal{T}$ is an edge and $x_n \in \phi_n(E)$, then there exists $y_n \in \phi_n(F(E))$ so that
$$
d_\mathcal{S}(\bp_n(x_n), y_n) \leq K \text{ for all } n.
$$
If $E$ is a periodic edge of period $q$, then 
$$
d_\mathcal{S}(\bp_n^q(x_n), x_n) \leq K \text{ for all } n.
$$
\end{itemize}
\end{theorem}

\begin{rmk}
By the construction in \cite[\S 3]{L21b}, if $s \in |\mathcal{S}|$ is periodic under $\Phi$, then $\mathcal{P}_s$ consists of a single point $p_s$, and $\phi_n(p_s) \in \{s\} \times \D$ is the unique attracting periodic point in $\{s\} \times \D$ for $\bp_n$.
\end{rmk}

\subsection*{Rescaling limits}
Let $v \in \mathcal{V}_s \subseteq \mathcal{V}$. 
We define a {\em normalization at $v$} or a {\em coordinate at $v$} as a sequence $\{M_{v,n}\in \Isom(\mathbb{D},d_{hyp})\}_n$ so that 
$$
\phi_n(v) = (s, M_{v,n}(0)).
$$
Here $(\mathbb{D},d_{hyp})$ denotes the unit disk with the hyperbolic metric. Note that different choices for the sequence $\{M_{v,n}\}_n$ differ by pre-composing with rotations that fix $0$, which form a compact group.
We define a map 
\begin{align*}
\hat M_{v,n} : \D &\longrightarrow \{s\} \times \D\\
z &\mapsto (s, M_{v,n}(z)).
\end{align*}

Let us fix such a normalization $\{M_{v,n}\}_n$ for each vertex $v\in \mathcal{V}$.
To distinguish normalizations at different vertices, we use $\D_v$ to denote the disk associated to the normalization at $v$.
By the quasi-invariance on vertices in Theorem \ref{thm:qit}, we have
$$
d_\mathcal{S}(\bp_n(\phi_n(v)), \phi_n(F(v))) \leq K
$$ 
for some constant $K$.
Thus, $d_\D(\hat M_{F(v),n}^{-1}\circ \bp_n \circ \hat M_{v,n}(0), 0) \leq K$.
Therefore, by \cite[Proposition 2.3]{L21a}, after possibly passing to a subsequence, the sequence of proper maps
$$
\hat M_{F(v),n}^{-1}\circ \bp_n \circ \hat M_{v,n}: \D_v \longrightarrow \D_{F(v)}
$$
converges compactly to a proper holomorphic map
$$
\rl_v = \rl_{v\to F(v)}: \D_v \longrightarrow \D_{F(v)}.
$$
We call this map $\rl_v$ the {\em rescaling limit} of $\bp_n$ at $v$ and denote its degree by $\delta(v)$.
We remark there are exactly $\delta(v) -1$ critical points of $\bp_n$ counted with multiplicity that stay within a uniform $d_\mathcal{S}$-distance from $\phi_n(v)$.
By critically approximating property in Theorem \ref{thm:qit}, we have for any $s\in |\mathcal{S}|$,
$$
\delta(s) = 1 + \sum_{v\in \mathcal{V}_s} (\delta(v)-1).
$$ 

A vertex $v \in \mathcal{V}$ is critical if $\delta(v) \geq 2$.
A vertex is called {\em a Fatou point} if it is eventually mapped to a critical periodic orbit, and is called a {\em Julia point} otherwise.

\subsection*{Angled forest map}
To give a combinatorial description of quasi-invariant forests, we introduce the notion of angled forest here.
This is a straightforward generalization of angled tree maps in \cite[\S 3]{L21a}, to which we refer the readers for more details and comparisons with abstract Hubbard trees.

Let $m_d: \mathbb{S}^1\longrightarrow \mathbb{S}^1$ denote the multiplication by $d$ map where $\mathbb{S}^1$ is identified with $\R/\Z$.
By our convention, $m_1$ is the identity map.
For $d\geq 2$, $m_d$ gives a topological model of the dynamics on the Julia set of a degree $d$ hyperbolic and a doubly parabolic Blaschke product (see \S \ref{subsec:lm} for definitions).

To set up a framework that also works for singly parabolic or boundary-hyperbolic Blaschke products uniformly, we consider an extended circle $\mathbb{S}^{1}_{d}$, which is naturally regarded as {\em cyclically ordered set} (see \cite[\S 2]{McM09}).
As a set, 
$\mathbb{S}^{1}_{d}$ is constructed from $\mathbb{S}^1$ by adding (formal symbols) $x^-, x^+$ for any point $x$ in the backward orbit of $0$ under $m_d$ for $d\geq 2$.
The cyclic ordering on $\mathbb{S}^{1}_{d}$ is defined so that $x^-$ (or $x^+$) is regarded as a point infinitesimally smaller than $x$ (or bigger than $x$ respectively) in the standard identification of $\mathbb{S}^1 = \R/\Z$.
Note that $\mathbb{S}^1$ naturally embeds into $\mathbb{S}^{1}_{d}$.
We call a point in the image of this embedding a {\em regular point}.
We use the convention that $\mathbb{S}^1_1 = \mathbb{S}^1$.

Given any integer $k \geq 1$, the map $m_k$ naturally extends to
$m_k: \mathbb{S}^1_{d} \longrightarrow \mathbb{S}^1_{d}$ 
by setting
$m_k (x^\pm) = m_k(x)^\pm$.
This is well-defined because if $x$ is in the backward orbit of $0$ under $m_d$, $m_k(x)$ is also in the backward orbit of $0$ under $m_d$.

If $f$ is a degree $d$ boundary-hyperbolic Blaschke product, i.e., $f$ has an attracting fixed point $a$ on the circle, 
the Julia set $J$ of $f$ is a Cantor set on $\mathbb{S}^1$. Note that by Denjoy-Wolff Theorem, $a$ is the unique attracting fixed point (see \S \ref{subsec:lm}).
The complement $\mathbb{S}^1 - J$ consists of countably many intervals, which are all eventually mapped to the unique interval $I \subseteq \mathbb{S}^1$ that contains the attracting fixed point $a$. The boundary $\partial I$ consists of two repelling fixed points of $f$.
Let $\mathcal{O}(a)$ be the backward orbit of the attracting fixed point $a$.
Then there exists bijective map $\eta_f: \mathbb{S}^{1}_{d} \longrightarrow J(f) \cup \mathcal{O}(a)$ which preserves the cyclic ordering so that
$f \circ \eta_f = \eta_f \circ m_d$.
Note that $\eta_f(0) = a$, and $\eta_f(0^\pm) = \partial I$ (see Figure \ref{fig:BHB}).


\begin{figure}[ht]
	\centering
		{\footnotesize
		\def\svgwidth{0.45\textwidth}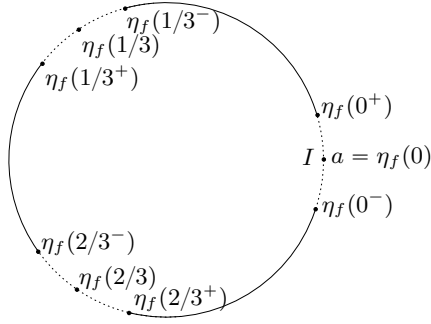
		}
	\caption{An illustration of the conjugacy $\eta_f$ for a degree $3$ boundary-hyperbolic Blaschke product. The Julia set is a Cantor set, constructed by removing the backward orbits of the interval $I$.}
	\label{fig:BHB}
\end{figure}

To model the dynamics of the pullbacks, we construct for any integer $D\ge1$ the set $\mathbb{S}^{1}_{d,D}$ by adding $x^-, x^+$ to $\mathbb{S}^1$ if $m_D(x)$ is in the backward orbit of $0$ under $m_d$, and the cyclic ordering is constructed in the same way.
Note that by this construction, $m_D:\mathbb{S}^{1}_{d,DD'} \longrightarrow \mathbb{S}^{1}_{d,D'}$ is a degree $D$ covering between cyclically ordered sets (see \cite[\S 2]{McM09} for detailed definitions).
Note that $\mathbb{S}^{1}_{d,1} = \mathbb{S}^{1}_{d}$.

Let $\mathcal{S} = (|\mathcal{S}|, \Phi, \delta)$ be a mapping scheme.
We define a {\em forest map} modeled on $\mathcal{S}$ as 
\begin{itemize}
\item a marked ribbon forest $(\mathcal{T}, \mathcal{P})=\bigcup_{s\in |\mathcal{S}|} (\mathcal{T}_s, \mathcal{P}_s)$, where $(\mathcal{T}_s, \mathcal{P}_s)$ is a marked ribbon finite tree such that $\mathcal{P}_s = \{p_s\}$ for any periodic point $s \in |\mathcal{S}|$, and 
\item a simplicial map 
$$
F = \bigcup_{s\in |\mathcal{S}|} F_s:  (\mathcal{T}, \mathcal{P})=\bigcup_{s\in |\mathcal{S}|} (\mathcal{T}_s, \mathcal{P}_s) \longrightarrow (\mathcal{T}, \mathcal{P})
$$
which is a union of simplicial maps 
$$
F_s: (\mathcal{T}_s, \mathcal{P}_s) \longrightarrow (\mathcal{T}_{\Phi(s)}, \mathcal{P}_{\Phi(s)}).
$$
\end{itemize}

Let $\mathcal{V}$ be the vertex set of $\mathcal{T}$ and $\mathcal{V}_s:=\mathcal{V} \cap \mathcal{T}_s$. We define the {\em local degree function} $\delta: \mathcal{V} \longrightarrow \Z_{\geq 1}$ which assigns an integer $\delta(v) \geq 1$ to each vertex $v\in \mathcal{V}$.
We say that $\delta$ is {\em compatible} with $\mathcal{S}$ if for any $s \in |\mathcal{S}|$,
$$
\delta(s) = 1 + \sum_{v\in \mathcal{V}_s} (\delta(v)-1).
$$

If $v \in \mathcal{V}$ has pre-period $l$ and period $q$. We define the {\em cumulative degree}
$$
\Delta(v) := \delta({F^l(v)}) \delta({F^{l+1}(v)}) ... \delta({F^{l+q-1}(v)}),
$$
and the cumulative pre-periodic degree
$$
\Delta_{pre}(v) = \delta(v) \delta({F(v)})... \delta({F^{l-1}(v)}).
$$
We use the convention that $\Delta_{pre}(v) = 1$ for all periodic vertices.

We define an {\em angle function} $\alpha$ at $v$ as an injective map
$$
\alpha_v: T_v\mathcal{T}\hookrightarrow \mathbb{S}^1_{\Delta(v), \Delta_{pre}(v)}.
$$
We say $\alpha$ is {\em regular} at $v$ if $\alpha_v(T_v\mathcal{T}) \subseteq \mathbb{S}^1 \subseteq \mathbb{S}^1_{\Delta(v), \Delta_{pre}(v)}$.

We say an angle function $\alpha$ is {\em compatible (with the forest map)} if for any $v \in \mathcal{V}$ the function $\alpha_v$ satisfies the following three compatibility conditions. 
\begin{enumerate}
\item $\alpha_v$ is {\em cyclically compatible} if $x_1, x_2, x_3 \in T_v\mathcal{T}$ are clockwise oriented, then $\alpha_v(x_1), \alpha_v(x_2), \alpha_v(x_3)$ are also clockwise oriented. 
\item $\alpha_v$ is {\em dynamically compatible} if
\begin{itemize}
\item when $v= \p_s$ for some periodic $s \in |\mathcal{S}|$ and $\Delta(\p_s) = 1$, there exists a rigid rotation $R$, which is necessarily a rational rotation, so that
$
R \circ \alpha_{\p_s} = \alpha_{\p_s} \circ DF^q|_{T_{\p_s}\mathcal{T}}
$, where $q$ is the period of $s$,
\item otherwise, 
$
m_{\delta(v)} \circ \alpha_v = \alpha_{f(v)} \circ DF|_{T_v\mathcal{T}}
$.
\end{itemize}

\item $\alpha_v$ is {\em $\p_s$-compatible} if $v \in \mathcal{V}_s - \{\p_s\}$ is periodic and $x \in T_v\mathcal{T}$ is the tangent vector in the direction of $\p_s$, then $\alpha_v(x) = 0$.
\end{enumerate}

We remark that if $v \in \mathcal{V}_s$ is periodic of period $q$, then $s$ is periodic with period dividing $q$.
Let $x\in T_v\mathcal{T}$ be the tangent vector in the direction of $\p_s$. 
Then $D_vF^q(x) = x$ as $F$ is simplicial.
Thus condition (3) is compatible with the dynamics.

\begin{defn}[Angled forest maps]
An {\em angled forest map} modeled on the mapping scheme $\mathcal{S} = (|\mathcal{S}|, \Phi, \delta)$ is a triple 
$$
(F: (\mathcal{T}, \mathcal{P}) \rightarrow (\mathcal{T}, \mathcal{P}), \delta, \alpha = \{\alpha_v\})
$$ 
of a forest map modeled on $\mathcal{S}$ together with a compatible local degree function $\delta$ and a compatible angle function $\alpha$ which is regular at $p_s$ for any periodic $s \in |\mathcal{S}|$.
\end{defn}

\subsection*{Realizing angled forest map}
Let $F: (\mathcal{T}, \mathcal{P}) \rightarrow (\mathcal{T}, \mathcal{P})$ be the model of quasi-invariant forests for a quasi \pcf degeneration $\bp_n$.
A compatible local degree function $\delta$ is also constructed.

To define a compatible angle function, we note that if $v$ is a periodic Fatou point of period $q$, then
$$
\rl_v^q := \rl_{F^{q-1}(v)\to v} \circ ... \circ \rl_{v\to F(v)} : \D_v \longrightarrow \D_{v}
$$
is a proper holomorphic self map of the disk $\D_v$ with degree $\Delta_v$.
Then $\rl_v^q$ is a Blaschke product.
Let $J \cup \mathcal{O} \subseteq \mathbb{S}^1$ be the union of the Julia set of $\rl_v^q$ and the backward orbits of the attracting fixed point on $\mathbb{S}^1$. Note that $\mathcal{O}$ is empty unless $\rl_v^q$ is boundary-hyperbolic.
We have a semiconjugacy between cyclically ordered sets 
$$
\eta: \mathbb{S}^{1}_{\Delta_v} \left( = \mathbb{S}^1_{\Delta_v, \Delta_{v,pre}} \right) \longrightarrow J \cup \mathcal{O},
$$ 
whose fiber is either a singleton set, or $\{x, x^+\}, \{x^-, x\}, \{x^-,x, x^+\}$.
We remark that when $J = \mathbb{S}^1$, such a semiconjugacy is not unique, and any two such conjugacies are differed by an element of the automorphism group $\Z/(\Delta_v-1)$ of $m_{\Delta_v}$.
This ambiguity is resolved by using the {\em $p_s$-compatible condition} if $v \in \mathcal{V}_s -\{p_s\}$, and by the {\em anchored convention} if $v = p_s$ (see \cite[Definition 3.3]{L21a}).

We define a {\em regular inverse}  
$$
\eta^{-1}: J \cup \mathcal{O} \longrightarrow \mathbb{S}^1_{\Delta_v, \Delta_{v,pre}}
$$
as a section of $\eta$ that takes the regular value if the fiber contains more than one element.
Each tangent vector in $T_v\mathcal{T}$ corresponds to a point on $J \cup \mathcal{O}$ (see \cite[\S 3]{L21a}).
Thus, the regular inverse $\eta^{-1}$ gives a natural angle function
$$
\alpha_v: T_v\mathcal{T} \longrightarrow \mathbb{S}^1_{\Delta_v, \Delta_{v,pre}}.
$$

Angles at periodic Julia points $v$ are not canonical. We artificially assign the angles $\{\frac{i}{\nu}~|~ i=0,\dots, \nu-1\}$ according to their cyclic order where $\nu$ is the valence at $v$.
By pulling back, we can also construct an angle function
$\alpha_v$ for any pre-periodic vertex $v\in \mathcal{V}$, and one can verify that this angle function is compatible.

Therefore, given a quasi \pcf sequence $\bp_n \in \BP^\mathcal{S}$, after passing to a subsequence, we can associate an angled forest map $(F: (\mathcal{T}, \mathcal{P}) \rightarrow (\mathcal{T}, \mathcal{P}), \delta, \alpha)$.
We say such an angled forest map is {\em realized} by quasi \pcf degenerations.

\subsection*{Admissible angled forest map}
Let $(F: (\mathcal{T}, \mathcal{P}) \rightarrow (\mathcal{T}, \mathcal{P}), \delta, \alpha)$ be an angled forest map modeled on $\mathcal{S}$.
We now describe a sufficient condition for realization. 

Let $s \in |\mathcal{S}|$ be a periodic point, and let $p_s$ be the marked point in $\mathcal{V}_s$.
A periodic vertex $v$ is said to be {\em attached to $p_s$} if $[\p_s, v)$ contains no Fatou point.
Here $[\p_s,v)$ is the path in $\mathcal{T}$ that connects $\p_s$ and $v$ with the boundary point $v$ removed.
The {\em core} $\mathcal{T}_s^{C} \subseteq \mathcal{T}_s$ is defined as the convex hull of all periodic vertices attached to $p_s$.
Since $F$ is simplicial, any vertex in $\mathcal{T}_s^C$ is periodic.
Note that if $\delta(\p_s) \geq 2$, then $\mathcal{T}^C_s = \{\p_s\}$.

Recall that a tree $T$ is a star if there exists a unique vertex in $T$ that is not an endpoint. By convention, we consider a single vertex as a star with no endpoints.
The core $\mathcal{T}^C_s$ is said to be {\em critically star-shaped} if
\begin{itemize}
\item $\mathcal{T}^C_s$ is a star with center $\p_s$,
\item every endpoint of $\mathcal{T}^C_s$ is a periodic Fatou point, and
\item for any vertex $v \in \mathcal{T}^C_s$, the angle function $\alpha_v$ is regular at $v$.
\end{itemize}

We remark that if $\delta(\p_s) \geq 2$, then $\mathcal{T}^C _s= \{\p_s\}$ and the conditions are trivially satisfied.
If $\delta(\p_s) =1$, these conditions give a way to `normalize' the dynamics at $\p_s$.

\begin{defn}\label{defn:adm}
Let $(F: (\mathcal{T}, \mathcal{P}) \rightarrow (\mathcal{T}, \mathcal{P}), \delta, \alpha)$ be an angled forest map modeled on $\mathcal{S}$.
It is {\em admissible} if for any periodic point $s \in |\mathcal{S}|$,
\begin{itemize}
\item the core $\mathcal{T}_s^C$ is critically star-shaped, and
\item every periodic branch point in $\mathcal{V}_s$ other than $\p_s$ is a Fatou point.
\end{itemize}
\end{defn}

We remark that Definition \ref{defn:adm} is a generalization of admissible angled tree maps in \cite[\S 3]{L21a}. See the reference for a discussion on necessities of the conditions.
A similar inductive argument as in \cite[Theorem 4.1]{L21a} on the degree of the mapping scheme $\mathcal{S}$ yields:

\begin{theorem}\label{thm:rs}
Let $(F: (\mathcal{T}, \mathcal{P}) \rightarrow (\mathcal{T}, \mathcal{P}), \delta, \alpha)$ be an admissible angled forest map modeled on $\mathcal{S}$.
Then it is realizable.
\end{theorem}

\subsection*{Admissible splitting on Hubbard forest}
Let $F:H \longrightarrow H$ be a simplicial Hubbard forest modeled on a mapping scheme $\mathcal{S}$.
We define a marked set $P$ on $H$ as follows.
Let $s \in |\mathcal{S}|$ be a periodic point of period $q$.
We choose a period $q$ point $p_s \in H_s$ (see \cite[Corollary 3.9]{Poi10} for existence), and define $P_s = \{p_s\}$.
We assume that the choice is compatible, i.e., $p_{\Phi(s)} = F(p_s)$.
If $s$ is strictly pre-periodic point, we inductively define $P_s = F^{-1}(P_{\Phi(s)}) \subseteq H_s$.
Let $P := \bigcup_{s\in |\mathcal{S}|} P_s$.

Thus, $F: (H, P) \longrightarrow (H, P)$ is a forest map modeled on $\mathcal{S}$.
The Hubbard forest also comes with a compatible local degree function $\delta$, and angles between any pair of adjacent edges incident at a vertex, which is compatible with the dynamics (see \cite{Poi13}).
Thus, to construct an angle function, we simply need to specify the angle $0$ at each vertex.
We do this using the $p_s$-compatible condition and the anchored convention (see \cite[Definition 3.3]{L21a}).
Therefore, $(F: (H, P) \longrightarrow (H, P), \delta, \alpha)$ is an angled forest map.

The core of $H$ is critically star-shaped.
Indeed, if $\delta(\p_s) \geq 2$, then this is vacuously true.
Otherwise, each adjacent vertex $v$ to $\p_s$ is a periodic Fatou point by the expanding property of Hubbard forests.

On the other hand, $H$ may contain many periodic Julia branch points.
In the following, we introduce an operation on these branch points, which is called a {\em split modification}, to get an admissible angled tree map.

We remark that all the angle functions for $H$ are regular.
It is at this modification stage that we need to introduce non-regular points in the extended circle.

Let $v \in \mathcal{V}_s -\{\p_s\}$ be a periodic Julia branch point.
After passing to an iterate, we may assume that $v$ is fixed.
Let $S$ be the subtree, which is a star, consisting of all vertices adjacent to $v$.
Let $a_0$ be the vertex in $S$ corresponding to the direction associated to $\p_s$, and label the other vertices with $a_1,..., a_m$ in counterclockwise order.
Since $F$ is simplicial on $H$ and fixes $\p$, $a_0$ is fixed, and thus all $a_i$ are fixed.
By the expanding property, each $a_i$ is a fixed Fatou point for $i=0, 1,..., m$.

Let us modify $H_s$ locally within $S$.
\begin{enumerate}
	\item We begin with removing the interior of $S$.
	\item On the first level, we choose $k_1 \in \{1,..., m\}$, connect $a_0$ with $a_{k_1}$, and define the level of $a_{k_1}$ by one.
	\item On the second level, we choose $k_{2,1} \in \{1,..., k_1-1\}$ and $k_{2,2} \in \{k_1+1,..., m\}$, connect $a_{k_1}$ with $a_{k_{2,1}}$ and $a_{k_{2,2}}$, and define the levels of $a_{k_{2,1}}$ and $a_{k_{2,2}}$ by two.
	\item Inductively, $k_1, k_{2,1}, k_{2,2}$ divide the set $\{1,..., m\}$ into $4$ subsets (some subset may be empty), and we proceed as above for each of the subinterval.
\end{enumerate}
The trees $\tilde{S}$ that can be constructed in this way will be called {\em admissible splitting} (see Figure \ref{fig:M}).
For an admissible splitting, the level for each vertex $a_i$ that we assigned above is equal to the edge distance between $a_i$ and $a_0$ for $i\in\{0,1,\dots,m\}$. Every edge connects a level $k$ to a level $k+1$ vertex for some $k\ge1$.


\begin{figure}[ht]
	\centering
	\begin{subfigure}{1\textwidth}
		\centering
		{
			\def\svgwidth{0.3\textwidth}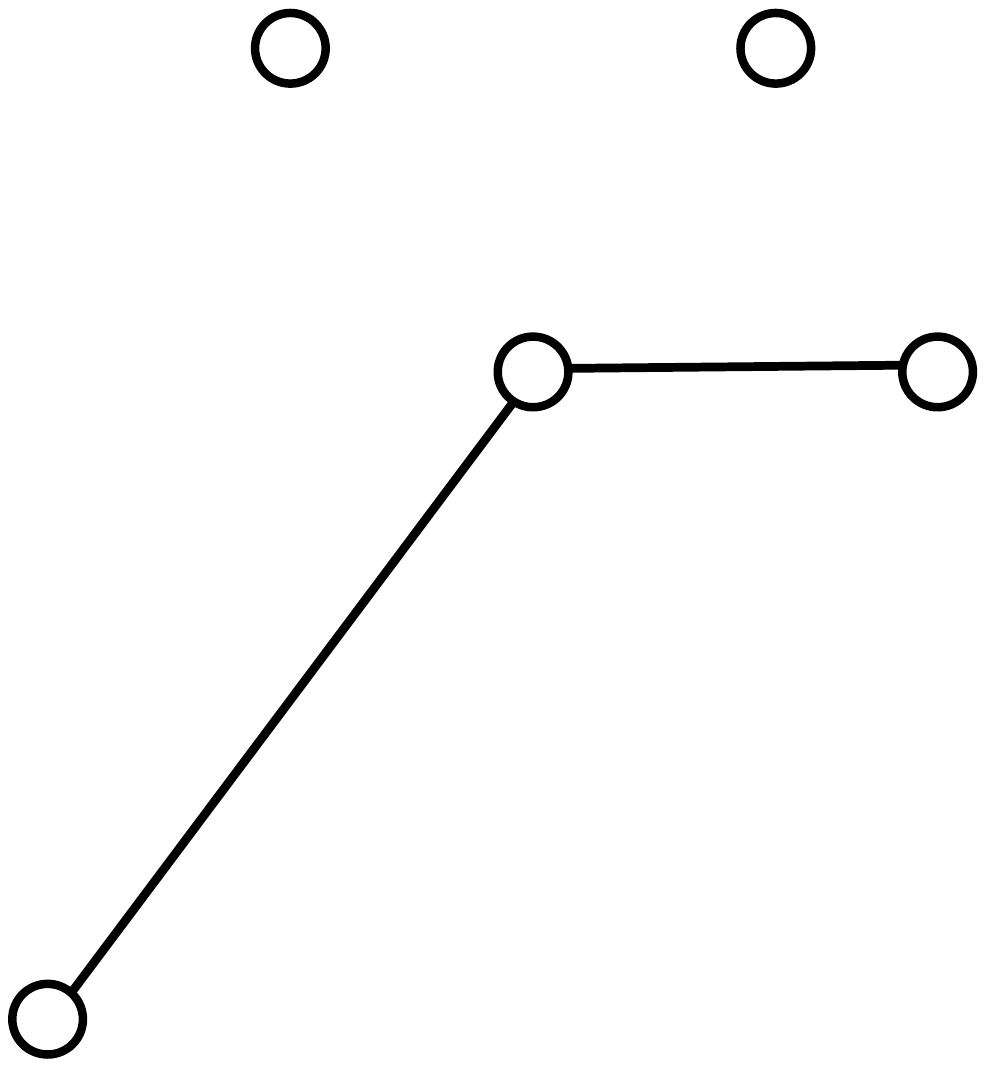
			\def\svgwidth{0.3\textwidth}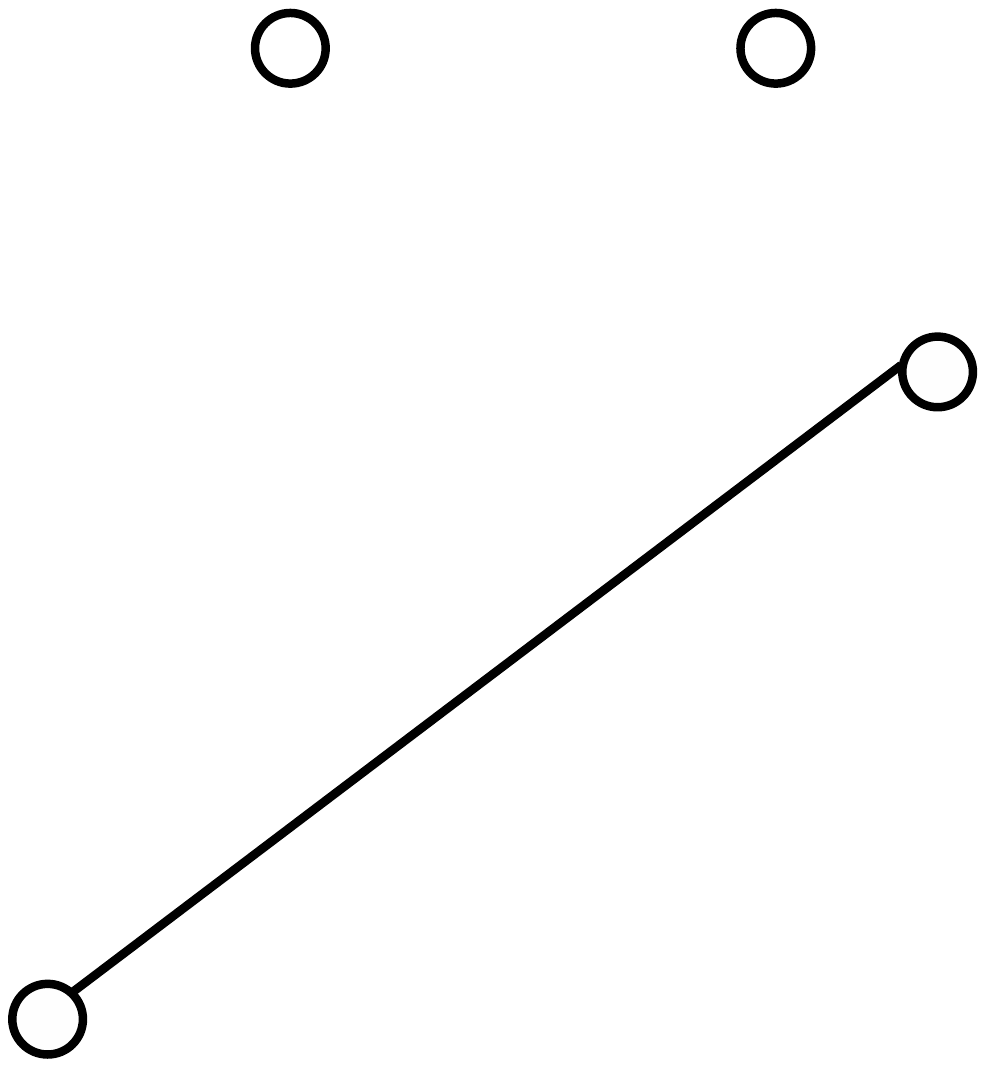
			\def\svgwidth{0.3\textwidth}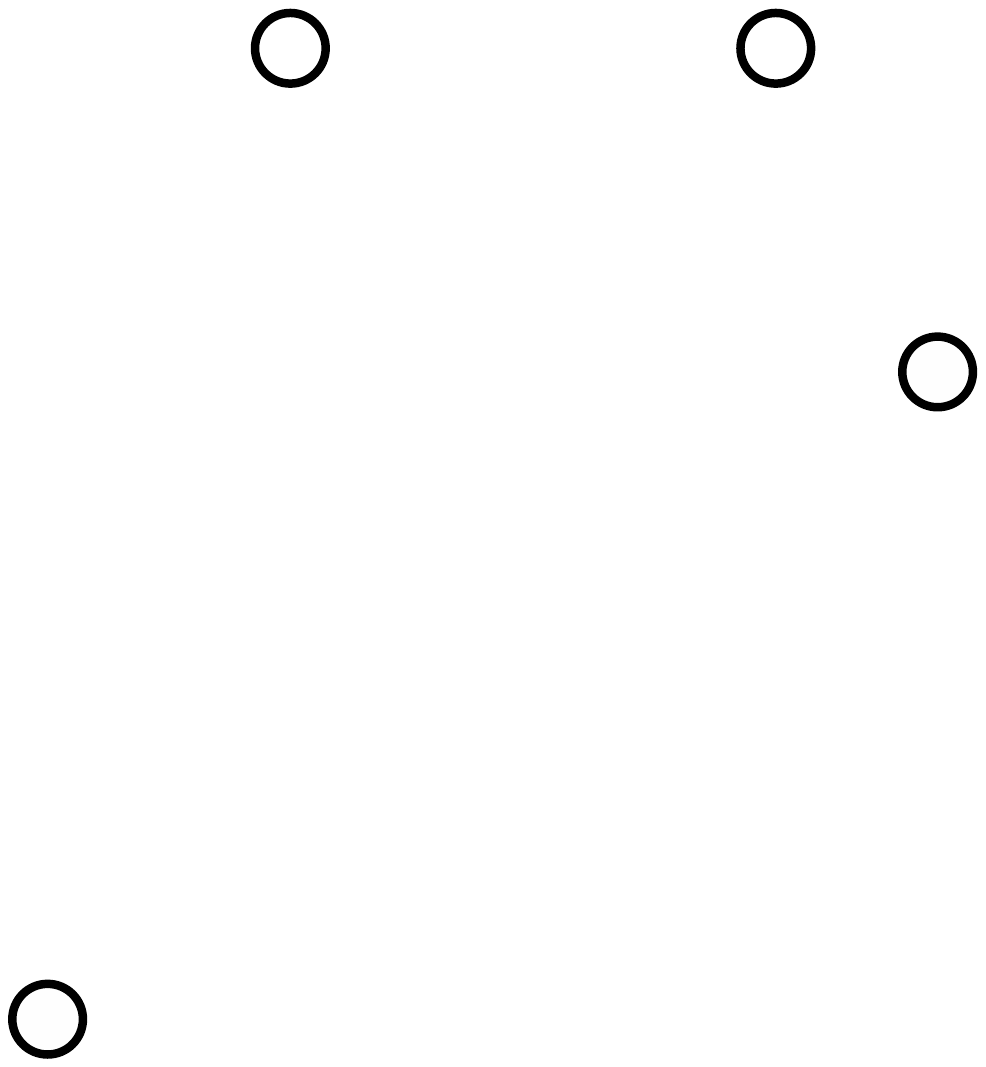
		}
		\caption{A star-shaped neighborhood of a periodic Julia branch point $v$ with two different admissible splittings with angles specified. In the first figure, the angles of tangent vectors at $v$ toward $a_i$'s are $i/5$.}
	\end{subfigure}
	\begin{subfigure}{1\textwidth}
		\centering
		{
			\def\svgwidth{0.3\textwidth}
\begingroup%
  \makeatletter%
  \providecommand\color[2][]{%
    \errmessage{(Inkscape) Color is used for the text in Inkscape, but the package 'color.sty' is not loaded}%
    \renewcommand\color[2][]{}%
  }%
  \providecommand\transparent[1]{%
    \errmessage{(Inkscape) Transparency is used (non-zero) for the text in Inkscape, but the package 'transparent.sty' is not loaded}%
    \renewcommand\transparent[1]{}%
  }%
  \providecommand\rotatebox[2]{#2}%
  \newcommand*\fsize{\dimexpr\f@size pt\relax}%
  \newcommand*\lineheight[1]{\fontsize{\fsize}{#1\fsize}\selectfont}%
  \ifx\svgwidth\undefined%
    \setlength{\unitlength}{792bp}%
    \ifx\svgscale\undefined%
      \relax%
    \else%
      \setlength{\unitlength}{\unitlength * \real{\svgscale}}%
    \fi%
  \else%
    \setlength{\unitlength}{\svgwidth}%
  \fi%
  \global\let\svgwidth\undefined%
  \global\let\svgscale\undefined%
  \makeatother%
  \begin{picture}(1,0.77272727)%
    \lineheight{1}%
    \setlength\tabcolsep{0pt}%
    \put(0,0){\includegraphics[width=\unitlength,page=1]{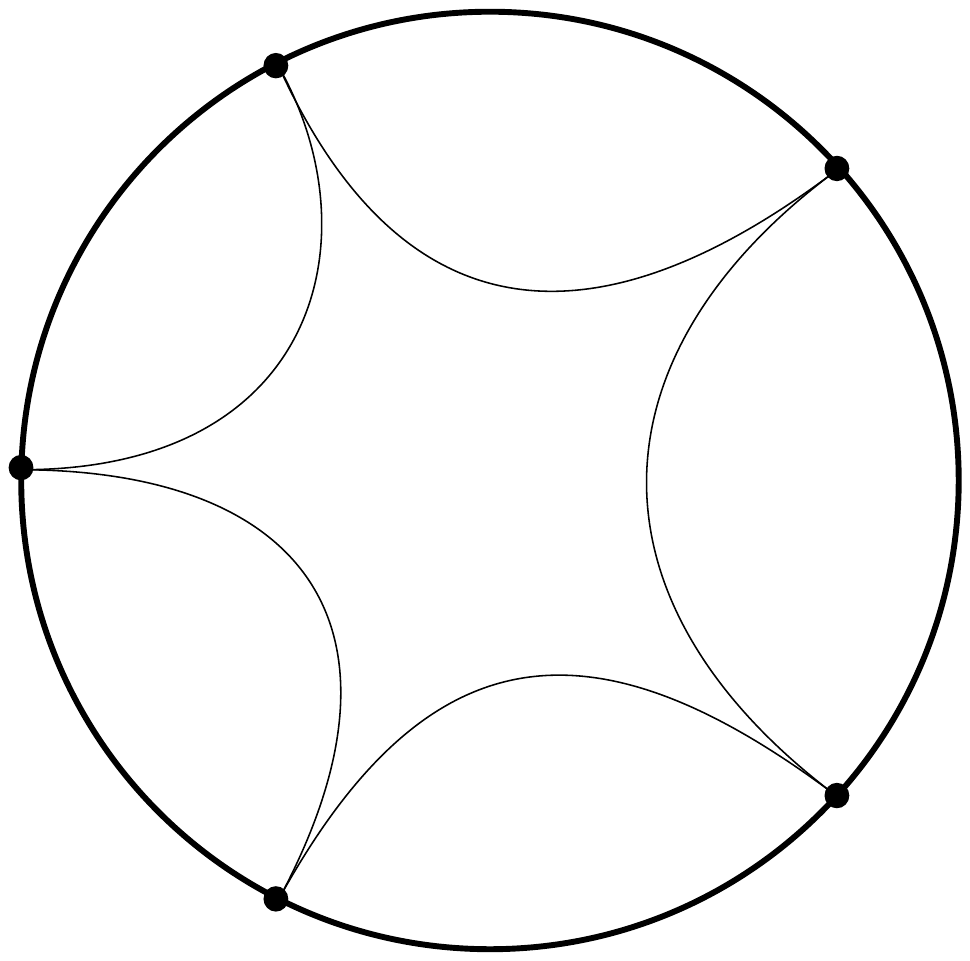}}%
    \put(0.05114024,0.35330933){\color[rgb]{0,0,0}\makebox(0,0)[lt]{\lineheight{1.25}\smash{\begin{tabular}[t]{l}{$A_{40}$}\end{tabular}}}}%
    \put(0.30211838,0.04025614){\color[rgb]{0,0,0}\makebox(0,0)[lt]{\lineheight{1.25}\smash{\begin{tabular}[t]{l}{$A_{01}$}\end{tabular}}}}%
    \put(0.69753348,0.12913013){\color[rgb]{0,0,0}\makebox(0,0)[lt]{\lineheight{1.25}\smash{\begin{tabular}[t]{l}{$A_{12}$}\end{tabular}}}}%
    \put(0.70327295,0.59150604){\color[rgb]{0,0,0}\makebox(0,0)[lt]{\lineheight{1.25}\smash{\begin{tabular}[t]{l}{$A_{23}$}\end{tabular}}}}%
    \put(0.25520522,0.68229325){\color[rgb]{0,0,0}\makebox(0,0)[lt]{\lineheight{1.25}\smash{\begin{tabular}[t]{l}{$A_{34}$}\end{tabular}}}}%
  \end{picture}%
\endgroup%

			\def\svgwidth{0.3\textwidth}
\begingroup%
  \makeatletter%
  \providecommand\color[2][]{%
    \errmessage{(Inkscape) Color is used for the text in Inkscape, but the package 'color.sty' is not loaded}%
    \renewcommand\color[2][]{}%
  }%
  \providecommand\transparent[1]{%
    \errmessage{(Inkscape) Transparency is used (non-zero) for the text in Inkscape, but the package 'transparent.sty' is not loaded}%
    \renewcommand\transparent[1]{}%
  }%
  \providecommand\rotatebox[2]{#2}%
  \newcommand*\fsize{\dimexpr\f@size pt\relax}%
  \newcommand*\lineheight[1]{\fontsize{\fsize}{#1\fsize}\selectfont}%
  \ifx\svgwidth\undefined%
    \setlength{\unitlength}{792bp}%
    \ifx\svgscale\undefined%
      \relax%
    \else%
      \setlength{\unitlength}{\unitlength * \real{\svgscale}}%
    \fi%
  \else%
    \setlength{\unitlength}{\svgwidth}%
  \fi%
  \global\let\svgwidth\undefined%
  \global\let\svgscale\undefined%
  \makeatother%
  \begin{picture}(1,0.77272727)%
    \lineheight{1}%
    \setlength\tabcolsep{0pt}%
    \put(0,0){\includegraphics[width=\unitlength,page=1]{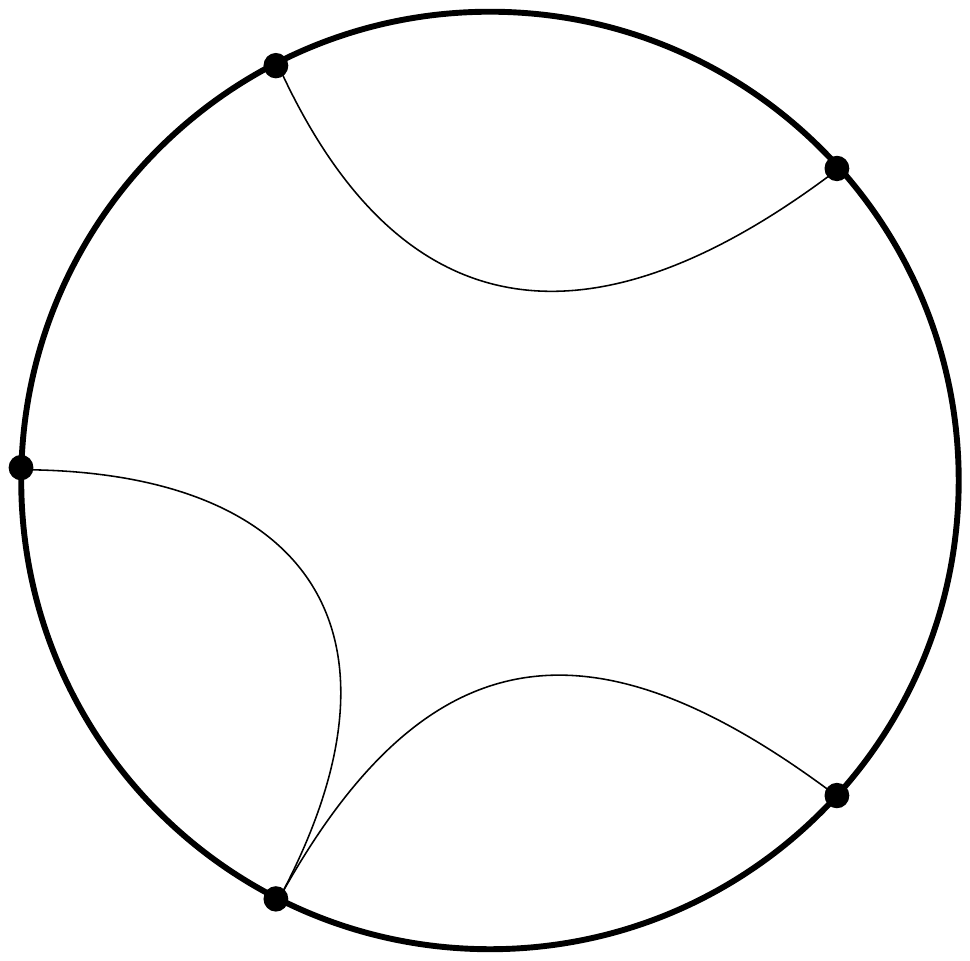}}%
    \put(0,0){\includegraphics[width=\unitlength,page=2]{L2.pdf}}%
    
    \put(0.05114024,0.35330933){\color[rgb]{0,0,0}\makebox(0,0)[lt]{\lineheight{1.25}\smash{\begin{tabular}[t]{l}{$A_{40}$}\end{tabular}}}}%
    \put(0.30211838,0.04025614){\color[rgb]{0,0,0}\makebox(0,0)[lt]{\lineheight{1.25}\smash{\begin{tabular}[t]{l}{$A_{01}$}\end{tabular}}}}%
    \put(0.69753348,0.12913013){\color[rgb]{0,0,0}\makebox(0,0)[lt]{\lineheight{1.25}\smash{\begin{tabular}[t]{l}{$A_{12}$}\end{tabular}}}}%
    \put(0.70327295,0.59150604){\color[rgb]{0,0,0}\makebox(0,0)[lt]{\lineheight{1.25}\smash{\begin{tabular}[t]{l}{$A_{23}$}\end{tabular}}}}%
    \put(0.25520522,0.68229325){\color[rgb]{0,0,0}\makebox(0,0)[lt]{\lineheight{1.25}\smash{\begin{tabular}[t]{l}{$A_{34}$}\end{tabular}}}}%

  \end{picture}%
\endgroup%

			\def\svgwidth{0.3\textwidth}
\begingroup%
  \makeatletter%
  \providecommand\color[2][]{%
    \errmessage{(Inkscape) Color is used for the text in Inkscape, but the package 'color.sty' is not loaded}%
    \renewcommand\color[2][]{}%
  }%
  \providecommand\transparent[1]{%
    \errmessage{(Inkscape) Transparency is used (non-zero) for the text in Inkscape, but the package 'transparent.sty' is not loaded}%
    \renewcommand\transparent[1]{}%
  }%
  \providecommand\rotatebox[2]{#2}%
  \newcommand*\fsize{\dimexpr\f@size pt\relax}%
  \newcommand*\lineheight[1]{\fontsize{\fsize}{#1\fsize}\selectfont}%
  \ifx\svgwidth\undefined%
    \setlength{\unitlength}{792bp}%
    \ifx\svgscale\undefined%
      \relax%
    \else%
      \setlength{\unitlength}{\unitlength * \real{\svgscale}}%
    \fi%
  \else%
    \setlength{\unitlength}{\svgwidth}%
  \fi%
  \global\let\svgwidth\undefined%
  \global\let\svgscale\undefined%
  \makeatother%
  \begin{picture}(1,0.77272727)%
    \lineheight{1}%
    \setlength\tabcolsep{0pt}%
    \put(0,0){\includegraphics[width=\unitlength,page=1]{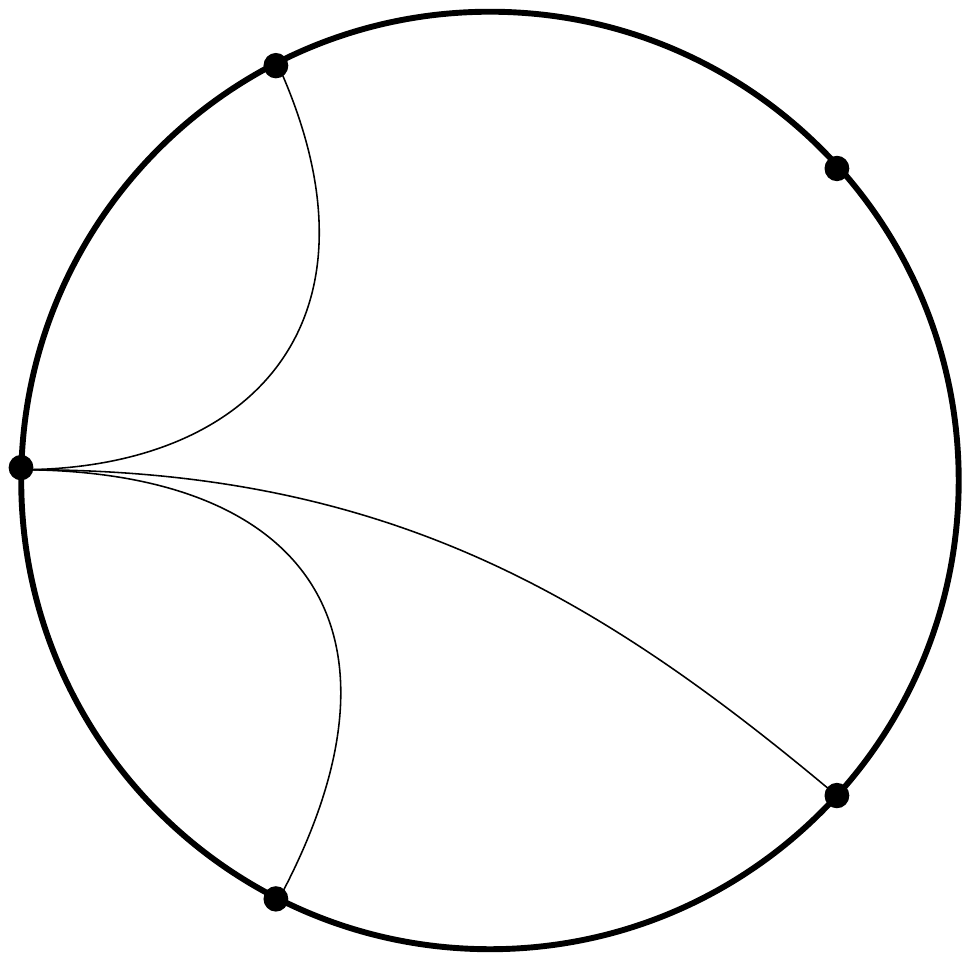}}%
    \put(0,0){\includegraphics[width=\unitlength,page=2]{L3.pdf}}%
    
    \put(0.05114024,0.35330933){\color[rgb]{0,0,0}\makebox(0,0)[lt]{\lineheight{1.25}\smash{\begin{tabular}[t]{l}{$A_{40}$}\end{tabular}}}}%
	\put(0.30211838,0.04025614){\color[rgb]{0,0,0}\makebox(0,0)[lt]{\lineheight{1.25}\smash{\begin{tabular}[t]{l}{$A_{01}$}\end{tabular}}}}%
	\put(0.69753348,0.12913013){\color[rgb]{0,0,0}\makebox(0,0)[lt]{\lineheight{1.25}\smash{\begin{tabular}[t]{l}{$A_{12}$}\end{tabular}}}}%
	\put(0.70327295,0.59150604){\color[rgb]{0,0,0}\makebox(0,0)[lt]{\lineheight{1.25}\smash{\begin{tabular}[t]{l}{$A_{23}$}\end{tabular}}}}%
	\put(0.25520522,0.68229325){\color[rgb]{0,0,0}\makebox(0,0)[lt]{\lineheight{1.25}\smash{\begin{tabular}[t]{l}{$A_{34}$}\end{tabular}}}}%

  \end{picture}%
\endgroup%

		}
		\caption{The corresponding dual laminations generating the same equivalence relations on $\mathbb{S}^1$.}
	\end{subfigure}
	\caption{The split modification and dual laminations.}
	\label{fig:M}
\end{figure}


The dynamics are modified in $S$ so that each edge of $\tilde{S}$ is fixed.
The local degree function $\delta$ is defined to be the same as it was before the modification.
The angle function at $a_i$ is modified with the following rule (see Figure \ref{fig:M}):
\begin{itemize}
\item If $a_ia_j$ is an edge where $a_j$ is closer to $a_0$ than $a_i$, we set the angle of the tangent direction corresponding to $a_j$ to be $0$.
\item If $a_ia_j$ is an edge where $a_j$ is further to $a_0$ than $a_i$, we set the angle of the tangent direction corresponding to $a_j$ to be $0^+$ if $j<i$ and $0^-$ if $j>i$.
\item The other angles remain the same.
\end{itemize}
We also modify $H$ on the backward orbits of vertices in $S$ by pullback (see \cite[\S 6]{L21a} for more details).

After such modifications for all periodic Julia branch points and their backward orbits, we obtain an angled forest map that can be easily demonstrated to be admissible.

\subsection*{Degenerating to a root of $\mathcal{H}_f$}
We are ready to prove Theorem \ref{prop:gbf}.
\begin{proof}[Proof of Theorem \ref{prop:gbf}]
Let $H$ be a simplicial Hubbard forest modeled on $\mathcal{S}$ associated to the simplicial tuning $f$ of $g$.
Let $(F: (\mathcal{T}, \mathcal{P}) \rightarrow (\mathcal{T}, \mathcal{P}), \delta, \alpha)$ be an admissible splitting of the simplicial Hubbard forest $H$, constructed as above.
Then by Theorem \ref{thm:rs}, there exists a quasi \pcf sequence $\bp_n \in \BP^\mathcal{S}$ realizing this angled forest map.
This sequence $\bp_n$ corresponds to a sequence of polynomials
$g_n \in \mathcal{H}_g$.
Using the same argument as in \cite[\S 6]{L21a}, we can prove that after passing to a subsequence, $g_n$ converges to a geometrically finite polynomial $g_\infty \in \mathcal{P}_n$ which is a root of $\mathcal{H}_f$.
This proves Theorem \ref{prop:gbf}.
\end{proof}

\begin{proof}[Proof of Theorem \ref{thm:bcbound}]
Suppose $f_0= f,..., f_k(z) = z^d$ where $f_i$ is a simplicial tuning of $f_{i+1}$.
Then $\mathcal{H}_{f_{i+1}}$ bifurcates to $\mathcal{H}_{f_i}$.
Therefore, $\mathscr{C}_{b}(f) \leq \mathscr{C }_{c}(f)$.

To prove the sharpness, we note that if $f \in \mathcal{P}_d$ is any \pcf polynomial so that $\mathscr{C }_{c}(f) = 1$ and $f(z) \neq z^d$, then $\mathscr{C}_{b}(f) = \mathscr{C}_{c}(f) = 1$.
\end{proof}

\section{Maps in $\Mol_d$ have core entropy zero}\label{subsec:miz}
In this section, we show that any \pcf polynomial in $\Mol_d$ has core entropy zero.
\begin{theorem}\label{prop:miz}
Let $f$ be a \pcf polynomial in the main molecule $\Mol_d$. Then $f$ has core entropy zero.
\end{theorem}

Recall that $\Mol_d = \overline{\bigcup_{\mathcal{H}\in \mathfrak{S}} \mathcal{H}}$.
We will breakdown the proof into two steps.
\begin{enumerate}
\item Firstly, we show that if $f$ is a \pcf polynomial in $\bigcup_{\mathcal{H}\in \mathfrak{S}} \mathcal{H}$, then $f$ has core entropy zero (see \S \ref{subsubsec:b}).
This is proved by induction.
\item In the second step, we show that if $f$ is a \pcf polynomial in $\Mol_d - \bigcup_{\mathcal{H}\in \mathfrak{S}} \mathcal{H}$, then $f$ has core entropy zero (see \S \ref{subsubsec:a}).
\end{enumerate}

We also established the following inequality (cf. Theorem \ref{thm:bcbound}).
\begin{theorem}\label{thm:bcboundr}
Let $f$ be a \pcf polynomial in the main molecule $\Mol_d$.
Then
$$
\mathscr{C}_{c}(f) \leq (d-1)\mathscr{C}_{b}(f).
$$
Moreover, the bound is sharp, i.e., there exists a degree $d$ \pcf polynomial $f$ with $\mathscr{C}_{c}(f) = (d-1)\mathscr{C}_{b}(f)$.
\end{theorem}

We remark that a posteriori, our proof gives rise to the following theorem.

\begin{theorem}\label{thm:npm}
There are no \pcf polynomials in 
$$
\Mol_d - \bigcup_{\mathcal{H}\in \mathfrak{S}} \mathcal{H}.
$$
\end{theorem}
\begin{proof}
Let $f \in \Mol_d$ be a \pcf polynomial. 
By Theorem \ref{prop:miz}, $f$ has core entropy zero.
By Theorem \ref{prop:fb} and Theorem \ref{prop:gbf}, $\mathcal{H}_f$ is obtained from the main hyperbolic component $\mathcal{H}_d$ through a finite chain of bifurcations.
Thus, $f \in\bigcup_{\mathcal{H}\in \mathfrak{S}} \mathcal{H}$.
\end{proof}

\subsection{Polynomials with positive core entropy}
In this subsection, we investigate properties of polynomials with positive core entropy.

\begin{lem}\label{lem:ld}
	Let $f \in \mathcal{P}_d$ be a \pcf polynomial with positive core entropy and $\mathcal{T}_f$ denote the Hubbard tree of $f$.
	Then there exist two integers $k_1, k_2$ and an edge $E \subseteq \mathcal{T}_f$ containing two subintervals $E^1, E^2$ with disjoint interiors so that for $i = 1,2$,
	$$
	f^{k_i}: E^i \longrightarrow E
	$$
	is a homeomorphism.
\end{lem}
\begin{proof}
	By Theorem \ref{thm:nic}, the directed graph $\mathcal{G}_f$ has intersecting cycles.
	Let $C^1: a_1^1 \to a_2^1 \to \cdots \to a_{k_1}^1 \to a_1^1$ and $C^2: a_1^2 \to a_2^2 \to \cdots \to a_{k_2}^2 \to a_1^2$ be two intersecting simple cycles with $a_1^1 = a_1^2$. Recall that paths of length $n$ in $\mathcal{G}_f$ are in 1-1 correspondence with level-$n$ subedges of $\mathcal{T}_f$ (see the proof of Proposition \ref{prop:PathtoJulia}).
	Let $E \subseteq \mathcal{T}_f$ be the corresponding edge for $a_1^1$. Then each cycle $C^i$ corresponds to the desired subinterval $E^i$. See Figure \ref{fig:aP} for an example.
\end{proof}

\begin{figure}[ht]
  \centering
  \resizebox{0.45\linewidth}{!}{
    \def\svgwidth{\columnwidth}
\begingroup%
  \makeatletter%
  \providecommand\color[2][]{%
    \errmessage{(Inkscape) Color is used for the text in Inkscape, but the package 'color.sty' is not loaded}%
    \renewcommand\color[2][]{}%
  }%
  \providecommand\transparent[1]{%
    \errmessage{(Inkscape) Transparency is used (non-zero) for the text in Inkscape, but the package 'transparent.sty' is not loaded}%
    \renewcommand\transparent[1]{}%
  }%
  \providecommand\rotatebox[2]{#2}%
  \newcommand*\fsize{\dimexpr\f@size pt\relax}%
  \newcommand*\lineheight[1]{\fontsize{\fsize}{#1\fsize}\selectfont}%
  \ifx\svgwidth\undefined%
    \setlength{\unitlength}{792bp}%
    \ifx\svgscale\undefined%
      \relax%
    \else%
      \setlength{\unitlength}{\unitlength * \real{\svgscale}}%
    \fi%
  \else%
    \setlength{\unitlength}{\svgwidth}%
  \fi%
  \global\let\svgwidth\undefined%
  \global\let\svgscale\undefined%
  \makeatother%
  \begin{picture}(1,0.31818182)%
    \lineheight{1}%
    \setlength\tabcolsep{0pt}%
    \put(0,0){\includegraphics[width=\unitlength,page=1]{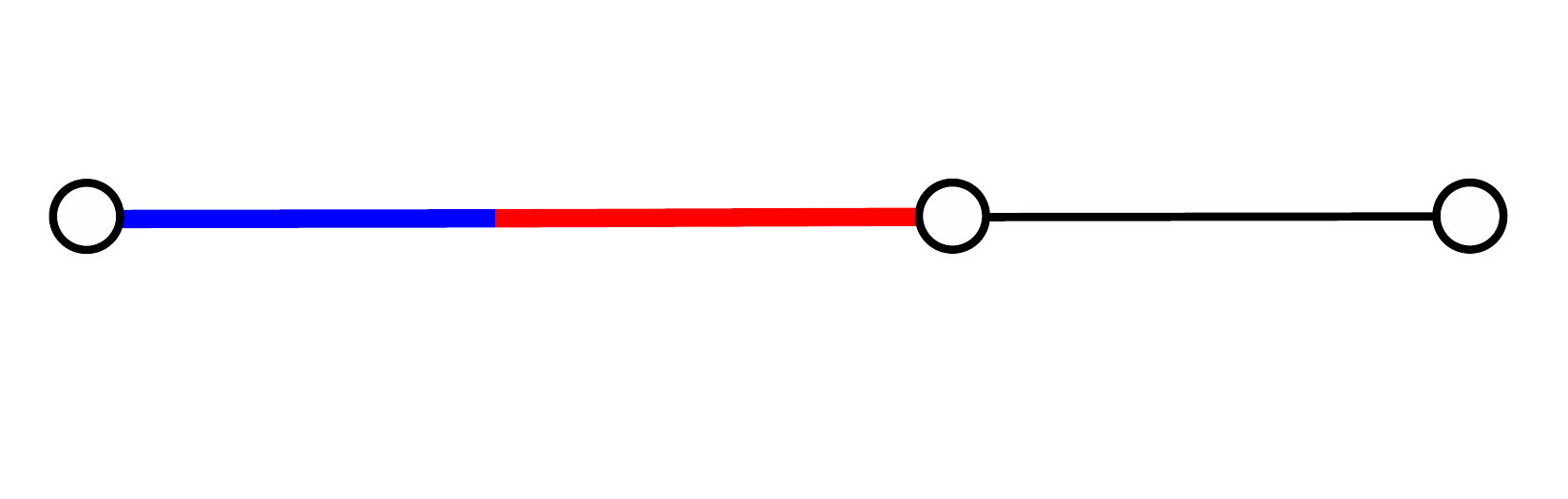}}%
    \put(0.29439709,0.12784085){\color[rgb]{0,0,0}\makebox(0,0)[lt]{\lineheight{1.25}\smash{\begin{tabular}[t]{l}{\Large$E_1$}\end{tabular}}}}%
    \put(0.76179504,0.13055118){\color[rgb]{0,0,0}\makebox(0,0)[lt]{\lineheight{1.25}\smash{\begin{tabular}[t]{l}{\Large$E_2$}\end{tabular}}}}%
  \end{picture}%
\endgroup%

  }
  \caption{The Hubbard tree of the airplane polynomial $f$. Note that  $f(E_1) = E_1 \cup E_2$ and $f(E_2) = E_1$. Let $a_i$ be the corresponding vertices in $\mathcal{G}_f$. We have intersecting cycles $a_1 \to a_1$ and $a_1 \to a_2 \to a_1$. The blue and red subintervals are mapped homeomorphically to $E_1$ by $f^2$ and $f$.}
  \label{fig:aP}
\end{figure}

\subsection*{Lamination and rational lamination}
A {\em lamination} $\lambda \subseteq \mathbb{S}^1 \times \mathbb{S}^1$ is an equivalence relation on the circle $\mathbb{S}^1=\R/\Z$ such that the convex hulls of distinct equivalence classes are disjoint.
We consider convex hulls of equivalence classes using the hyperbolic metric on $\D$.
A {\em leaf} $l$ of $\lambda$ is a pair $(t,t') \in \lambda$ with $t \neq t'$.
We shall identify a leaf $l$ with the hyperbolic geodesic in $\D$ connecting $t$ and $t'$.

Suppose $f\in \mathcal{P}_d$ has connected Julia set.
The {\em rational lamination of $f$}, denoted by $\lambda_{\Q}(f)$, is defined by $t\sim t'$ if $t = t'$ or if $t$ and $t'$ are rational and the external rays $R_t$ and $R_{t'}$ land at the same point in the Julia set \cite[\S 6.4]{McM94}.
When $f$ has locally connected Julia set, the {\em lamination of $f$}, denoted by $\lambda(f)$, is defined by $t\sim t'$ if the external rays $R_t$ and $R_{t'}$ land at the same point in the Julia set.
Note that in this case, $\lambda(f)$ is the closure of $\lambda_\Q(f)$, and provides a topological model of the Julia set in the sense that $\mathcal{J}_f = \mathbb{S}^1 / \lambda(f)$.

Note that the non-escaping set of the system $f^{k_1}|_{E^1}$ and $f^{k_2}|_{E^2}$ in Lemma \ref{lem:ld} contains a Cantor set with a dense subset of repelling periodic points of $f$.
Thus, we have:
\begin{lem}\label{lem:ul}
Let $f \in \mathcal{P}_d$ be a \pcf polynomial with positive core entropy.
The lamination $\lambda(f)$ contains infinitely many periodic leaves $\{l_i: i \in \N\}$ so that
$\overline{\bigcup_i l_i}$ contains uncountably many leaves.
\end{lem}

\subsection*{Perturbation of polynomials with positive core entropy}
Let $f \in \mathcal{P}_d$ be a \pcf polynomial with $h(f) >0$.
Then by Lemma \ref{lem:ld}, we have an edge $E \subseteq \mathcal{T}_f$ and two subintervals $E^1, E^2$ with disjoint interiors so that
$$
	f^{k_i}: E^i \longrightarrow E
$$
is a homeomorphism for $i=1,2$.
By taking $f^{2k_i}: f^{-k_i}|_{E \to E^i}(E^i) \longrightarrow E$ if necessary, we may assume that the homeomorphism is orientation preserving.
Similarly, by taking some further iterates if necessary, we may also assume that $E^1, E^2$ are disjoint and contained in $\Int(E)$.
Let $x_i \in E^i$ be a fixed point of $f^{k_i}$.
Let $I = [x_1, x_2] \subseteq \Int(E)$, and $I_i:= f^{-k_i}|_{E \to E^i}(I) \subseteq I$.
Denote $I_1 = [x_1, y_1]$ and $I_2 = [y_2, x_2]$. 
Note that $x_i$ is a repelling periodic point.
Since $x_i$ is contained in the interior of $E$, the orbit of $y_i$ avoids critical points of $f$.

\begin{figure}[ht]
	\centering
	\resizebox{0.9\linewidth}{!}{
    \def\svgwidth{\columnwidth}
    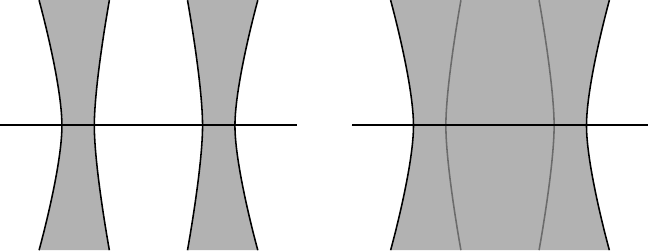

	}
	\caption{}
	\label{fig:CantorLami}
\end{figure}

Let $\mathcal{R}(\theta^\pm_{x_i})$ and $\mathcal{R}(\theta^\pm_{y_i})$ be the {\em left} and {\em right supporting external rays} of $I_i$ landing at $x_i$ and $y_i$ respectively. This means that $\mathcal{R}(\theta^\pm_{x_i})$ are external rays landing at $x_i$, and the component of $\C - \overline{\mathcal{R}(\theta^+_{x_i})} \cup \overline{\mathcal{R}(\theta^-_{x_i})}$ containing $I_i$ contains no other external rays landing at $x_i$.
Let $\mathcal{U}_I$ be the component of 
$$
\C - \overline{\mathcal{R}(\theta^+_{x_1}) \cup \mathcal{R}(\theta^-_{x_1})} -  \overline{\mathcal{R}(\theta^+_{x_2}) \cup \mathcal{R}(\theta^-_{x_2})}
$$
that contains $\Int(I)$.
Similarly, let $\mathcal{U}_{I_i}$ be the component of 
$$
\C - \overline{\mathcal{R}(\theta^+_{x_i}) \cup \mathcal{R}(\theta^-_{x_i})} -  \overline{\mathcal{R}(\theta^+_{y_i}) \cup \mathcal{R}(\theta^-_{y_i})}
$$
that contains $\Int(I_i)$, $i=1,2$.
Note that 
$$
f^{k_i}: \mathcal{U}_{I_i} \longrightarrow \mathcal{U}_I
$$
is a conformal map. See Figure \ref{fig:CantorLami}.

For any polynomial $g$ with connected Julia set that is sufficiently close to $f$, the corresponding external rays $\mathcal{R}_g(\theta^\pm_{x_i})$ and $\mathcal{R}_g(\theta^\pm_{y_i})$ for $g$ land at the continuation of  periodic or pre-periodic points $x_{i,g}$ and $y_{i,g}$ for $g$ respectively (\cite[Appendix B]{GM93}). 
Denote $\mathcal{U}_{I_i, g}$ and $\mathcal{U}_{I, g}$ be the corresponding sets for $g$.
Then
$$
g^{k_i}: \mathcal{U}_{I_i, g} \longrightarrow \mathcal{U}_{I,g}
$$
is a conformal map.
By inductively taking inverse images of $\overline{\mathcal{R}_g(\theta^+_{x_i}) \cup \mathcal{R}_g(\theta^-_{x_i})}$ and $\overline{\mathcal{R}_g(\theta^+_{y_i}) \cup \mathcal{R}_g(\theta^-_{y_i})}$ under the two maps $g^{k_1}$ and $g^{k_2}$, we obtain infinitely many leaves separating $x_{i,g}$, and thus the post-critical set of $g$.
Suppose $g$ has locally connected Julia set. 
Then by taking closure, we have the following:

\begin{lem}\label{lem:uncl}
Let $f \in \mathcal{P}_d$ be a \pcf polynomial with positive core entropy.
Then for any polynomial $g \in \mathcal{P}_d$ with connected and locally connected Julia set that is sufficiently close to $f$, the lamination $\lambda(g)$ contains uncountably many leaves.
\end{lem}

\subsection{\Pcf polynomials via adjacency}\label{subsubsec:b}
\begin{prop}\label{prop:icez}
Let $f \in \bigcup_{\mathcal{H}\in \mathfrak{S}} \mathcal{H}$ be a \pcf polynomial.
Then $f$ has core entropy zero.
\end{prop}

Proposition \ref{prop:icez} immediately follows from the following proposition.

\begin{prop}\label{prop:bzce}
Let $f, g$ be \pcf polynomials so that $\mathcal{H}_f$ and $\mathcal{H}_g$ are adjacent.
Then $f$ has core entropy zero if and only if $g$ has core entropy zero.
\end{prop}
\begin{proof}
Suppose that $f$ has positive core entropy.
Let $h \in \overline{\mathcal{H}_f} \cap \overline{\mathcal{H}_g}$.
Let $h = \lim f_n$ with $f_n \in \mathcal{H}_f$ and let $\{l_i: i \in \N\}$ be the collection of leaves in $\lambda(f) (= \lambda(f_n))$ in Lemma \ref{lem:ul}.
Since there are only finitely many parabolic periodic points for $h$, by \cite[Lemma B.4]{GM93}, the rational lamination $\lambda_{\Q}(h)$ contains all but finitely many leaves in $\{l_i: i \in \N\}$.
Since these leaves land on repelling periodic points, the lamination $\lambda(g) (= \lambda(g_n))$ also contains all but finitely many leaves in $\{l_i: i \in \N\}$ where $g_n \in \mathcal{H}_g$ with $g_n \to h$.
Since $\overline{\bigcup_i l_i}$ contains uncountably many leaves, we conclude that the Julia set $\mathcal{J}_g \cap \mathcal{T}_g$ is uncountable.
Thus, $g$ has positive core entropy by Theorem \ref{thm:cs}.
Then the proposition follows by symmetry.
\end{proof}

\subsection{\Pcf polynomials via accumulation}\label{subsubsec:a}
\begin{prop}\label{prop:azce}
Let $f$ be a \pcf polynomial.
Suppose that there exists a sequence $f_n \in \bigcup_{\mathcal{H}\in \mathfrak{S}} \mathcal{H}$ so that $f_n \to f$.
Then $f$ has core entropy zero.
\end{prop}
We remark that if the sequence $f_n$ can be chosen to be post-critically finite, then the statement immediately follows from \cite[Theorem 1.2]{GT21} as the core entropy is a continuous function on the space of \pcf polynomials of degree $d$.
We do not know if this can always be achieved, as we do not know if the diameters of relative hyperbolic components are shrinking to zero.
\begin{proof}[Proof of Proposition \ref{prop:azce}]
Suppose for contradiction that $h(f) >0$.
By Lemma \ref{lem:uncl}, for all sufficiently large $n$, there are uncountably many leaves separating the post-critical set of $f_n$.
Let $g_n$ be the \pcf center of the \shc that contains $f_n$.
Then the lamination of $g_n$ is the same as the lamination of $f_n$.
Thus, the Hubbard tree $g_n$ intersects the Julia set in an uncountable set.
Therefore $h(g_n) >0$ by Theorem \ref{thm:cs}, which is a contradiction to Proposition \ref{prop:icez}.
\end{proof}

\begin{proof}[Proof of Theorem \ref{prop:miz}]
The theorem follows from Proposition \ref{prop:icez} and Proposition \ref{prop:azce}.
\end{proof}

\subsection{Bounds on the complexity}\label{subsec:bc}
In this subsection, we shall use the analysis in \cite{L21a} to prove $\mathscr{C}_{c}(f) \leq (d-1)\mathscr{C}_{b}(f)$.
\subsection*{Pointed simplicial tuning}
We use a notion of {\em pointed iterated simplicial Hubbard trees} which was introduced in \cite{L21a}.
We define a {\em pointed Hubbard tree} $(\mathcal{T}_f, p)$ as a Hubbard tree $\mathcal{T}_f$ together with a fixed point $p \in \mathcal{T}_f$.
Intuitively, this fixed point keeps track of where the attracting fixed point goes to.

A pointed Hubbard tree $(\mathcal{T}_g, q)$ is a {\em pointed simplicial tuning} of $(\mathcal{T}_f, p)$ if 
\begin{itemize}
\item $\mathcal{T}_g$ is the Hubbard tree of the \pcf polynomial $g$, which is obtained by a simplicial tuning of $f$ on the Fatou components associated to $p$ and its backward orbits, and
\item $q$ is a fixed point on the new tuned-in tree for $p$.
\end{itemize}
Let $\mult_{go}(p)$ be the sum of multiplicities of critical points in the grand orbit of $p$.
Note that if $\mathcal{T}_f$ is a degree $d$ Hubbard tree, then $\mult_{go}(p) \leq d-1$.
One can verify that if $(\mathcal{T}_g, q)$ is a pointed simplicial tuning of $(\mathcal{T}_f, p)$, then $\mult_{go}(q) \leq \mult_{go}(p)-1$.
Thus, we have:
\begin{lem}\label{lem:fs}
Let $(\mathcal{T}_0, p_0),..., (\mathcal{T}_k, p_k)$ be a sequence of pointed Hubbard trees of degree $d$ so that $(\mathcal{T}_{i+1}, p_{i+1})$ is a pointed simplicial tuning of $(\mathcal{T}_{i}, p_{i})$.
Then $k\leq d-1$.
\end{lem}

A pointed Hubbard tree is called {\em pointed iterated-simplicial} if it is obtained from the trivial pointed Hubbard tree $(\mathcal{T}_f=\{p\},p)$ associated to $z^d$ via a (necessarily finite) sequence of pointed simplicial tunings.

Let $g$ be a \pcf polynomial so that the main hyperbolic component $\mathcal{H}_d$ bifurcates to $\mathcal{H}_g$.
Then it follows from \cite[Theorem 1.1]{L21a} that there exists a fixed point $p$ on the Hubbard tree $\mathcal{T}_g$ of the \pcf polynomial $g$ so that the pointed Hubbard tree $(\mathcal{T}_g, p)$ is pointed iterated-simplicial.

\subsection*{Bifurcation}
More generally, let $f, g$ be \pcf polynomials so that $\mathcal{H}_f$ bifurcates to $\mathcal{H}_g$.
Let $U$ be a critical or post-critical Fatou component of $f$.
Let $\mathcal{T}_U = \mathcal{T}_f \cap U$.
Let $\{x_1,..., x_k\} \in \partial \mathcal{T}_U$.
Note that by definition (see Definition \ref{defn:bif}), we have $\lambda(f) \subseteq \lambda(g)$.
Therefore, there exists a corresponding point $y_i \in \mathcal{J}_g$ so that the angles landing at $x_i$ also land at $y_i$.
By taking the regulated hull of $\{y_1,..., y_k\}$, we obtain a subtree, denoted by $\mathcal{T}_{g,U}$, of the Hubbard tree $\mathcal{T}_g$.

Suppose $U$ is periodic with period $q$, then $\mathcal{T}_{g,U}$ is invariant under $g^q$. It can be thus viewed abstractly as a Hubbard tree with degree $\deg (f^q: U \longrightarrow U)$.
Now using the same argument as in \cite[\S 5]{L21a}, we have:

\begin{prop}\label{prop:is}
Let $f, g$ be \pcf polynomials so that $\mathcal{H}_f$ bifurcates to $\mathcal{H}_g$. Let $U$ be a periodic Fatou component of $f$ with period $q$.
Then there exists $p \in \mathcal{T}_{g,U}$ that is a fixed point of $g^q$ so that the pointed Hubbard tree $(\mathcal{T}_{g,U}, p)$ is pointed iterated-simplicial.
\end{prop}

\begin{proof}[Proof of Theorem \ref{thm:bcboundr}]
Let $f, g\in\mathcal{P}_d$ be \pcf polynomials.
Suppose that $\mathcal{H}_f$ bifurcates to $\mathcal{H}_g$.
Then by Proposition \ref{prop:is} and Lemma \ref{lem:fs}, $g$ is obtained from $f$ via a sequence of at most $d-1$ simplicial tunings.
Thus, $\mathscr{C}_{c}(f) \leq (d-1)\mathscr{C}_{b}(f)$.

To prove the sharpness of the bound, one can construct a sequence of pointed iterated simplicial tunings of the trivial pointed Hubbard tree with the pattern illustrated as in Figure \ref{fig:EV} for degree $4$.
Let $f$ be the \pcf polynomial associated to the last underlying Hubbard tree in the sequence, where the marked fixed point has local degree $1$.
It follows from \cite[Corollary 1.2]{L21a} that the main hyperbolic component $\mathcal{H}_d$ bifurcates to $\mathcal{H}_f$. Thus, $\mathscr{C}_b(f) =1$. On the other hand, it is easy to verify that $\mathscr{C}_c(f) = \mathscr{C}_{t}(f) = d-1$ (see Figure \ref{fig:DG}).
\end{proof}

\begin{figure}[ht]
  \centering
  \resizebox{0.75\linewidth}{!}{
    \def\svgwidth{\columnwidth}
    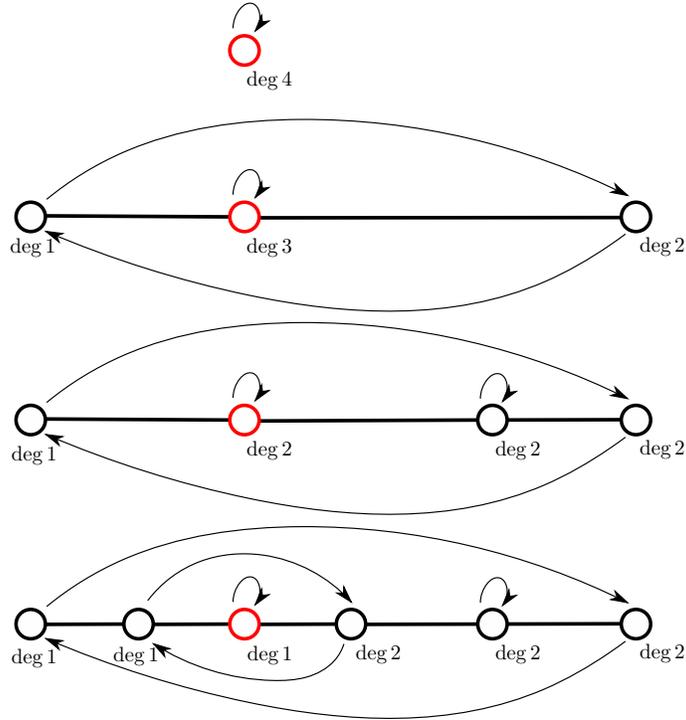

  }
  \caption{A sequence of pointed iterated simplicial tunings of the trivial pointed Hubbard tree, where the red vertex corresponds to the marked point.}
  \label{fig:EV}
\end{figure}

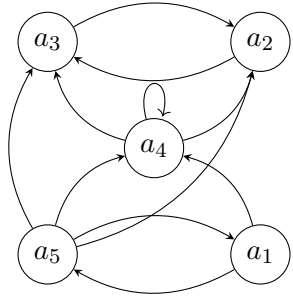
\begin{figure}[ht]
  \centering
  \begin{tikzpicture}[node distance=2
		cm,gnode/.style={circle,draw,font=\bfseries}]
		\node[gnode](3) {$a_3$};
		\node[gnode](4) [below right of=3] {$a_4$};
		\node[gnode](2) [above right of=4] {$a_2$};
		\node[gnode](5) [below left of=4] {$a_5$};
		\node[gnode](1) [below right of=4] {$a_1$};

		\path[-stealth]
		(3) edge [bend left] node {} (2)
		(2) edge [bend left] node {} (3)
		(4) edge [bend left] node {} (3)
		(4) edge [bend right] node {} (2)
		(4) edge [loop above] node {} ()
		(5) edge [bend left] node {} (1)
		(1) edge [bend left] node {} (5)
		(5) edge [bend left] node {} (4)
		(1) edge [bend right] node {} (4)
		(5) edge [bend left] node {} (3)
		(5) edge [bend right] node {} (2);
		
	\end{tikzpicture}
  \caption{The directed graph associated to the bottom Hubbard tree in Figure \ref{fig:EV}. The vertices are labeled so that $a_i$ corresponds to the $i$-th edge of the Hubbard tree from the left. It can be verified easily the maximal depth is $2$, so $\mathscr{C}_t(f) = 3$ by Theorem \ref{thm:CBDecompJuliaIntersectGraph}. }
  \label{fig:DG}
\end{figure}

\appendix
\section{Bisets and relative polynomial activity growth}\label{sec:biset}
In this appendix, we will give another characterization of core entropy zero polynomials using the theory of self-similar groups. This idea was suggested by L.\@ Bartholdi and V.\@ Nekrashevych.

\begin{theorem}\label{thm:mat}
	For a \pcf polynomial of degree $d>1$, $h(f)=0$ if and only if the automaton about the adding machine basis has polynomial activity growth.
\end{theorem}

\subsection*{Bisets of \pcf topological branched coverings}
We briefly summarize the notion of self-similar groups and bisets. See \cite{BD18} \cite{Nek05} for more details.

Let $f:(S^2,P_f)\righttoleftarrow$ be a \pcf topological branched covering of degree $d>1$ with the post-critical set $P_f$. We consider $(S^2,P_f)$ as an orbifold whose orbifold structure is given by the order function $\ord: P_f \to \{2,3,\dots\} \cup \{\infty\}$ satisfying 
\[
\ord(z)=lcm\{\deg_{f^k}(y)~|~y\in f^{-k}(z),~k\ge0\}.
\]
When $z\in P_f$ is in the preimage of a periodic critical point, $\ord(z)=\infty$. Fix a base point $x$ of the orbifold fundamental group $G:=\pi_1^{orb}(S^2\setminus P_f,x)$. Define a set
\[
B(f)=\{ \gamma:[0,1]\to S^2\setminus P_f~|~\gamma(0)=x,\gamma(1)\in f^{-1}(x)\}/\sim_{\mathrm{orbi.homotopy}}
\]
where the homotopy is relative to the endpoints with $(S^2,P_f)$ being considered as an orbifold. The set $B(f)$ is called the {\it biset} of $f$ because it has $G$-actions from both the left and right. More precisely, for $\alpha,\beta\in G$ and $\delta\in B(f)$, we define $\alpha\cdot \delta \cdot \beta$ as the concatenation of $\beta$, $\delta$, and the lift of $\alpha$ through $f$ that starts from the terminal point of $\delta$. We remark that the concatenation is performed from right to left. It easily follows from definition that the left $G$-action is transitive and the right $G$-action is free on $B(f)$.

Since $B(f)$ has $G$-actions from both left and right, we can define its tensor powers $B(f)^{\otimes n}$. For example, $B(f)\otimes B(f)$ is the quotient of $B(f) \times B(f)$ by $\delta \cdot \alpha \times \delta' = \delta \times \alpha \cdot \delta'$ for any $\delta,\delta'\in B(f)$ and $\alpha \in G$. The $G$-actions can be extended to the tensor powers $B(f)^{\otimes n}$ in a natural way. For example, for $\delta, \delta'\in B(f)$, $\delta \otimes \delta'$ is the concatenation of $\delta'$ and the lift of $\delta$ starting from the terminal point of $\delta'$ so that $\delta \cdot \alpha \otimes \delta'=\delta \otimes \alpha \cdot \delta'$ for any $\alpha \in G$; both are equal to the concatenation of (1) $\delta'$ followed by (2) the lifting $\widetilde{\alpha}$ of $\alpha$ that starts from the terminal point of $\delta$ followed by (3) the lifting of $\delta$ starting from the terminal point of $\widetilde{\alpha}$. Again, the left $G$-action is transitive and the right $G$-action is free on $B(f)^{\otimes n}$.

Note that the right action cannot change the terminal point of $\delta \in B(f)$. So the right action has exactly $d$ orbits, which correspond to the terminal points of $\delta \in B(f)$. We call a choice of representatives of the $d$ orbits a {\it basis} of $B(f)$. For a basis $X$ of $B(f)$, $X^{\otimes n}$ is a basis of $B(f)^{\otimes n}$ for any $n\ge1$.

Let $X:=\{\delta_0,\delta_1,\dots,\delta_{d-1}\}$ be a basis of $B(f)$. We have $X\cdot G=B(f)$. More precisely, for any $g\in G$ and $\delta_i\in X$, there exist unique $j$ and $h\in G$ so that
\[
g\cdot \delta_i=\delta_j \cdot h.
\]
We define $g|_{\delta_i}:=h$.
Similarly, for any $\delta_{i_1} \otimes \delta_{i_2} \otimes \dots \otimes \delta_{i_k}$ and any $g\in G$, there exist unique $j_1,j_2,\dots,j_k$ and $h\in G$ so that 
\[
g\cdot (\delta_{i_1} \otimes \delta_{i_2} \otimes \dots \otimes \delta_{i_k})=(\delta_{j_1} \otimes \delta_{j_2} \otimes \dots \otimes \delta_{j_k}) \cdot h.
\]
Again we define $g|_{\delta_{i_1} \otimes \delta_{i_2} \otimes \dots \otimes \delta_{i_k}}:=h$.

A {\it nucleus} is defined as the smallest subset $\Ncal$ of $G$ satisfying the following property: For any $g\in G$ there exists $N>0$ such that $g \cdot X^{\otimes n}\subset X^{\otimes n} \cdot \Ncal$ for any $n>N$. If $f$ is not doubly covered by a torus endomorphism, $B(f)$ has a finite nucleus if and only if $f$ does not have a Levy cycle \cite{BD18}.

In what follows, we suppose that $B(f)$ has a finite nucleus $\mathcal{N}$ about a basis $X$. Define $X^{-\omega}$ as the set of left infinite sequences $\dots \otimes \delta_{i_2}\otimes \delta_{i_1}$. With respect to the product topology, $X^{-\omega}$ is homeomorphic to a Cantor set. We define an equivalence class in such a way that $\dots \otimes \delta_{i_2}\otimes \delta_{i_1} \sim \dots \otimes \delta_{j_2}\otimes \delta_{j_1}$ if and only if for any $n\ge 1$ there exist $g_n,h_n\in \Ncal$ such that
\begin{equation}\label{eqn:EqRel}
g_n\cdot(\delta_{i_n}\otimes \dots \otimes \delta_{i_2}\otimes \delta_{i_1})= (\delta_{j_n}\otimes \dots \otimes \delta_{j_2}\otimes \delta_{j_1})\cdot h_n,
\end{equation}
where $e$ is the identity element of $G$.
Then $(X^{-\omega}/\sim,\sigma)$ is topologically conjugate to $(\Julia_f,f)$ where $\sigma$ is the right shift map \cite{Nek05}.

A {\em Moore diagram} is a directed graph whose vertex set is $\Ncal$ and each edge from $g$ to $h$ corresponds to an equation $g \cdot \delta_i= \delta_j \cdot h$. We label the edge with $(\delta_i,\delta_j)$, which is also a part of the information of the Moore diagram.

\begin{example}\label{eg:BasilicaBiset}
	Let us see the Basilica polynomial $f(z)=z^2+1$ as an example. The post-critical set is $P_f=\{0,-1,\infty\}$. We can choose a generating set $\{g,h\}$ of $\pi_1(\hCbb\setminus P_f)$ and a basis $\{\delta_0,\delta_1\}$ for the biset $B(f)$ as in Figure \ref{fig:BasilicaBiset}. Then the nucleus is $\{g,g^{-1},h,h^{-1},e\}$ and the Moore diagram is given as Figure \ref{fig:BasilicaMooreDiagram}.
	
	\begin{figure}[ht]
		\centering
		\begin{subfigure}[b]{0.45\textwidth}
			\centering
			\def\svgwidth{\textwidth}
\begingroup%
  \makeatletter%
  \providecommand\color[2][]{%
    \errmessage{(Inkscape) Color is used for the text in Inkscape, but the package 'color.sty' is not loaded}%
    \renewcommand\color[2][]{}%
  }%
  \providecommand\transparent[1]{%
    \errmessage{(Inkscape) Transparency is used (non-zero) for the text in Inkscape, but the package 'transparent.sty' is not loaded}%
    \renewcommand\transparent[1]{}%
  }%
  \providecommand\rotatebox[2]{#2}%
  \newcommand*\fsize{\dimexpr\f@size pt\relax}%
  \newcommand*\lineheight[1]{\fontsize{\fsize}{#1\fsize}\selectfont}%
  \ifx\svgwidth\undefined%
    \setlength{\unitlength}{302.05088323bp}%
    \ifx\svgscale\undefined%
      \relax%
    \else%
      \setlength{\unitlength}{\unitlength * \real{\svgscale}}%
    \fi%
  \else%
    \setlength{\unitlength}{\svgwidth}%
  \fi%
  \global\let\svgwidth\undefined%
  \global\let\svgscale\undefined%
  \makeatother%
  \begin{picture}(1,0.53328728)%
    \lineheight{1}%
    \setlength\tabcolsep{0pt}%
    \put(0,0){\includegraphics[width=\unitlength,page=1]{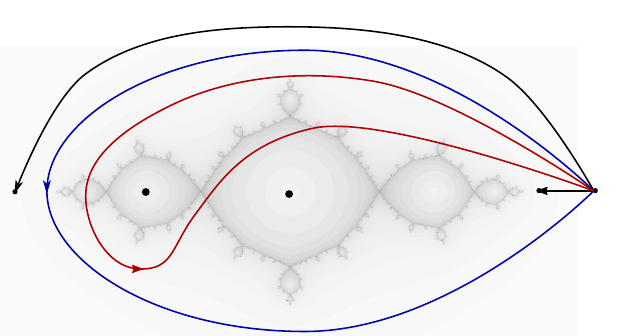}}%
    \put(0.1162838,0.03328511){\color[rgb]{0,0,0}\makebox(0,0)[lt]{\lineheight{1.25}\smash{\begin{tabular}[t]{l}$g$\end{tabular}}}}%
    \put(0.53802626,0.27834204){\color[rgb]{0,0,0}\makebox(0,0)[lt]{\lineheight{1.25}\smash{\begin{tabular}[t]{l}$h$\end{tabular}}}}%
    \put(0.88268162,0.18522776){\color[rgb]{0,0,0}\makebox(0,0)[lt]{\lineheight{1.25}\smash{\begin{tabular}[t]{l}$\delta_0$\end{tabular}}}}%
    \put(0.45928909,0.50358087){\color[rgb]{0,0,0}\makebox(0,0)[lt]{\lineheight{1.25}\smash{\begin{tabular}[t]{l}$\delta_1$\end{tabular}}}}%
  \end{picture}%
\endgroup%

		\end{subfigure}
		\begin{subfigure}[b]{0.45\textwidth}
			\centering
			\def\svgwidth{\textwidth}
\begingroup%
  \makeatletter%
  \providecommand\color[2][]{%
    \errmessage{(Inkscape) Color is used for the text in Inkscape, but the package 'color.sty' is not loaded}%
    \renewcommand\color[2][]{}%
  }%
  \providecommand\transparent[1]{%
    \errmessage{(Inkscape) Transparency is used (non-zero) for the text in Inkscape, but the package 'transparent.sty' is not loaded}%
    \renewcommand\transparent[1]{}%
  }%
  \providecommand\rotatebox[2]{#2}%
  \newcommand*\fsize{\dimexpr\f@size pt\relax}%
  \newcommand*\lineheight[1]{\fontsize{\fsize}{#1\fsize}\selectfont}%
  \ifx\svgwidth\undefined%
    \setlength{\unitlength}{302.05088323bp}%
    \ifx\svgscale\undefined%
      \relax%
    \else%
      \setlength{\unitlength}{\unitlength * \real{\svgscale}}%
    \fi%
  \else%
    \setlength{\unitlength}{\svgwidth}%
  \fi%
  \global\let\svgwidth\undefined%
  \global\let\svgscale\undefined%
  \makeatother%
  \begin{picture}(1,0.53786409)%
    \lineheight{1}%
    \setlength\tabcolsep{0pt}%
    \put(0,0){\includegraphics[width=\unitlength,page=1]{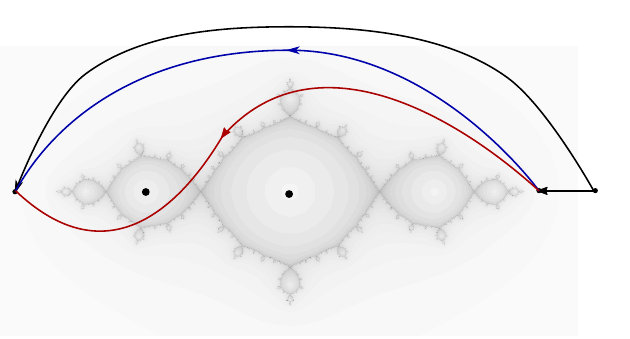}}%
    \put(0.6981942,0.38776062){\color[rgb]{0,0,0}\makebox(0,0)[lt]{\lineheight{1.25}\smash{\begin{tabular}[t]{l}$g$\end{tabular}}}}%
    \put(0.55056743,0.12866243){\color[rgb]{0,0,0}\makebox(0,0)[lt]{\lineheight{1.25}\smash{\begin{tabular}[t]{l}$h$\end{tabular}}}}%
    \put(0.88268162,0.18980457){\color[rgb]{0,0,0}\makebox(0,0)[lt]{\lineheight{1.25}\smash{\begin{tabular}[t]{l}$\delta_0$\end{tabular}}}}%
    \put(0.45928909,0.50815768){\color[rgb]{0,0,0}\makebox(0,0)[lt]{\lineheight{1.25}\smash{\begin{tabular}[t]{l}$\delta_1$\end{tabular}}}}%
    \put(0,0){\includegraphics[width=\unitlength,page=2]{BasilicaBiset1.pdf}}%
    \put(0.15665494,0.01016422){\color[rgb]{0,0,0}\makebox(0,0)[lt]{\lineheight{1.25}\smash{\begin{tabular}[t]{l}$g$\end{tabular}}}}%
    \put(0.27340118,0.32430471){\color[rgb]{0,0,0}\makebox(0,0)[lt]{\lineheight{1.25}\smash{\begin{tabular}[t]{l}$h$\end{tabular}}}}%
  \end{picture}%
\endgroup%

		\end{subfigure}
		\caption{The left figure indicates the choice of $\delta_0,\delta_1,g,h$. The right figure indicates the lifts of $g$ and $h$.}
		\label{fig:BasilicaBiset}
	\end{figure}
	
	\begin{figure}[ht]
		\begin{tikzpicture} [node distance = 2cm, on grid, auto]
			
			\node (e) [state] {$e$};
			\node (g) [state, above left = of e] {$g$};
			\node (g-) [state, below left = of e] {$g^{-1}$};
			\node (h) [state, left = of g] {$h$};
			\node (h-) [state, left = of g-] {$h^{-1}$};
			
			\path [-stealth, thick]
			(e) edge [loop above] node [right] {(0,0)}()
			(e) edge [loop below] node [right] {(1,1)}()
			(g) edge node [below left=-0.2cm] {(0,1)}   (e)
			(g) edge [loop above] node [right] {(1,0)}()
			(g-) edge node [above left=-0.2cm] {(1,0)}  (e)
			(g-) edge [loop below] node [right] {(0,1)}()
			(h) edge node {(1,0)}   (g)
			(h) edge [bend left] node {(0,1)}   (h-)
			(h-) edge node {(0,1)}  (g-)
			(h-) edge [bend left] node {(1,0)}  (h);
		\end{tikzpicture}
		\caption{The Moore diagram of the Basilica biset with the choice of basis $\{\delta_0,\delta_1\}$ in Figure \ref{fig:BasilicaBiset}. For simplicity, we omit $\delta$ in the labels of edges in the diagram, i.e., $(i,j)$ means $(\delta_i,\delta_j)$.}
		\label{fig:BasilicaMooreDiagram}
	\end{figure}
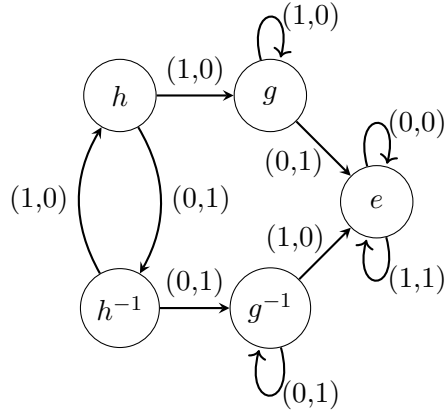
	
\end{example}

The equivalence relation on $X^{-\omega}$ corresponds to the backward infinite directed paths in the Moore diagram. More precisely, suppose that $\dots,g_2,g_1,g_0$ is the sequence of vertices in a backward infinite directed path in the Moore diagram such that the edge from $g_{n}$ to $g_{n-1}$ is labeled by $(\delta_{i_n},\delta_{j_n})$ for any $n\ge 1$. Then Equation \eqref{eqn:EqRel} holds for $h_n=g_0$ for any $n\ge1$. The converse requires further work \cite[Proposition 3.1.6, Theorem 3.5.3]{Nek05}. Thus the more backward infinite directed paths are, the more complicated the quotient $X^{-\omega}/\sim$ is.

The biset $B(f)$ defines a function $A:G \times X \to X \times G$ defined by $A(g,\delta_i)=(\delta_j,h)$ for $g \cdot \delta_i=\delta_j \cdot h$. This function $A$ is called the {\it automaton} of $(B(f),X)$. We say that the automaton (about the basis $X$) has {\it polynomial activity growth} if the sub-directed graph of the Moore diagram generated by all the elements in $\Ncal$ but $e$ has no intersecting cycles.

\begin{defn}[Adding machine basis]
	Let $f$ be a \pcf polynomial of degree $d>1$. Choose a base point $x$ of the orbifold fundamental group $\pi_1((\hCbb,P_f,\ord),x)$ on the external ray $\mathcal{R}_f(0)$ of angle zero. A basis $X=\{\delta_0,\delta_1,\dots,\delta_{d-1}\}$ for the biset $B(f)$ is the {\it adding machine basis} if $\delta_k$ is a curve that connects $\mathcal{R}_f(0)$ to $\mathcal{R}_f(e^{k\cdot 2\pi i/d})$ without any rotation around the filled Julia set of $f$.
\end{defn}

The following lemma is the reason why we call $X$ the adding machine basis. The proof is immediate from definitions.

\begin{lem}
	Let $f$ be a \pcf polynomial of degree $d>1$ and $X$ be the adding machine basis of the biset $B(f)$. For $g\in \pi_1((\hCbb,P_f,\ord),x)$ that is the counter-clockwise loop in the external Fatou component, we have $g\cdot \delta_i=\delta_{i+1} \cdot e$ for $0\le i\le d-2$ and $g \cdot \delta_{d-1}=\delta_0 \cdot g$.
\end{lem}

\begin{proof}[Proof of Theorem \ref{thm:mat}]
	Let $\mathcal{N}$ be the nucleus of $B(f)$ about the adding machine basis X. By the above lemma, $g^\pm\in \mathcal{N}$. Consider the sub-diagram $A$ of the Moore diagram generated by $g,g^{-1},e$. It is easy to show that the intermediate quotient of $X^{-\omega}$ by the backward infinite paths supported in $A$ is homeomorphic to the circle. Moreover, this circle can be identified with the circle at infinity $S^1_\infty$, which parametrizes the external angles. The Julia set $\Julia_f$ is the further quotient of $S^1_\infty$ by the backward infinite paths in the Moore diagram that are not supported in $A$. This further quotient corresponds to the quotient of the circle $S^1_\infty$ by the invariant lamination $\lambda(f)$. That is, there is a bijective correspondence between the backward infinite paths in the Moore diagram that are not supported in $A$ and the leaves of $\lambda(f)$. Hence, the automaton about the adding machine basis has polynomial activity growth if and only if $\lambda(f)$ has countably many leaves if and only if $h(f)=0$. 
\end{proof}

\begin{example}[Example \ref{eg:BasilicaBiset} continued]
	Let $D$ be the sub-diagram generated by $\{g,g^{-1},e\}$ of the Moore diagram in Figure \ref{fig:BasilicaMooreDiagram}. Backward infinite paths supported in $D$ yield the following equivalences: $\cdots0000 \sim \cdots1111$, and for any possibly empty finite word $w$ 
	\[
	\cdots1110w \sim \cdots0001w.
	\]
	These are exactly the identifications of the two endpoints of the middle intervals missed from the ternary Cantor set. Hence, 
	we obtain the circle as the intermediate quotient. Any backward infinite path that is not supported in $D$ is eventually supported in the loop between $h$ and $h^{-1}$. Its equivalence is written as
	\[
	\cdots0101\,\underline{1}\,\underline{11\cdots11}\,\underline{0}\,\underline{w} \sim \cdots1010\,\underline{0}\,\underline{00\cdots00}\,\underline{1}\,\underline{w}
	\]
	for any finite word $w$. The word $w$ is from loops at $e$, the finite repeats of 0 or 1 in the middle are from the loop at $g$ or $g^{-1}$, and the last repeats of $10$ or $01$ is from the loop between $h$ and $h^{-1}$. The underlined words may be empty. Thus the simplest example is $\cdots010101\sim \cdots 101010$, which means $2/3 \sim 1/3$ where we identify these left infinite words in $0$ and $1$ as dyadic expansions $0.101010\cdots \sim 0.010101\cdots$. See Table \ref{tab:EqClassLami} and Figure \ref{fig:BasilicaLam} for a few more examples.
	
	\begin{table}[h!]
		\begin{center}
			\begin{tabular}{c|c} 
				\textbf{Eq. classes in the Moore diagram} & \textbf{Leaf of lamination}\\
				\hline
				$\cdots010101\sim \cdots 101010$ & $2/3 \sim 1/3$\\
				$\cdots0101011 \sim \cdots 1010100$ & $5/6 \sim 1/6$\\
				$\cdots01010111 \sim \cdots 10101000$ & $11/12 \sim 1/12$\\
				$\cdots01010110 \sim \cdots 10101001$ & $5/12 \sim 7/12$
			\end{tabular}
			\caption{Equivalence classes and leaves of laminations}
			\label{tab:EqClassLami}
		\end{center}
	\end{table}
	\begin{figure}[h!]
		\centering
		\includegraphics[width=0.4\textwidth]{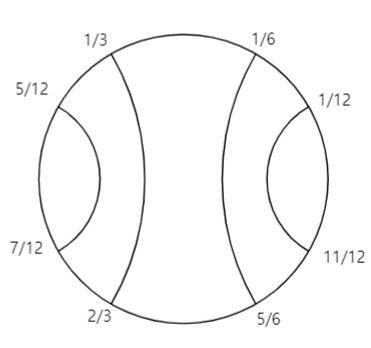}
		\caption{Leaves of the Basilica lamination described in Table \ref{tab:EqClassLami}}
		\label{fig:BasilicaLam}
	\end{figure}
	
\end{example}

\subsection*{Moore diagrams and Hubbard trees}
Let $f$ be a \pcf polynomial and $\mathcal{T}_f$ be the Hubbard tree. 

Take a base point $x$ of the fundamental group $\pi_1(\mathbb{C}\setminus P_f,x)$ on the external ray $\mathcal{R}_f(0)$ of angle zero and near the infinity. Let $E$ be an oriented edge of $\mathcal{T}_f$. Then there is a unique element $g_E\in \pi_1(\mathbb{C}\setminus P_f,x)$ so that (1) $g_E$ intersects $\mathcal{T}_f$ exactly at one point in the interior of $E$, (2) $g_E$ does not cross $\mathcal{R}_f(0)$, and (3) the ordered pair $(E,g_E)$ near the intersection point is compatible with the orientation of the plane. We call $g_E$ the {\it dual} of $E$. See Figure \ref{fig:BasilicaBiset}.

Let $E$ be an oriented edge of $\mathcal{T}_f$ and $\{\delta_i\}_{i=0,1,\dots,d-1}$ be the adding machine basis. For any $i$, there exist unique $j\in\{0,1,\dots,d-1\}$ and $h\in \pi_1(\mathbb{C}\setminus P_f)$ so that 
\[
g_E\otimes \delta_i= \delta_j \otimes h.
\]
The curve $g_E \otimes \delta_i$, which is $\delta_i$ followed by the lifting of $g_E$ starting at the terminal point of $\delta_i$, passes through some lift $E'$ of $E$ through $f$. Note that $f^{-1}(\mathcal{T}_f)$ is a tree containing $\mathcal{T}_f$. We have two case as follows.
\begin{itemize}
	\item If $E'$ is not in $\mathcal{T}_f$, then $h$ is the trivial element $e$.
	\item If $E'$ is contained in an oriented edge $E''$ of $\mathcal{T}_f$ so that the orientations of $E'$ and $E''$ are compatible (resp.\@ not compatible), then $h$ is $g_{E''}$ (resp.\@ $g_{E''}^{-1}$).
\end{itemize}
Therefore, a directed edge in the Moore diagram from $g_E$ to $g_{E''}$ corresponds to the directed edge from $E''$ to $E$ in the directed graph of the Markov map $f:\mathcal{T}_f\righttoleftarrow$.

Let $g_\infty$ is the element in $\pi_1(\mathbb{C}\setminus P_f)$ that is peripheral to the infinity. Then, the subdiagram of the Moore diagram generated by the complement of $\{g^\pm_\infty,e\}$ is equivalent with reversed directions to the directed graph associated to the oriented edges of the Hubbard tree $\mathcal{T}_f$. For example, in Figure \ref{fig:BasilicaMooreDiagram}, the complement of $\{g^{\pm},e\}$ is the diagram $h \leftrightarrow h^{-1}$. The Hubbard tree of the Basilica polynomial has one edge and the induced dynamics reverses the orientation of the edge. This relation between two diagrams also provides a more explicit proof of Theorem \ref{thm:mat}.

To obtain the entire Moore diagram, we can consider an extended Hubbard tree. Let $\beta$ be the fixed point which is the landing point of $\mathcal{R}_f(0)$. Let $e_\infty$ be the union of $\mathcal{R}_f(0)$ and the regulated path from $\beta$ to $\mathcal{T}_f$ in the filled Julia set. We define the {\it augmented Hubbard tree $\hat{\mathcal{T}}_f$} by $\mathcal{T}_f \cup \{e_\infty\}$. Then $g_\infty$ is the dual of the augment edge $e_\infty$, and we obtain the entire Moore diagram from the dynamcis on the oriented edges of the augmented Hubbard tree.

\end{document}